\documentclass{acmtrans2m}
\usepackage[T1]{fontenc}
\usepackage{epsfig}
\usepackage[english]{babel}
\usepackage{amssymb,amsmath}
\usepackage{color}
\usepackage{subfigure,pstricks,pst-node}
\usepackage{xspace}
\graphicspath{{Figures/}}
\newtheorem{theorem}{Theorem}[section]
\newtheorem{lemma}[theorem]{Lemma}
\newtheorem{proposition}[theorem]{Proposition}
\newtheorem{corollary}[theorem]{Corollary}
\newdef{definition}{Definition}
\newcommand{\cC}{\mathcal C}
\def\Pn{\ensuremath{\mathcal{P}'_n}\xspace}    
\def\Pij{\ensuremath{\mathcal{P}'_{ij}}\xspace}
\def\P{\ensuremath{\mathcal{P}'}\xspace}
\def\Bn{\ensuremath{\mathcal{B}'_n}\xspace}    
\def\Bij{\ensuremath{\mathcal{B}'_{ij}}\xspace}
\def\B{\ensuremath{\mathcal{B}'}\xspace}
\def\nBn{\ensuremath{\mathcal{B}_n}\xspace}   
\def\nBij{\ensuremath{\mathcal{B}_{ij}}\xspace}
\def\mBij{\ensuremath{\mathcal{\bar B}_{ij}}\xspace} 
\def\cBijb{\ensuremath{\mathcal{B}_{ij}^{\bullet}}\xspace}
\def\cBijw{\ensuremath{\mathcal{B}_{ij}^{\circ}}\xspace}

\def\Dn{\ensuremath{\mathcal{D}'_n}\xspace}    
\def\Dij{\ensuremath{\mathcal{D}'_{ij}}\xspace}
\def\D{\ensuremath{\mathcal{D}'}\xspace}
\def\nDn{\ensuremath{\mathcal{D}_n}\xspace}   

\def\Uij{\ensuremath{\mathcal{U}'_{ij}}\xspace}
\def\U{\ensuremath{\mathcal{U}'}\xspace}
\def\cP{\mathcal{P}}
\def\cR{\mathcal{R}}
\def\cS{\mathcal{S}}
\def\cC{\mathcal{C}}
\def\cT{\mathcal{T}}
\def\cSsy{\mathcal{S}^{\mathrm{sym}}}
\def\cSas{\mathcal{S}^{\mathrm{asy}}}
\def\cSsym{\mathcal{S}^{'\mathrm{sym}}}
\def\cSasy{\mathcal{S}^{'\mathrm{asy}}}
\def\xb{x_{\bullet}}
\def\xw{x_{\circ}}
\def\eleft{e_{\mathrm{l}}}
\def\eright{e_{\mathrm{r}}}

\def\Qn{\ensuremath{\mathcal{Q}'_n}\xspace}    
\def\Qij{\ensuremath{\mathcal{Q}'_{ij}}\xspace}
\def\Q{\ensuremath{\mathcal{Q}'}\xspace}
\newpsobject{mygrid}{psgrid}{gridcolor=lightgray,subgriddiv=1,gridlabels=0pt}
\psset{unit=1em,nodesep=3pt}
\def\noeud(#1,#2)#3#4{\rput[B](#1,#2){\rnode{#3}{#4}}}
\def\grosnoeud(#1,#2)#3#4{\rput[B](#1,#2){\rnode{#3}{\parbox{12\psunit}{\centering #4}}}}
\newpsobject{fleche}{ncline}{arrows=->}
\makeatletter 
\def\subsubsection{\@ucheadfalse
  \@startsection{subsubsection}{3}{\z@}{6pt plus 
    1pt}{-8pt}{\reset@font\normalsize\sffamily}} 
\makeatother
\def\paragraph{\subsubsection*}
\firstfoot{}
\runningfoot{}
\title{Dissections, orientations, and trees,\\ with applications to
  optimal mesh encoding\\ and to random sampling}
\author{\'ERIC FUSY, DOMINIQUE POULALHON and GILLES SCHAEFFER\\
\'E.F and G.S: LIX, \'Ecole Polytechnique. D.P: Liafa, Univ. Paris 7. France 
}
\begin{abstract}
  We present a bijection between some quadrangular dissections of an
  hexagon and unrooted binary trees, with
  interesting consequences for enumeration, mesh compression and
  graph sampling.

  Our bijection yields an efficient uniform random sampler for
  3-connected planar graphs, which turns out to be determinant for the
  quadratic complexity of the current best known uniform random
  sampler for labelled planar graphs [{\bf Fusy, Analysis of Algorithms
      2005}].

  It also provides an encoding for the set $\mathcal{P}(n)$
  of $n$-edge 3-connected planar graphs that matches the entropy bound
  $\frac1n\log_2|\mathcal{P}(n)|=2+o(1)$ bits per edge (bpe). This
  solves a theoretical problem recently raised in mesh compression, as
  these graphs abstract the combinatorial part of meshes with
  spherical topology. We also achieve the {optimal parametric
  rate} $\frac1n\log_2|\mathcal{P}(n,i,j)|$ bpe for graphs of
  $\mathcal{P}(n)$ with $i$ vertices and $j$ faces, matching in
  particular the optimal rate for triangulations.

  Our encoding relies on a linear time algorithm to
  compute an orientation associated to the minimal Schnyder wood of a
  3-connected planar map. This algorithm is of independent interest,
  and it is for instance a key ingredient in a recent straight line
  drawing algorithm for 3-connected planar graphs [\bf Bonichon et
  al., Graph Drawing 2005].

\end{abstract}
\begin{document}
\maketitle

\section{Introduction}

One origin of this work can be traced back to an article of Ed Bender
in the \emph{American Mathematical Monthly} \cite{B87}, where he asked for a
simple explanation of the remarkable asymptotic formula
\begin{equation}\label{eq:bender}
|\mathcal{P}(n,i,j)| ~\sim~ \frac1{3^52^4ijn}{2i-2\choose
j+2}{2j-2\choose i+2}
\end{equation}
%
for the cardinality of the set of 3-connected (unlabelled) planar
graphs with $i$ vertices, $j$ faces and $n=i+j-2$ edges, $n$ going to
infinity.  By a theorem of \citeN{Whitney33},
these graphs have essentially a
unique embedding on the sphere up to homeomorphisms, so that their
study amounts to that of \emph{rooted 3-connected maps}, where a map is a
graph embedded in the plane and \emph{rooted} means with a marked oriented
edge.

\subsection{Graphs, dissections and trees}
Another known property of 3-connected planar graphs with $n$
edges is the fact that they are in direct one-to-one correspondence
with dissections of the sphere into $n$ quadrangles that have no
non-facial 4-cycle. The heart of our paper lies in a further
one-to-one correspondence.

\begin{theorem}\label{thm:intro}
  There is a one-to-one correspondence between unrooted binary trees
  with $n$ nodes and unrooted quadrangular dissections of an hexagon
  with $n$ interior vertices and no non-facial 4-cycle.
\end{theorem}

The mapping from binary trees to dissections, which we call the
\emph{closure}, is easily described and resembles constructions that
were recently proposed for simpler kinds of maps
\cite{Sc97,BdFG02,PS03b}.  The proof that the mapping is a bijection is instead
rather sophisticated, relying on new properties of constrained
orientations \cite{Oss}, related to Schnyder woods of triangulations
and 3-connected planar maps \cite{S90,DiTa,F01} .

Conversely, the reconstruction of the tree from the dissection relies
on a linear time algorithm to compute the minimal Schnyder woods of a
3-connected map (or equivalently, the minimal $\alpha_0$-orientation of the
associated derived map, see Section~\ref{section:compute}). This
problem is of independant interest and our algorithm has for example
applications in the graph drawing context~\cite{BFM04}. It
is akin to Kant's canonical ordering \cite{Ka96,CGHKL98,BGH03,CD04}, but
again the proof of correctness is quite involved.

Theorem~\ref{thm:intro} leads directly to the implicit representation
of the numbers $|\cP_n'|$ ---counting rooted 3-connected maps with $n$
edges--- due to~\citeN{Tu63}), and its refinement as discussed in
Section~\ref{section:counting} yields that of $|\cP_{ij}'|$ the number
of rooted 3-connected maps with $i$ vertices and $j$ faces (due
to~\citeN{Mu}) from which Formula~(\ref{eq:bender}) follows.  It
partially explains the combinatorics of the occurrence of the cross
product of binomials, since these are typical of binary tree
enumerations.  Let us mention that the one-to-one correspondence
specializes particularly nicely to count plane triangulations (i.e.,
3-connected maps with all faces of degree~ 3), leading to the first
bijective derivation of the counting formula for \emph{unrooted plane
  triangulations} with $i$ vertices, originally found by~\citeN{Br}
using algebraic methods.
\subsection{Random sampling}
A second byproduct of Theorem~\ref{thm:intro} is an efficient uniform
random sampler for rooted 3-connected maps, i.e.,  an
algorithm that, given $n$, outputs a
random element in the set $\mathcal{P}_n'$ of rooted 3-connected maps
with $n$ edges with equal chances for all
elements. The same principles yield a uniform sampler for $\cP_{ij}'$.

The uniform random generation of classes of maps like triangulations
or 3-connected graphs was first considered in mathematical physics
(see references in \cite{ABBJP94,PS03b}), and various types of random
planar graphs are commonly used for testing graph drawing algorithms
(see \citeA{taxiplan}).

The best previously known algorithm \cite{Sc99} had expected
complexity $O(n^{5/3})$ for $\cP_n'$, and was much less
efficient for $\cP_{ij}'$, having even exponential complexity
for $i/j$ or $j/i$ tending to~2 (due to Euler's formula these ratio
are bounded above by 2 for 3-connected maps).  In
Section~\ref{section:sampling}, we show that our generator for
$\mathcal{P}_n'$ or $\cP_{ij}'$ performs in linear time
except if $i/j$ or $j/i$ tends to~2 where it becomes at most cubic.

From the theoretical point of view, it is also desirable to work with
the uniform distribution on planar graphs. However, random (labelled)
planar graphs appear to be challenging mathematical objects
\cite{OPT03,MdStWe03}.  A Markov chain converging to the uniform
distribution on planar graphs with $i$ vertices was given by
\citeN{DVW96}, but it resists known approaches for perfect sampling
\cite{W98}, and has unknown mixing time.  As opposed to this, a
recursive scheme to sample planar graphs was proposed by \citeN{BGK03},
with amortized complexity $O(n^{6.5})$. This result is based on a
recursive decomposition of planar graphs: a planar graph can be
decomposed into a tree-structure whose nodes are occupied by rooted
3-connected maps. Generating a planar graph reduces to computing
branching probabilities so as to generate the decomposition tree with
suitable probability; then a random rooted 3-connected map is
generated for each node of the decomposition tree. \citeN{BGK03} use
the so-called recursive method~\cite{NiWi78,FZVC94,Wi97} to take
advantage of the recursive decomposition of planar graphs.  Our new
random generator for rooted 3-connected maps reduces their amortized
cost to~$O(n^3)$. Finally a new uniform random generator for planar
graphs was recently developped by one of the authors~\cite{Fusy}, that
avoids the expensive preprocessing computations of~\cite{BGK03}. The
recursive scheme is similar to the one used in~\cite{BGK03}, but the
method to translate it to a random generator relies on Boltzmann
samplers, a new general framework for the random generation recently
developed in~\cite{DuFlLoSc}.  Thanks to our random generator for
rooted 3-connected maps, the algorithm of \cite{Fusy} has a
time-complexity of $\mathcal{O}(n^2)$ for exact size uniform sampling
and even performs in linear time for approximate size uniform
sampling.
 
\subsection{Succinct encoding} 
A third byproduct of Theorem~\ref{thm:intro} is the possibility to
encode in linear time a 3-connected planar graph with $n$ edges by a
binary tree with $n$ nodes. In turn the tree can be encoded by a
balanced parenthesis word of $2n$ bits. This code is \emph{optimal} in
the information theoretic sense: the entropy per edge
 of this class of graphs, i.e., the quantity $\frac1n\log_2|\mathcal{P}(n)|$, tends to~2
when $n$ goes to infinity, so that a code for $\mathcal{P}(n)$ cannot
give a better guarantee on the compression rate.

Applications calling for compact storage and fast transmission of 3D
geometrical meshes have recently motivated a huge literature on
compression, in particular for the combinatorial part of the meshes.
The first compression algorithms dealt only with triangular faces
\cite{R99,TG98}, but many meshes include larger faces, so that
polygonal meshes have become prominent (see \cite{AlGo03} for a
recent survey).

The question of optimality of coders was raised in relation with
exception codes produced by several heuristics when dealing
with meshes with spherical topology \cite{Go03,KADS02}. Since these
meshes are exactly triangulations (for triangular meshes) and
3-connected planar graphs (for polyhedral ones), the coders in
\cite{PS03b} and in the present paper respectively prove that
traversal based algorithms can achieve optimality. 

On the other hand, in the context of succinct data structures,
almost optimal algorithms have been proposed \cite{HKL00,L02}, that
are based on separator theorems.  However these algorithms are not
truly optimal (they get $\varepsilon$ close to the entropy but at the
cost of an uncontrolled increase of the constants in the linear
complexity). Moreover, although they rely on a sophisticated recursive
structure, they do not support efficient adjacency requests.

As opposed to that, our algorithm shares with \cite{HKL99a,BGH03} the
property that it produces essentially the code of a spanning
tree. More precisely it is just the balanced parenthesis code of a
binary tree, and adjacencies of the initial dissection that are not
present in the tree can be recovered from the code by a simple
variation on the interpretation of the symbols.  Adjacency queries can
thus be dealt with in time proportional to the degree of
vertices~\cite{Ca} using the approach of \cite{MR97,HKL99a}.

Finally we show that the code can be modified to be optimal on the
class $\mathcal{P}(n,i,j)$. Since the entropy of this class is
strictly smaller than that of $\mathcal{P}(n)$ as soon as $|i-n/2|\gg
n^{1/2}$, the resulting parametric coder is more efficient in this
range. In particular in the case $j=2i-4$ our new algorithm
specializes to an optimal coder for triangulations.

\subsection{Outline of the paper}

The paper starts with two sections of preliminaries: definitions of
the maps and trees involved (Section~\ref{sec:defs}), and some basic
correspondences between them (Section~\ref{sec:twocorres}). Then comes
our main result (Section~\ref{sec:bijection}), the mapping between
binary trees and some dissections of the hexagon by quadrangular
faces. The fact that this mapping is a bijection follows
from the existence and uniqueness of a certain tri-orientation of our
dissections.  The proof of this auxiliary theorem, which requires the
introduction of the so-called derived maps and their
$\alpha_0$-orientations, is delayed to Section~\ref{section:proof},
that is, after the three sections dedicated to applications of our
main result: in these sections we successively discuss counting
(Section~\ref{section:counting}), sampling
(Section~\ref{section:sampling}) and coding
(Section~\ref{section:codingAlgo}) rooted 3-connected maps. The third
application leads us to our second important result: in
Section~\ref{section:compute} we present a linear time algorithm to
compute the minimal $\alpha_0$-orientation of the  
derived map of a 3-connected
planar map (which also corresponds to the minimal Schnyder woods
alluded to above). Finally, Section~\ref{section:proofCorrect} is
dedicated to the correctness proof of this orientation
algorithm. Figure~\ref{fig:relations} summarizes the connections
between the different families of objects we consider.

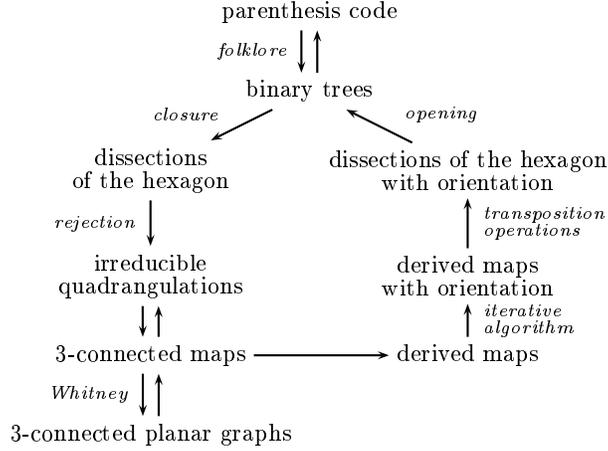
\begin{figure}
  \centering
  \begin{pspicture}(0,-3)(20,14)\def\baselinestretch{.8}\small
    \noeud(4,-3){3cg}{3-connected planar graphs}
    \noeud(4,0){3c}{3-connected maps} 
    \noeud(16,0){dm}{derived maps} 
    \grosnoeud(4,3){iq}{irreducible\\ quadrangulations}
    \grosnoeud(16,3){odm}{derived maps\\ with orientation}
    \grosnoeud(4,7){dh}{dissections\\ of the hexagon}
    \grosnoeud(16,7){odh}{dissections of the hexagon with
      orientation} 
    \noeud(10,10){bt}{binary trees}
    \noeud(10,13){pc}{parenthesis code} \em\scriptsize
    \fleche{3c}{dm}
    \fleche{dm}{odm}\Bput{\parbox{4\psunit}{iterative\\ algorithm}}
    \fleche{odm}{odh}\Bput{\parbox{4\psunit}{transposition\\
        operations}} 
    \fleche{odh}{bt}\bput(0){opening}
    \fleche{bt}{dh}\bput(1){closure} 
    \fleche{dh}{iq}\Bput{rejection}
    \fleche[offset=-3pt]{iq}{3c}
    \fleche[offset=-3pt]{3c}{iq}
    \fleche[offset=-3pt]{3c}{3cg}\Bput{Whitney}
    \fleche[offset=-3pt]{3cg}{3c} \fleche[offset=-3pt]{bt}{pc}
    \fleche[offset=-3pt]{pc}{bt}\Bput{folklore}
  \end{pspicture}
  \caption{Relations between involved objects.}\label{fig:relations}
\end{figure}

\section{Definitions}\label{sec:defs}

\subsection{Planar maps} 
A \emph{planar map} is a proper embedding of an unlabelled connected
graph in the plane, where \emph{proper} means that edges are smooth simple
arcs that do not meet but at their endpoints. A planar map is said to be
\emph{rooted} if one edge of the outer face, called the
\emph{root-edge}, is marked and oriented such that the outer face lays
on its right. The origin of the root-edge is called
\emph{root-vertex}. Vertices and edges are said to be \emph{outer} or \emph{inner}
depending on whether they are incident to the outer face or not.

A planar map is \emph{3-connected} if it has at least 4 edges and 
can not be disconnected
by the removal of two vertices. The first 3-connected planar map is
the tetrahedron, which has 6 edges.  
We denote by \Pn (respectively \Pij) the set
of rooted 3-connected planar maps with $n$ edges (resp. $i$ vertices
and $j$ faces). A 3-connected planar map is
\emph{outer-triangular} if its outer face is triangular.

\subsection{Plane trees, and half-edges} 
\emph{Plane trees} are planar maps with a single face ---the outer
one. A vertex is called a \emph{leaf} if it has degree~1, and
\emph{node} otherwise.  Edges incident to a leaf are called
\emph{stems}, and the other are called \emph{entire edges}.
Observe that plane trees are
\emph{unrooted} trees.  

\emph{Binary trees} are plane trees whose nodes have degree~3. By
convention we shall require that a \emph{rooted binary tree} has a
root-edge that is a stem. The root-edge of a rooted binary tree thus
connects a node, called the \emph{root-node}, to a leaf,
called the \emph{root-leaf}. With this definition of rooted binary
tree, upon drawing the tree in a top down manner starting with the
root-leaf, every node (including the root-node) has a father, a left son
and a right son. This (very minor) variation on the usual definition
of rooted binary trees will be convenient later on. For $n\geq 1$, we denote
respectively by \nBn and \Bn the sets of binary and rooted binary
trees with $n$ nodes (they have $n+2$ leaves, as proved by induction
on $n$).
These rooted trees are well known to be counted by the Catalan
numbers: $|\Bn|=\frac1{n+1}{2n\choose n}$.

The vertices of a binary tree can be greedily bicolored ---say in
black or white--- so that adjacent vertices have distinct colors.  The
bicoloration is unique up to the choice of the color of the first
node.  As a consequence, rooted bicolored binary trees are either
\emph{black-rooted} or \emph{white-rooted}, depending on the color of
the root node.  The sets of black-rooted (resp. white-rooted) binary
trees with $i$ black nodes and $j$ white nodes is denoted by $\cBijb$
(resp. by $\cBijw$); and the total set of rooted bicolored binary
trees with $i$ black nodes and $j$ white nodes is denoted by
$\mathcal{B}_{ij}'$.


It will be convenient to view each entire edge of a tree as a pair of
opposite \emph{half-edges} ---each one incident to one extremity of the
edge--- and to view each stem as a single half-edge ---incident to the
node holding the stem. More generally we shall consider maps that
have \emph{entire edges} (made of two half-edges) and \emph{stems}
 (made of only one
half-edge). It is then also natural to associate one face to each
half-edge, say, the face on its right. In the case of trees, there is
only the outer face, so that all half-edges get the same associated
face.


\subsection{Quadrangulations and dissections}
A \emph{quadrangulation} is a planar map whose faces (including the
outer one) have degree~4. A \emph{dissection of the hexagon by
quadrangular faces} is a planar map whose outer face has degree~6 and
inner faces have degree~4.


%

Cycles that do not delimit a face are said to be \emph{separating}. A
quadrangulation or a dissection of the hexagon by quadrangular faces
is said to be \emph{irreducible} if it has at least 4 faces and 
has no separating 4-cycle. The first irreducible quadrangulation is the 
cube, which has 6 faces. We denote
by $\Qn$ the set of rooted irreducible quadrangulations 
with $n$ faces,
including the outer one.  Euler's relation ensures that these
quadrangulations have $n+2$ vertices. We denote by $\nDn$ 
($\Dn$) the set
of  (rooted, respectively) 
irreducible dissections of the hexagon with $n$ inner
vertices. These have $n+2$ quadrangular faces, according to Euler's
relation.  From now on, irreducible dissections of the hexagon by
quadrangular faces will simply be called \emph{irreducible
dissections}. The classes of rooted irreducible quadrangulations and of 
rooted irreducible dissections are respectively denoted by 
$\mathcal{Q}'=\cup_n\mathcal{Q}_n'$ and $\mathcal{D}'=\cup_n\mathcal{D}_n'$. 

As faces of dissections and quadrangulations have even degree, the 
vertices of these maps can be greedily bicolored, say, in black and white,
so that each edge connects a black vertex to a white one. Such a
bicoloration is unique up to the choice of the colors. We denote by
\Qij the set of rooted bicolored irreducible
quadrangulations with $i$ black vertices and $j$ white vertices and
such that the root-vertex is black; and by \Dij the set
of rooted bicolored irreducible dissections with $i$ black
\emph{inner} vertices and $j$ white \emph{inner} vertices and such
that the root-vertex is black.

A bicolored irreducible dissection is \emph{complete} if the
three outer white vertices of the hexagon have degree
exactly~2. Hence, these three vertices are incident to two 
adjacent edges on the hexagon.

\section{Correspondences between families of planar maps}
\label{sec:twocorres}

This section recalls a folklore bijection between
irreducible quadrangulations and 3-connected maps, hereafter called
angular mapping, see~\cite{Mu}, and its adaptation to
outer-triangular 3-connected maps.

\subsection{3-connected maps and irreducible quadrangulations}

Let us first recall how the angular mapping works.
Given a rooted quadrangulation $Q \in \Qn$ endowed with its vertex
bicoloration, let $M$ be the rooted map obtained by
linking, for each face $f$ of $Q$ (even the outer face), the two
diagonally opposed black vertices of $f$; the root of $M$ is chosen to
be the edge corresponding to the outer face of $Q$, oriented so that
$M$ and $Q$ have same root-vertex, see Figure~\ref{figure:tutte}. The
map $M$ is often called the \emph{primal map} of $Q$. A similar
construction using white vertices instead of black ones would
give its \emph{dual map} (i.e., the map with a vertex in each face of $M$ 
and edge-set corresponding to the adjacencies between vertices and 
faces of $M$).

\begin{figure}
  \centering
  \subfigure[\scriptsize A
    quadrangulation]{\includegraphics[height=10em]{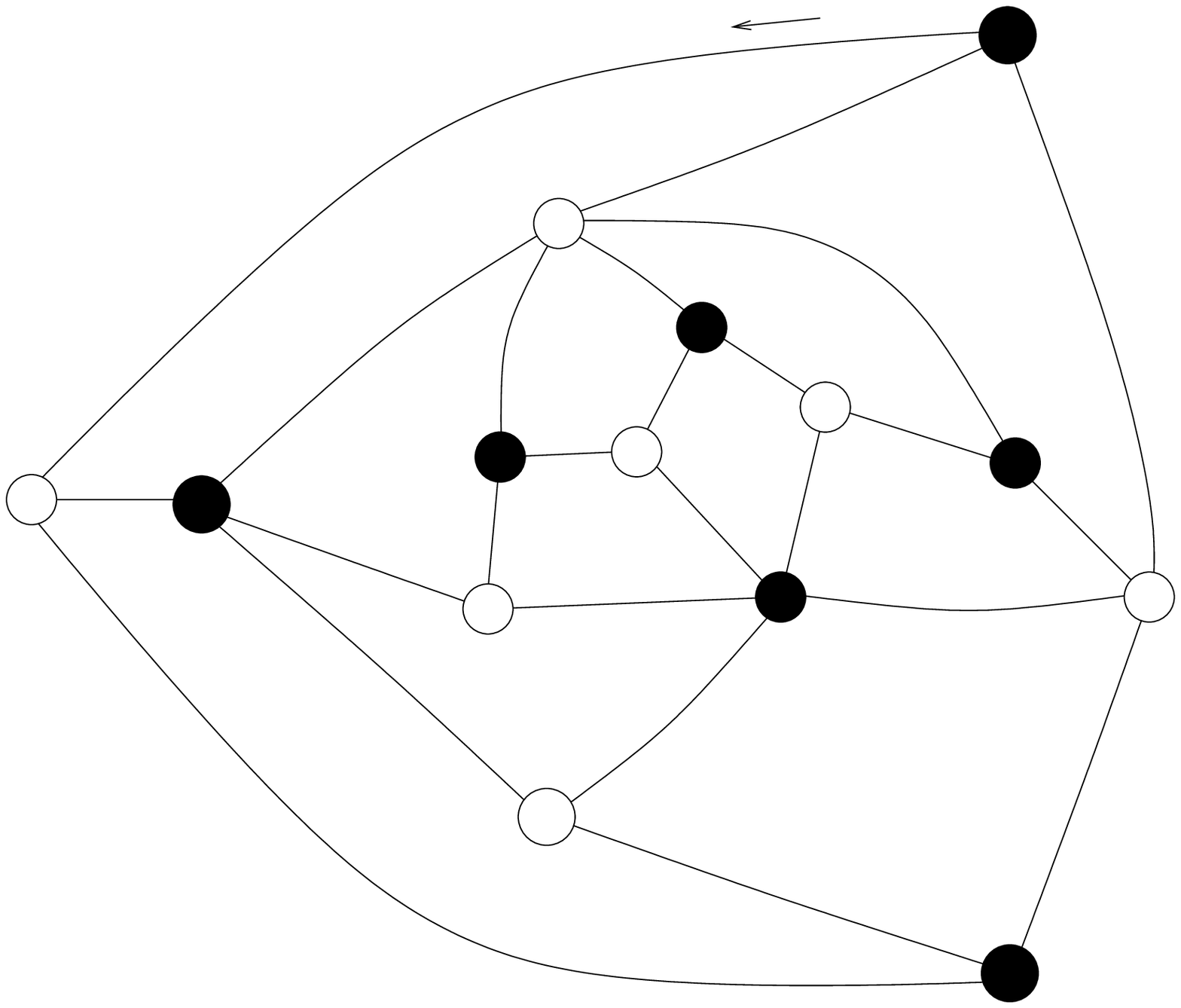}}\quad
  \subfigure[\scriptsize with its black
    diagonals]{\quad\includegraphics[height=10em]{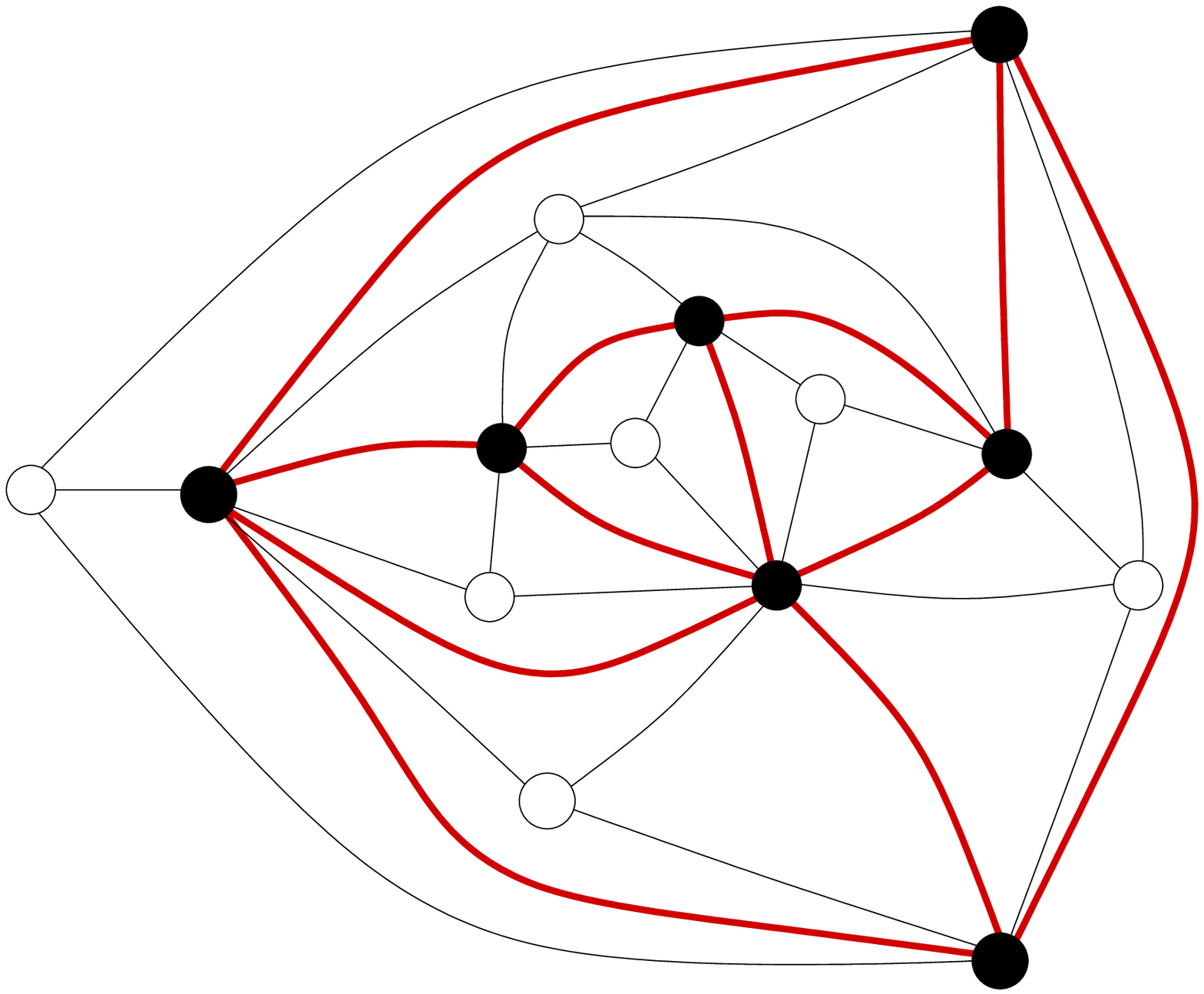}\quad}\quad
  \subfigure[\scriptsize gives a planar
    map.]{\;\includegraphics[height=10em]{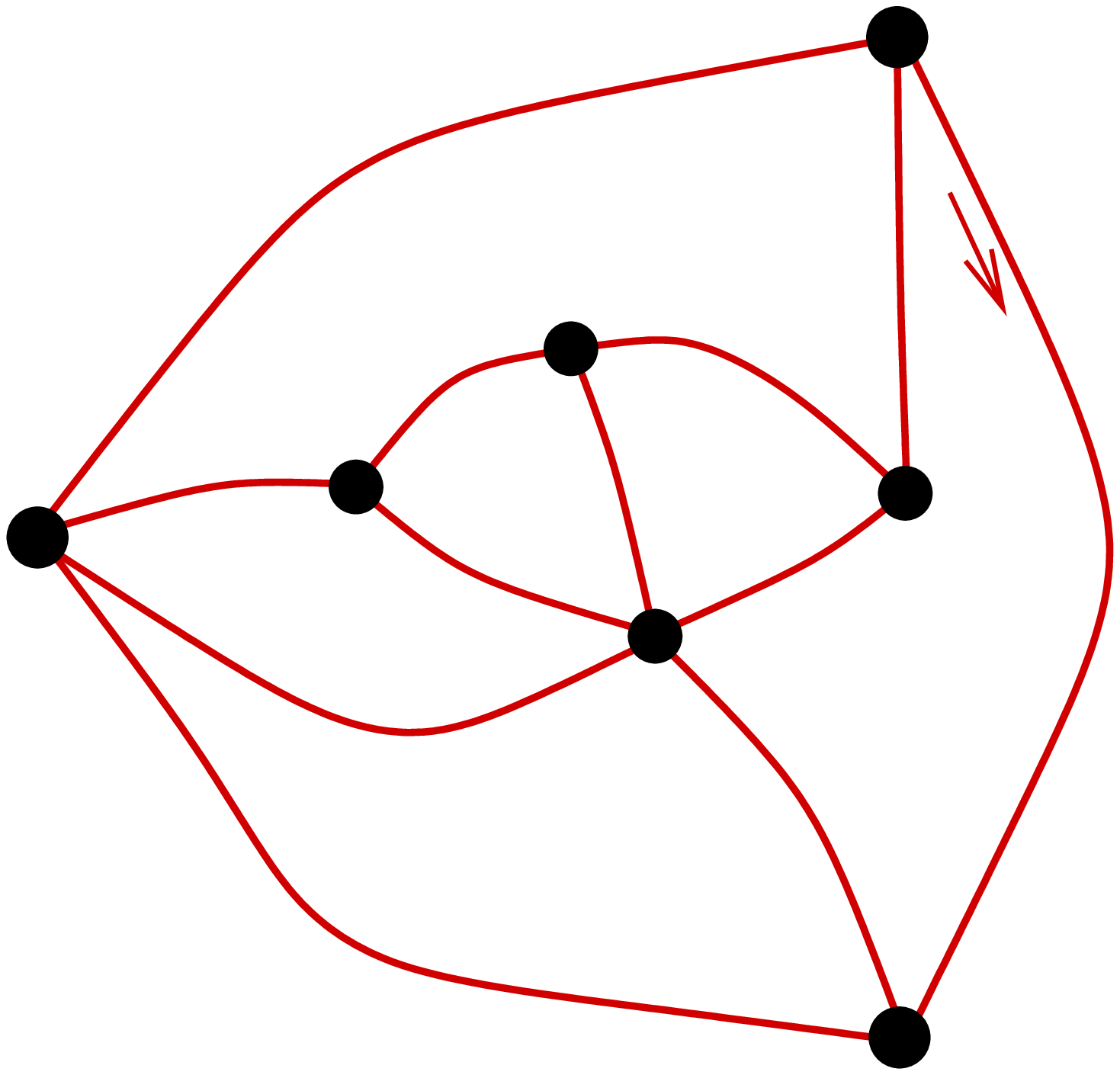}\quad}
  \caption{The angular mapping: from a rooted irreducible quadrangulation
    to a rooted 3-connected planar map.}
  \label{figure:tutte}
\end{figure}

The construction of the primal map is easily invertible. Given any
rooted map $M$, the inverse construction consists in adding a vertex
called a \emph{face-vertex} in each face (even the outer one) of $M$
and linking a vertex $v$ and a face-vertex $v_f$ by an edge if $v$ is
incident to the face $f$ corresponding to $v_f$. Keeping only these
face-vertex incidence edges yields a quadrangulation. The root is
chosen as the edge that follows the root of $M$ in counter-clockwise
order around its origin.

The following theorem is a classical result in the theory of maps.

\begin{theorem}[(Angular mapping)]
  \label{theorem:tutte}
  The angular mapping is a bijection between \Pn and \Qn and more
  precisely a bijection between \Pij and \Qij.
\end{theorem}

\subsection{Outer-triangular 3-connected maps and bicolored complete
  irreducible dissections} 
\label{section:bijectionViaHexagon}

The same principle yields a bijection, also called
angular mapping, between outer-triangular 3-connected
maps and bicolored complete irreducible dissections, which will prove
very useful in Sections~\ref{section:codingAlgo}
and~\ref{section:proof}. This mapping is very similar to the angular
mapping: given a complete dissection $D$, associate to $D$ the 
map $M$ obtained by linking the two black vertices of each \emph{inner face}
of $D$ by a new edge, see
Figure~\ref{figure:bijectionOuterComplete}. The map $M$ is called the 
\emph{primal map} of $D$.

\begin{theorem}[(Angular mapping with border)]
\label{thm:bijOuter}
The angular mapping, formulated for complete dissections, is a bijection 
between bicolored complete irreducible dissections with $i$ black
vertices and $j$
white vertices and outer-triangular 3-connected maps with $i$ vertices
and $j-3$ inner faces.
\end{theorem}
\begin{proof} 
  The proof follows similar lines as that of
  Theorem~\ref{theorem:tutte}, see~\cite{Mu}.\hfill$\qed$
\end{proof}

\begin{figure}
\subfigure[{\scriptsize A dissection,}\label{figure:bijectionOuterComplete1}]{\includegraphics[height=.21\linewidth]{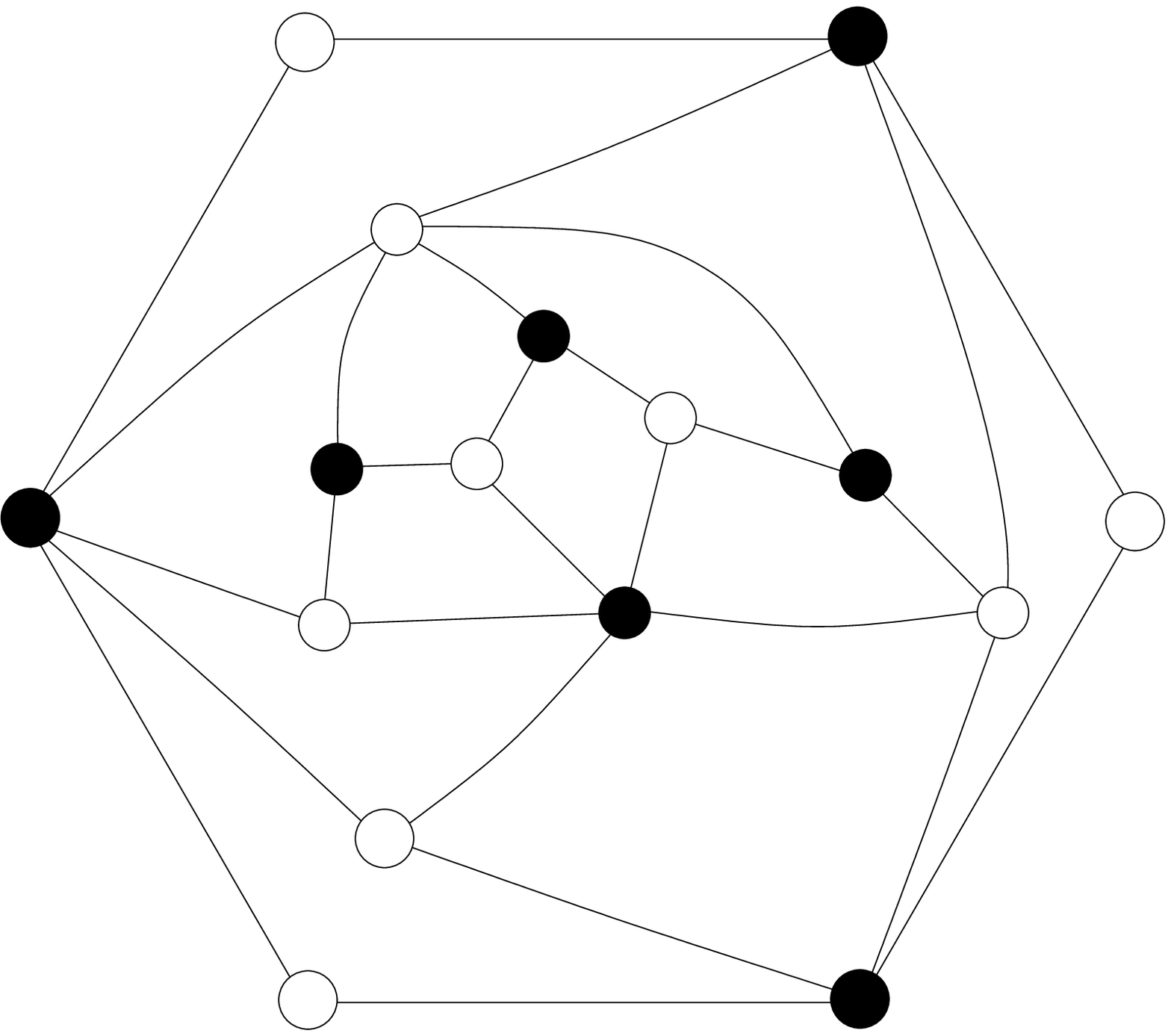}}\quad\;
\subfigure[\scriptsize black diagonals,\label{figure:bijectionOuterComplete2}]{\includegraphics[height=.21\linewidth]{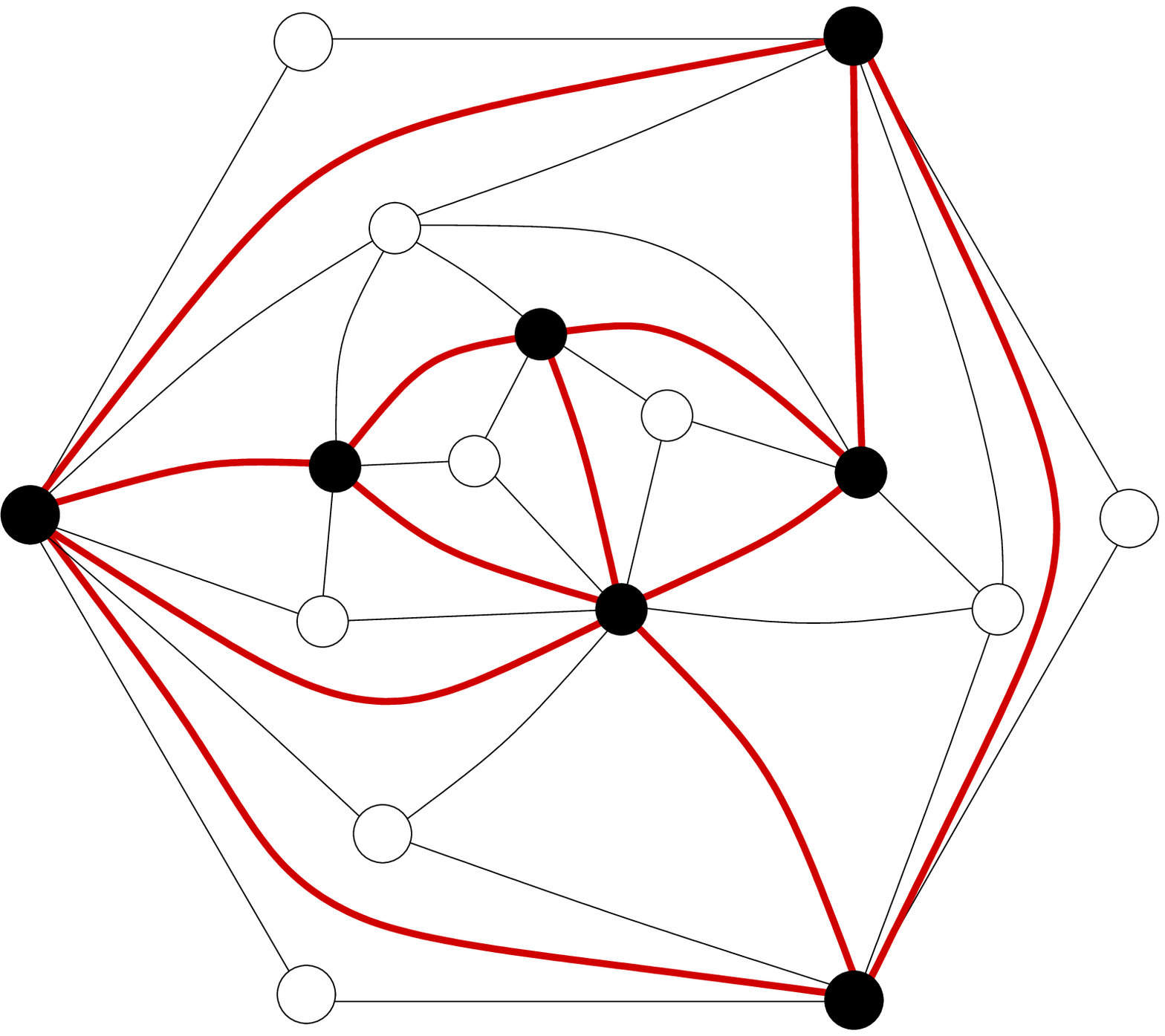}}\;
\subfigure[\scriptsize the 3-connected map,\label{figure:bijectionOuterComplete3}]{\quad\includegraphics[height=.21\linewidth]{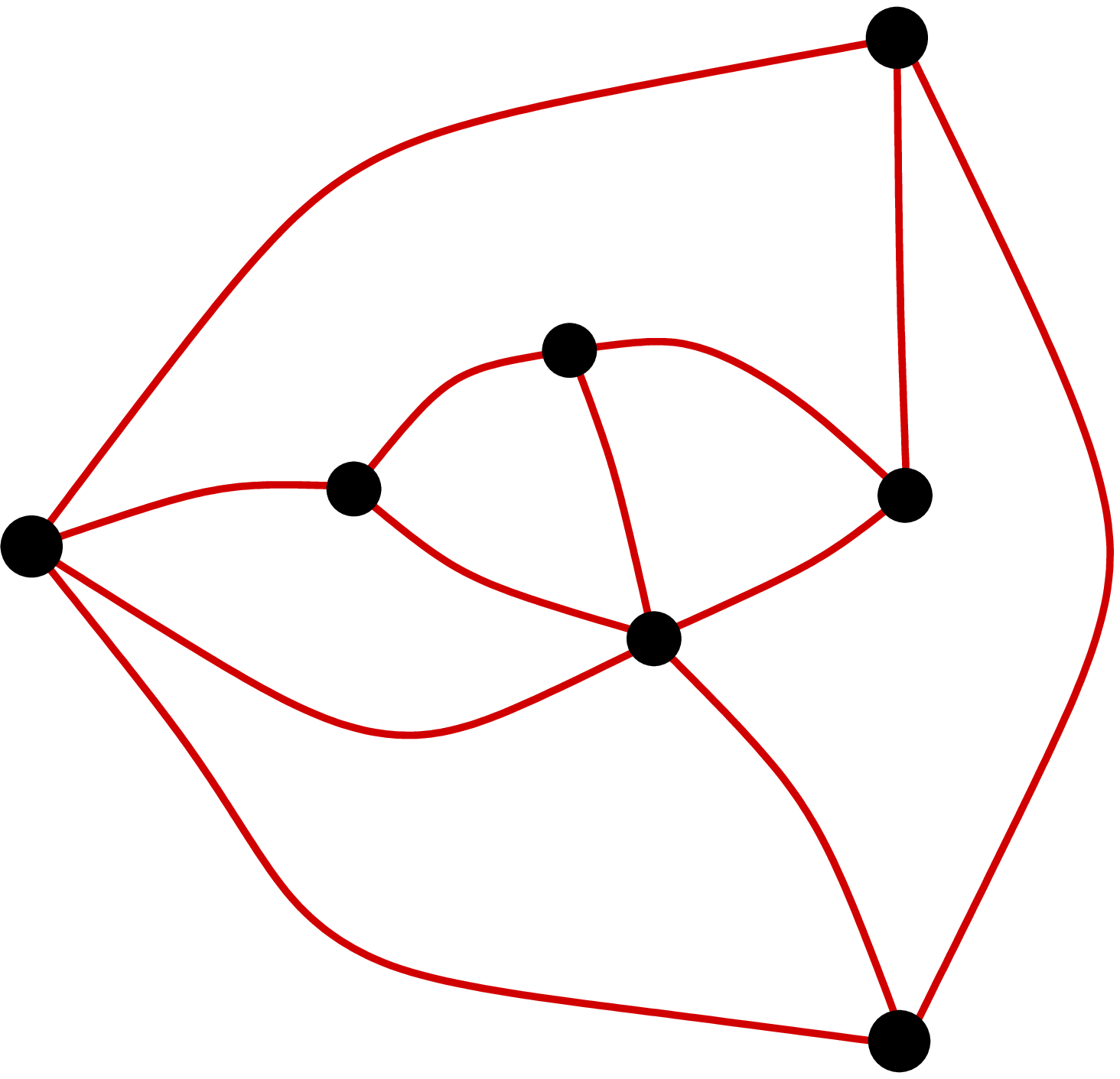}\quad}\;
\subfigure[\scriptsize the derived map.\label{figure:bijectionOuterComplete4}]{\includegraphics[height=.21\linewidth]{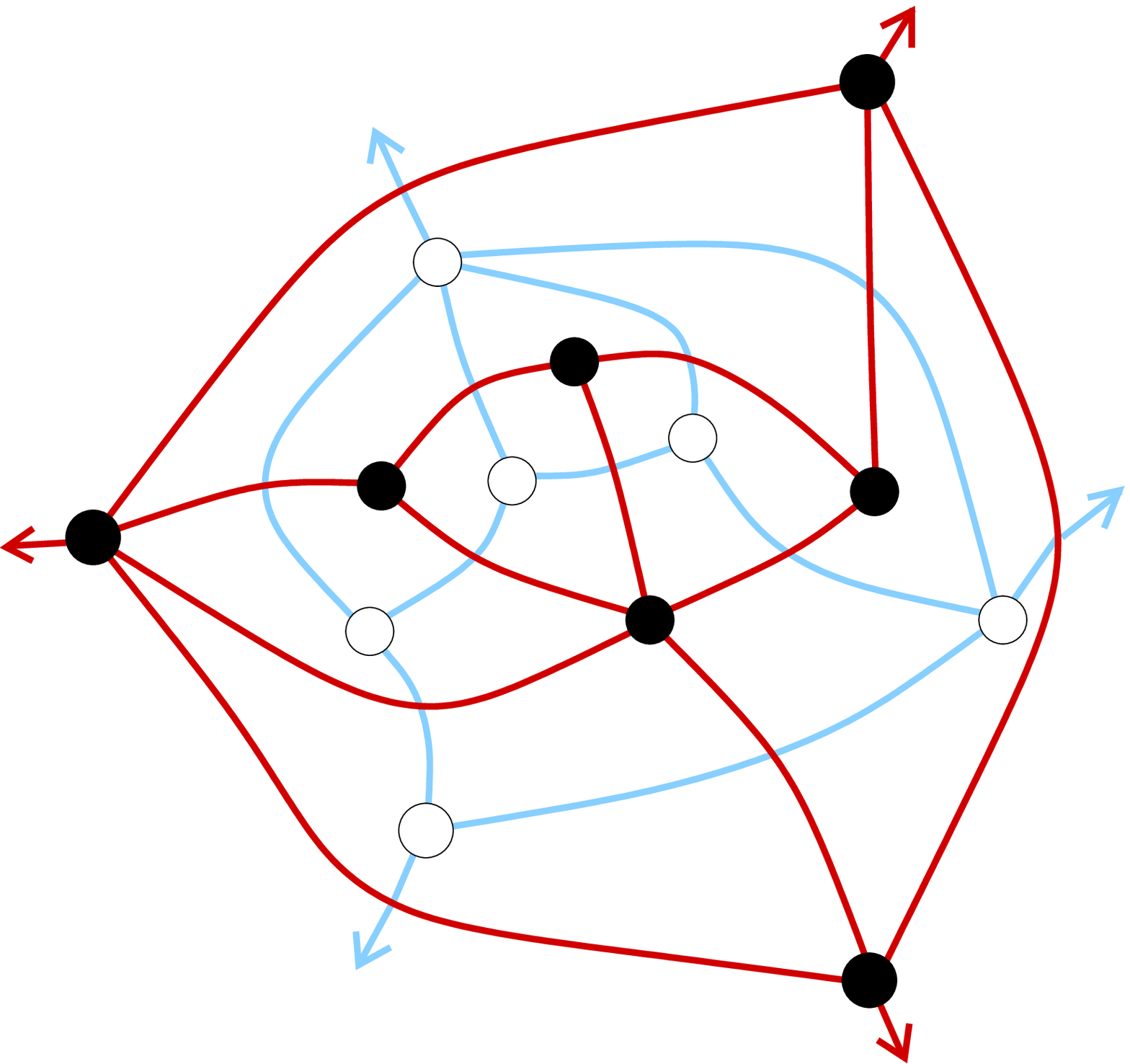}}
\caption{The angular mapping with border: from a bicolored complete
       irreducible dissection (a) to an outer-triangular 3-connected map (c). The common derived map is shown in (d).}
     \label{figure:bijectionOuterComplete}
\end{figure}

\subsection{Derived maps}
\label{section:derivatedmap}
In its version for complete dissections, the angular mapping can also be
formulated using the concept of derived map, which will be very useful 
throughout this article (in particular when dealing with orientations).

Let $M$ be an outer-triangular 3-connected map, and let $M^*$ be the
map obtained from the dual of $M$ by removing the dual vertex
corresponding to the outer face of $M$.  Then the \emph{derived map}
$M'$ of $M$ is the superimposition of $M$ and $M^*$, where each outer
vertex receives an additional half-edge directed toward the outer
face.  For example, Figure~\ref{figure:bijectionOuterComplete4} shows
the derived map of the map given in
Figure~\ref{figure:bijectionOuterComplete3}. The map $M$ is called the
\emph{primal map} of $M'$ and the map $M^*$ is called the \emph{dual
  map} of $M'$. Observe that the superimposition of $M$ and $M^*$
creates a vertex of degree 4 for each edge $e$ of $M$, due to the
intersection of $e$ with its dual edge. These vertices of $M'$ are
called \emph{edge-vertices}. An edge of $M'$ either corresponds to an
half-edge of $M$ when it connects an edge-vertex and a primal vertex,
or to an half-edge of $M^*$ when it connects an edge-vertex and a dual
vertex.

Similarly, one defines derived maps of complete irreducible
dissections.  Given a bicolored complete irreducible dissection $D$,
the derived map $M'$ of $D$ is constructed as follows; for each inner
face $f$ of $D$, link the two black vertices incident to $f$ by a
\emph{primal edge}, and the two white ones by a \emph{dual
  edge}. These two edges, which are the two diagonals of~$f$,
intersect at a new vertex called an \emph{edge-vertex}. The derived
map is then obtained by keeping the primal and dual edges and all
vertices except the three outer white ones and their incident
edges. Finally, for the sake of regularity, each of the six outer
vertices of $M'$ receives an additional half-edge directed toward the
outer face.  For example, the derived map of the dissection of
Figure~\ref{figure:bijectionOuterComplete1} is shown in
Figure~\ref{figure:bijectionOuterComplete4}.
Black vertices are called \emph{primal vertices} and white vertices
are called \emph{dual vertices} of the derived map $M'$. The submap
$M$ ($M^*$) of $M'$ consisting of the primal vertices and primal edges
(resp. the dual vertices and dual edges) is called the \emph{primal
  map} (resp. the \emph{dual map}) of the derived map. Clearly, $M$
has a triangular outer face; and, by construction, a bicolored
complete irreducible dissection and its primal map have the same
derived map.

\section{Bijection between binary trees and irreducible dissections}
\label{sec:bijection}

\subsection{Closure mapping: from trees to dissections}
\label{sec:closures}



\paragraph{Local and partial closure} 
Given a map with entire edges and stems (for instance a tree), we
define a \emph{local closure} operation, which is based on a
counter-clockwise walk around the map: this walk alongside the boundary
of the outer map visits a succession of stems and entire edges, or more
precisely, a sequence of half-edges having the outer face on their
right-hand side.
When a stem is immediately followed in this walk by three entire edges,
its local closure consists in the creation of an opposite half-edge
for this stem, which is attached to farthest endpoint of the third
entire edge: this amounts to completing the stem into an entire edge, so
as to create ---or \emph{close}--- a quadrangular face. This operation is
illustrated in Figure~\ref{figure:localClosure}.

\begin{figure}
  \def\etalon{.28\linewidth}
  \centering
  \subfigure[\scriptsize A binary
    tree,\label{figure:binaryTree}]{\includegraphics[width=\etalon]{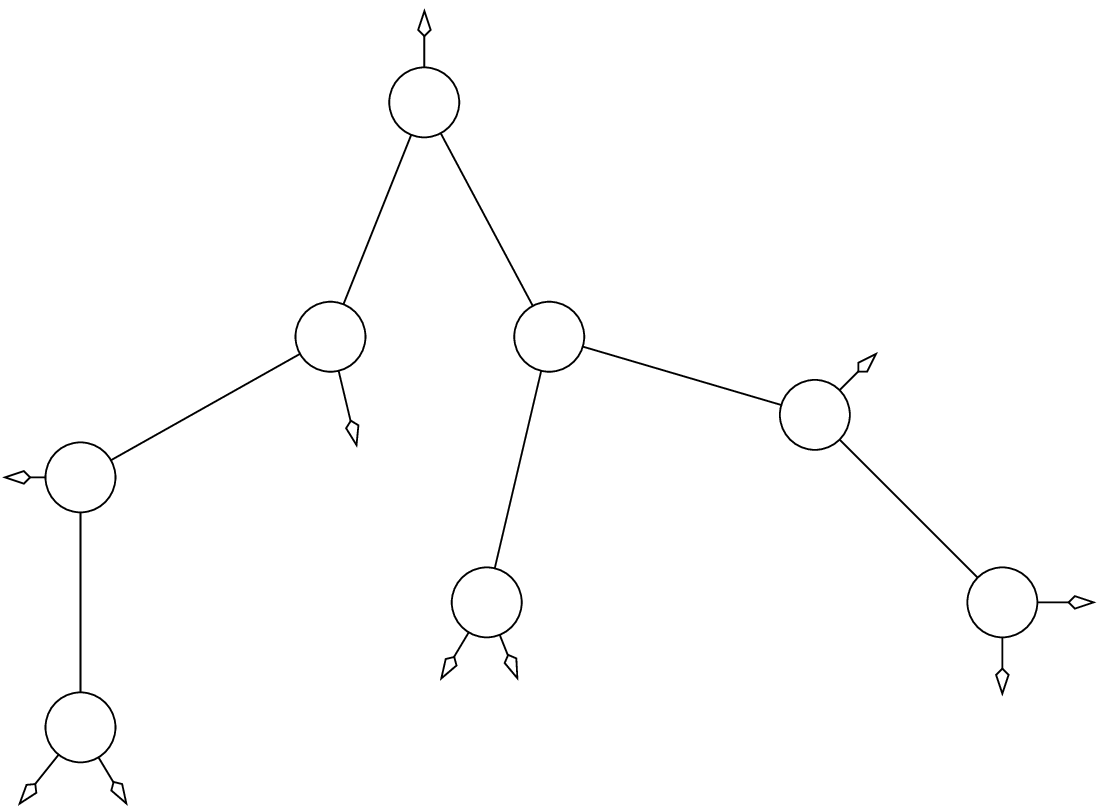}} 
  \qquad
  \subfigure[\scriptsize a local
    closure,\label{figure:localClosure}]{\includegraphics[width=\etalon]{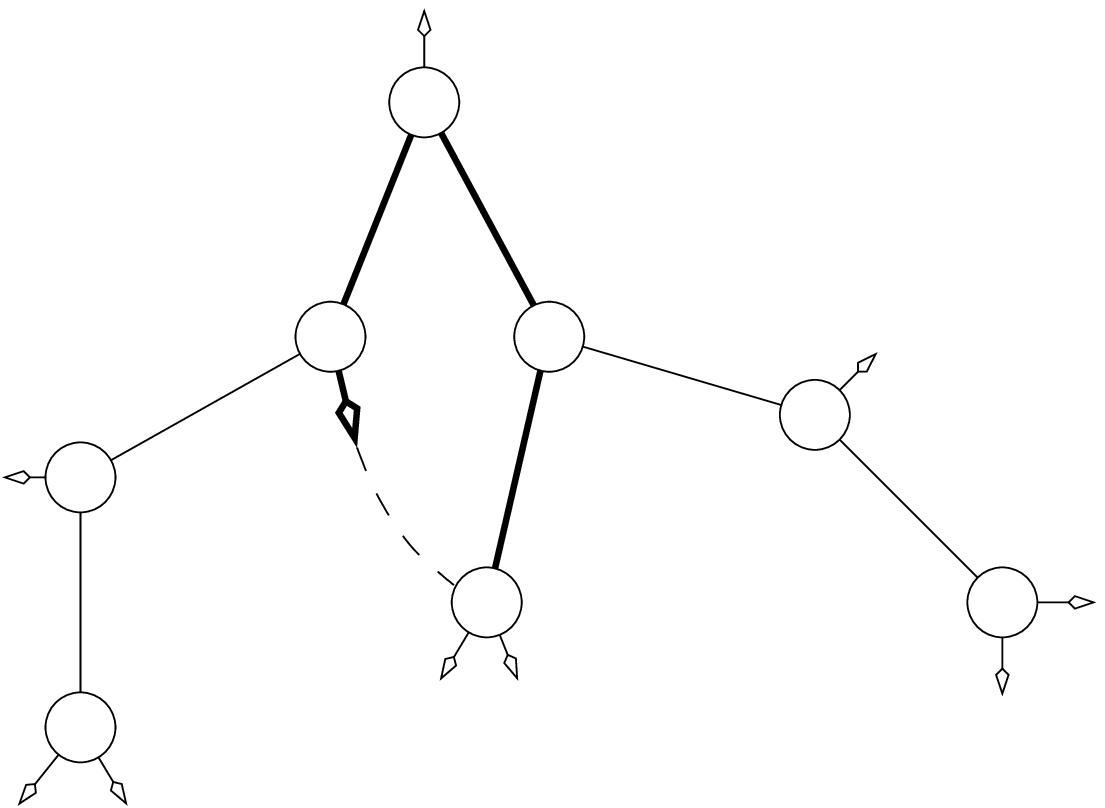}} 
  \quad
  \subfigure[\scriptsize and the partial
    closure.\label{figure:partialClosure}]{\quad\includegraphics[width=\etalon]{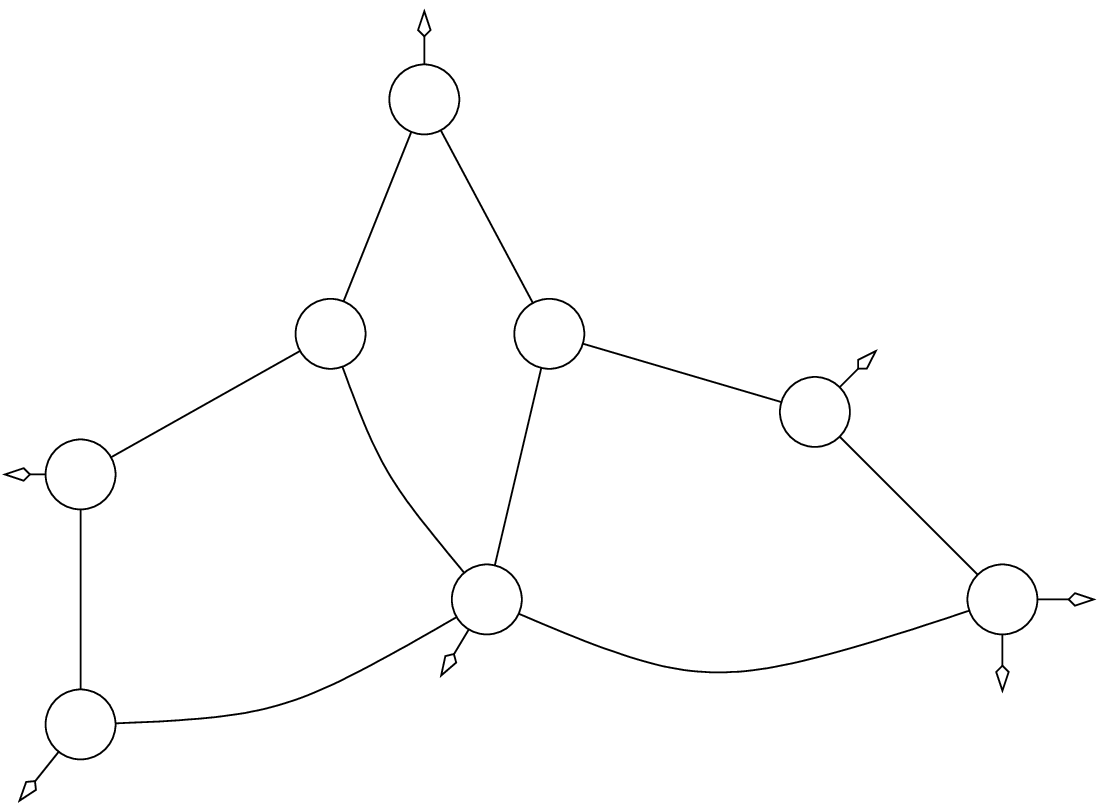}\quad} 
  \caption{The partial closure.}
\end{figure}

Given a binary tree $T$, the local closure can be performed greedily
until no more local closure is possible. Each local closure creates a
new entire edge, maybe making a new local closure possible. It is
easy to see that the final map, called the \emph{partial closure}
of $T$, does not depend on the order of the local closures. Indeed, a
cyclic parenthesis word is associated to the counter-clockwise boundary
of the tree, with an opening parenthesis of weight~3 for a stem and
a closing parenthesis for a side of entire edge; then the future local
closures correspond to matchings of the parenthesis word. An
example of partial closure is shown in
Figure~\ref{figure:partialClosure}.

\paragraph{Complete closure}
Let us now complete the partial closure operation to obtain a
dissection of the hexagon with quadrangular faces. An \emph{outer
  entire half-edge} is an half-edge belonging to an entire edge and
incident to the outer face.  Observe that a binary tree $T$ with $n$
nodes has $n+2$ stems and $2n-2$ outer entire half-edges. Each local
closure decreases by~1 the number of stems and by~2 the number of
outer entire half-edges.  Hence, if $k$ denotes the number of
(unmatched) stems in the partial closure of $T$, there are $2k-6$
outer entire half-edges. Moreover, stems delimit intervals of inner
half-edges on the contour of the outer face; these intervals have
length at most~2, otherwise a local closure would be possible.  Let
$r$ be the number of such intervals of length~1 and $s$ be the number
of such intervals of length~0 (that is, the number of nodes incident
to two unmatched stems). Then $r$ and $s$ are clearly related by the
relation $r+2s=6$.

\begin{figure}
    \subfigure[Generic case when $r=2$ and $s=2$.\label{figure:exampleClosureCompleted}]{%
      \qquad\includegraphics[height=.33\linewidth]{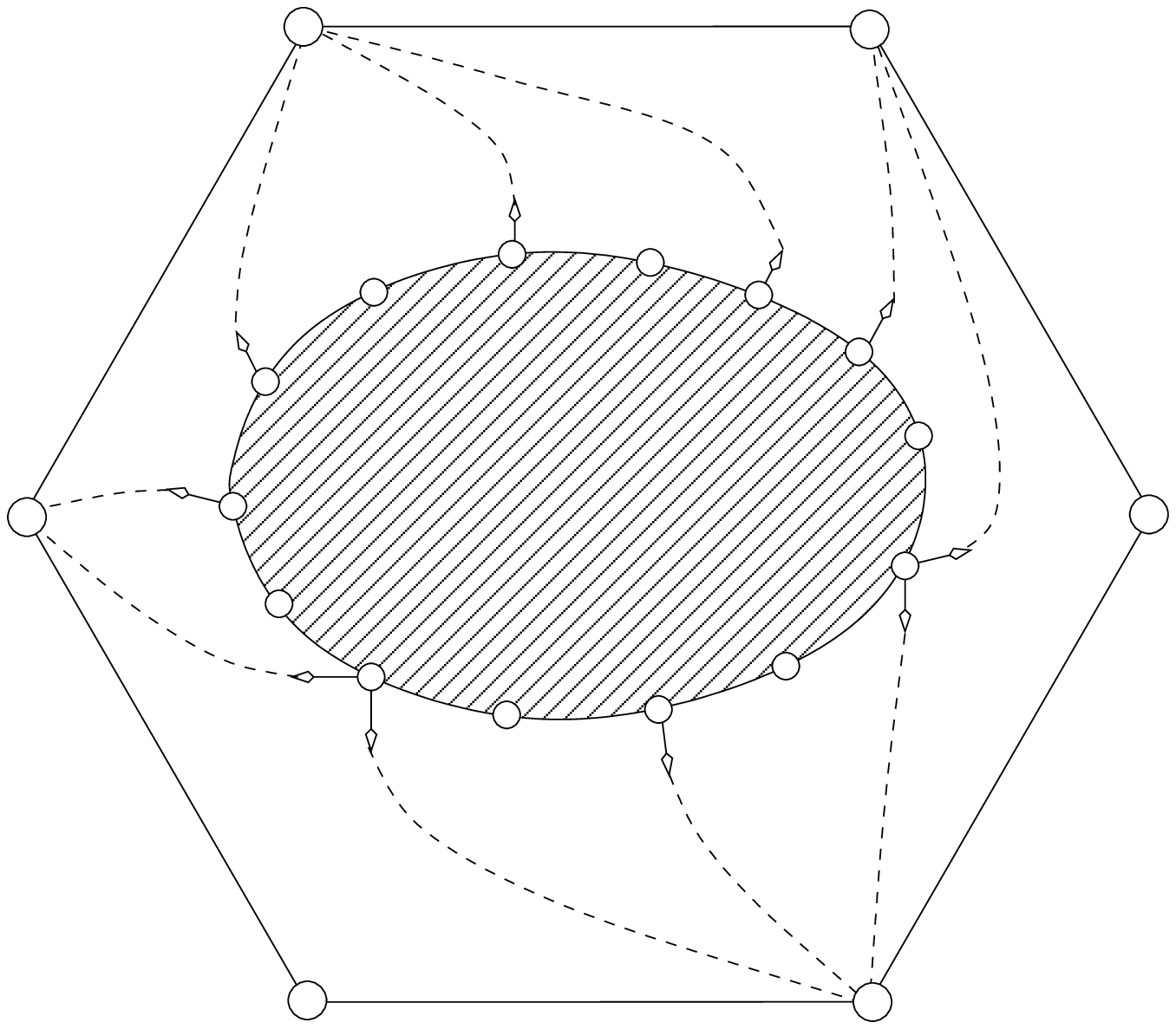}\qquad}
    \quad   \vline 
    \subfigure[Case of the binary tree of
      Figure~\ref{figure:binaryTree}.\label{figure:completeClosure}]{%
      \quad \qquad\includegraphics[height=.33\linewidth]{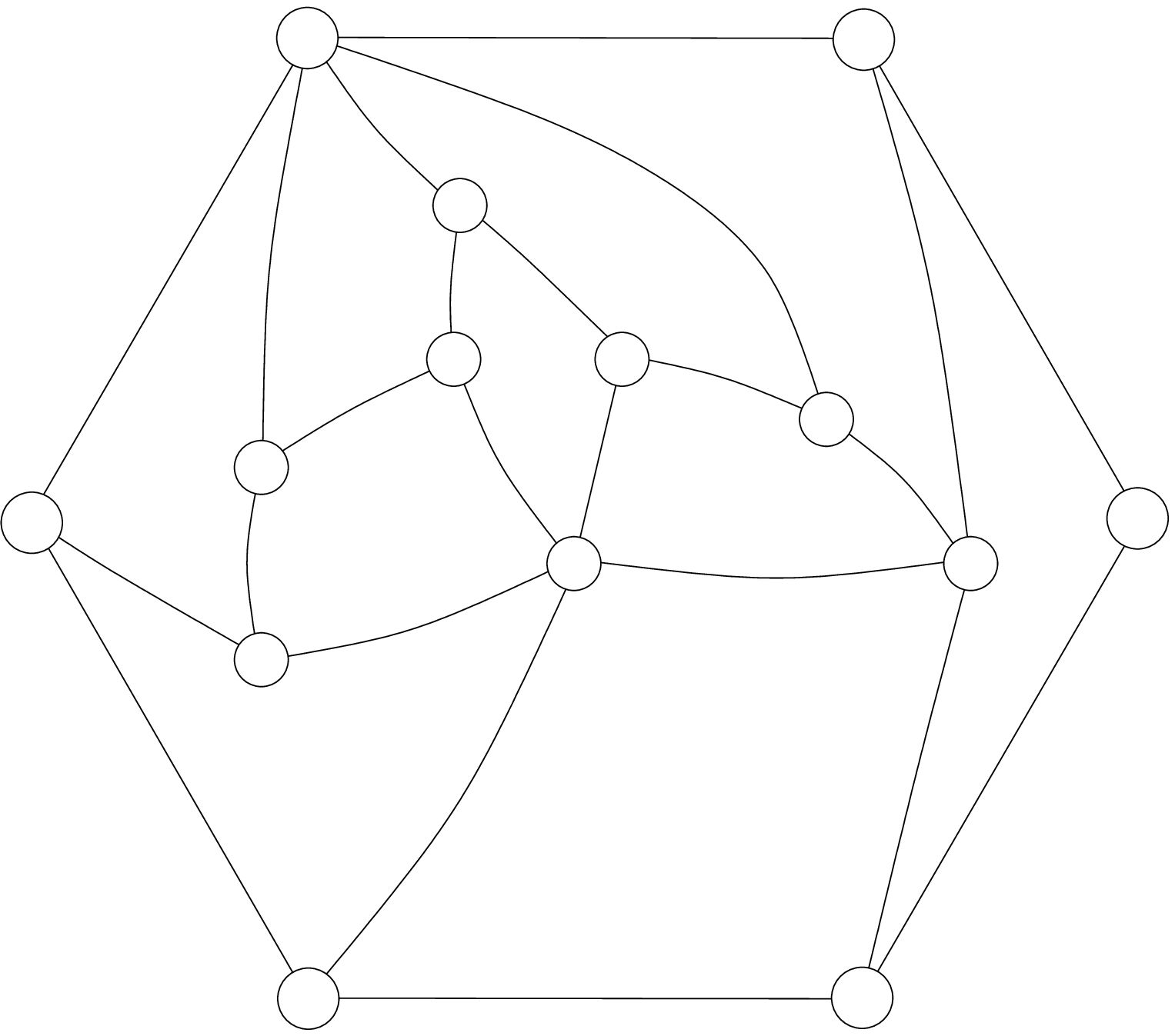}\qquad\quad} 
    \caption{The complete closure.}
\end{figure}

The \emph{complete closure} consists in completing all unmatched stems
with half-edges incident to vertices of the hexagon in the unique way
(up to rotation of the hexagon) that creates only quadrangular bounded
faces.  Figure~\ref{figure:exampleClosureCompleted} illustrates the
complete closure for the case $(r=2,s=2)$, and a particular example is given
in Figure~\ref{figure:completeClosure}.

\begin{lemma}
  \label{lemma:closureirreducible}
  The closure of a binary tree is an irreducible dissection of the hexagon.
\end{lemma}

\begin{proof}
  Assume that there exists a separating 4-cycle $\mathcal{C}$ in the
  closure of $T$. Let $m\geq 1$ be the number of vertices in the
  interior of $\mathcal{C}$. Then there are $2m$ edges in the interior
  of $\mathcal{C}$ according to Euler's relation.  Let $v$ be a vertex
  of $T$ that belongs to the interior of $\mathcal{C}$ after the
  closure. Consider the orientation of edges of $T$ away from $v$
  (only for the sake of this proof). Then nodes of $T$ have
  outdegree~2, except~$v$, which has outdegree~3. This orientation
  naturally induces an orientation of edges of the closure-dissection
  with the same property (except that vertices of the hexagon have
  outdegree~0). Hence there are at least $2m+1$ edges in the interior
  of $\mathcal{C}$, a contradiction.  \hfill $\qed$\end{proof}

\subsection{Tri-orientations and opening} \label{section:tri-orientations}


\paragraph{Tri-orientations}
In order to define the mapping inverse to the closure, we need a
better description of the structure induced on the closure map by the
original tree. Let us consider orientations of the \emph{half-edges}
of a map (in contrast to the usual notion of orientation, where edges
are oriented). An half-edge is said to be \emph{inward} if it is
oriented toward its origin and \emph{outward} if it is oriented out of
its origin.  If a map is endowed with an orientation of its
half-edges, the \emph{outdegree} of a vertex $v$ is naturally defined
as the number of its incident half-edges oriented outward.  The
(unique) \emph{tri-orientation} of a binary tree is defined as the
orientation of its half-edges such that any node has outdegree~3, see
Figure~\ref{figure:binaryTreeOriented} for an example.  A
\emph{tri-orientation} of a dissection is an orientation of its inner
half-edges (i.e., half-edges belonging to inner edges) such that outer
and inner vertices have respectively outdegree 0 and~3, and such that
two half-edges of a same inner edge can not both be oriented inward,
see Figure~\ref{figure:dissectionOriented}.
An edge is said to be \emph{simply oriented} if its two 
half-edges have same direction
(that is, one is oriented inward and the other one outward), and
\emph{bi-oriented} if they are both oriented outward.

\begin{figure}
  \centering
  \subfigure[A tri-oriented binary
    tree,\label{figure:binaryTreeOriented}]{\qquad \includegraphics[height=.33\linewidth]{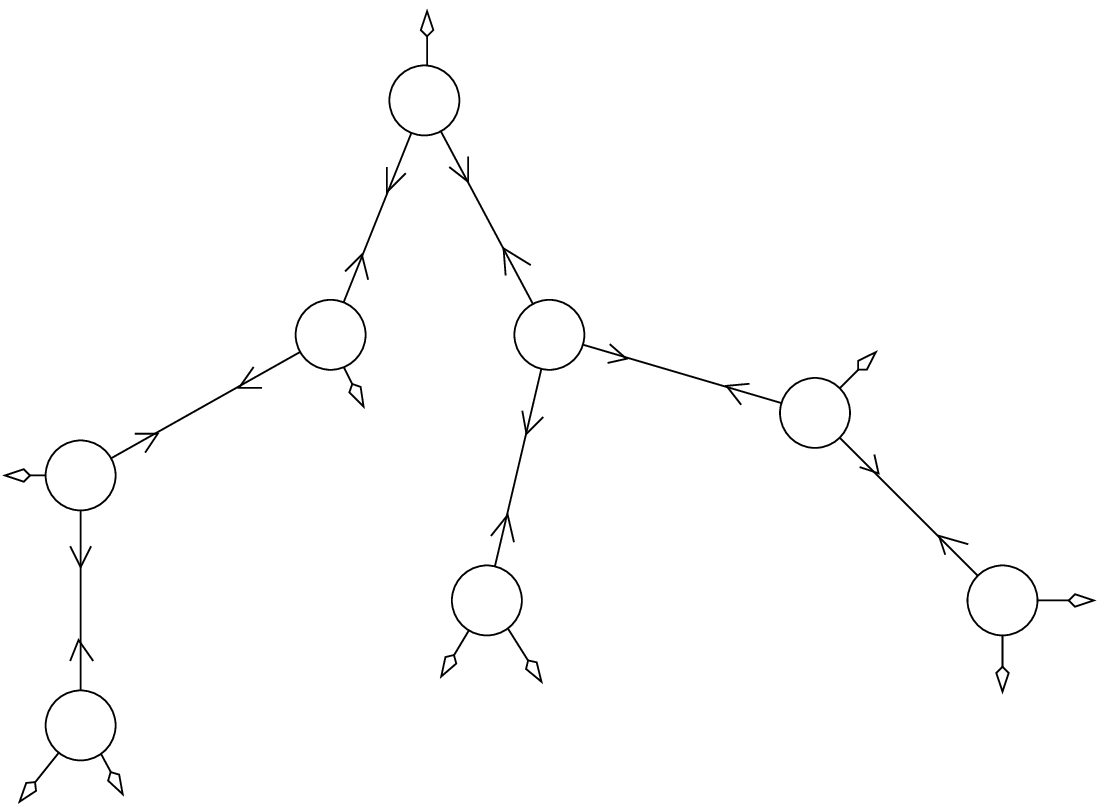}\qquad }
  \qquad \subfigure[and its tri-oriented
    closure.\label{figure:dissectionOriented}]{\includegraphics[height=.33\linewidth]{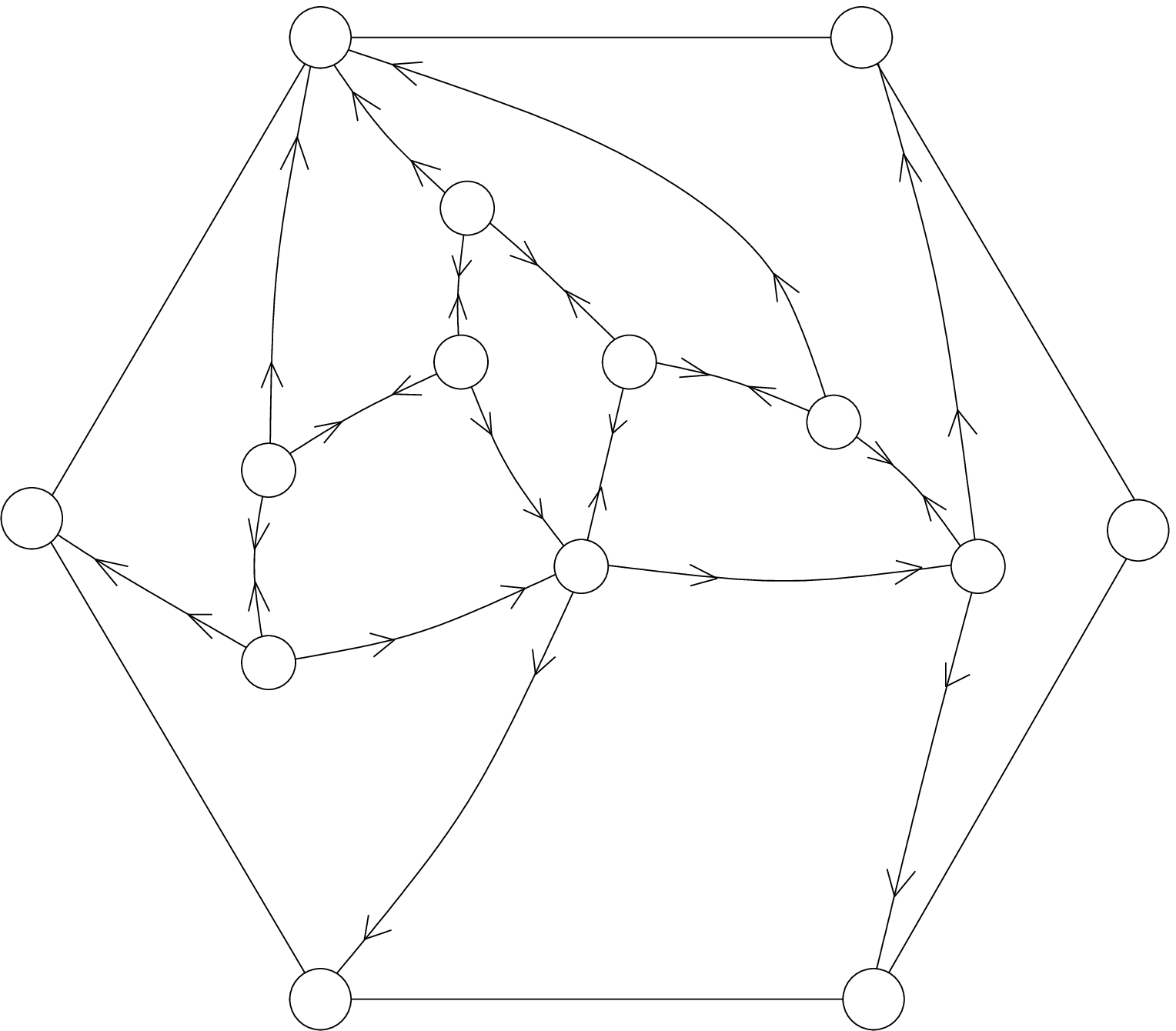}}
  \caption{Examples of tri-orientations.}
\end{figure}


Let $D$ be an irreducible dissection endowed with a tri-orientation. A
\emph{clockwise circuit} of $D$ is a simple cycle $\mathcal{C}$
consisting of edges that are either bi-oriented or simply oriented
with the interior of $\mathcal{C}$ on their right.

\begin{lemma}  \label{proposition:triOrientationDissection}
  Let $D$ be an irreducible dissection with $n$ inner vertices.  Then
  a tri-orientation of $D$ has $n-1$ bi-oriented edges and $n+2$
  simply oriented edges.
  
  If a tri-orientation of a dissection has no clockwise circuit,
  then its bi-oriented edges form a tree spanning the inner
  vertices of the dissection.
\end{lemma}

\begin{proof} 
  Let $s$ and $r$ denote the numbers of simply and bi-oriented edges
  of $D$.  According to Euler's relation (using the degrees of the
  faces), $D$ has $2n+1$ inner edges, i.e., $2n+1=r+s$. Moreover, as
  all inner vertices have outdegree~3, $3n=2r+s$.  Hence $r=n-1$ and
  $s=n+2$.

  If the tri-orientation has no clockwise circuit, the subgraph $H$
  induced by the bi-oriented edges has $r=n-1$ edges, no cycle
  (otherwise the cycle could be traversed clockwise, as all its edges
  are bi-oriented), and is incident to at most $n$ vertices, which are
  the inner vertices of $D$. According to a classical result of graph
  theory, $H$ is a tree spanning the $n$ inner vertices of~$D$.
  \hfill $\qed$\end{proof}

\paragraph{Closure-tri-orientation of a dissection}
Let $D$ be a dissection obtained as the closure of a binary tree
$T$. The tri-orientation of $T$ clearly induces \emph{via} the closure
a tri-orientation of $D$, called \emph{closure-tri-orientation}. On
this tri-orientation, bi-oriented edges correspond to inner edges of
the original binary tree, see Figure~\ref{figure:dissectionOriented}.

\begin{lemma}
  \label{lemma:noclock}
  A closure-tri-orientation has no clockwise circuit.
\end{lemma}

\begin{proof} 
  Since vertices of the hexagon have outdegree~0, they can not belong
  to any circuit.  Hence clockwise circuits may only be created during
  a local closure.  However closure edges are simply oriented with the
  outer face on their right, hence may only create counterclockwise
  circuits.  \hfill $\qed$\end{proof}

This property is indeed quite strong: the following theorem ensures
that the property of having no clockwise circuit characterizes the
closure-tri-orientation and that a tri-orientation without clockwise
circuit exists for any irreducible dissection. The proof of this
theorem is delayed to Section~\ref{section:proof}.

\begin{theorem}
  \label{theorem:uniquenessexistence}\label{THEOREM:UNIQUENESSEXISTENCE}
  Any irreducible dissection has a unique tri-orientation without
  clockwise circuit.
\end{theorem}

\penalty 10000

\paragraph{Recovering the tree: the opening mapping}\label{section:opening}
Lemma~\ref{proposition:triOrientationDissection} and the present
section give all necessary elements to describe the inverse mapping of
the closure, which is called the \emph{opening}: let $D$ be an
irreducible dissection endowed with its (unique by
Theorem~\ref{theorem:uniquenessexistence}) tri-orientation without
clockwise circuit. The \emph{opening} of $D$ is the binary tree
obtained from $D$ by deleting outer vertices, outer edges, and all
inward half-edges.

\subsection{The closure is a bijection}
\label{section:statementtheo}
In this section, we show that the opening is inverse to the
closure. By construction of the opening, the following lemma is
straightforward:

\begin{lemma}
  \label{lemma:closopen}
  Let $D$ be an irreducible dissection obtained as the closure of a
  binary tree $T$. Then the opening of $D$ is~$T$.
\end{lemma}

Conversely, the following also holds:

\begin{lemma}
  \label{lemma:openclos}
  Let $T$ be a binary tree obtained as the opening of an
  irreducible dissection $D$. Then the closure of $T$ is~$D$.
\end{lemma}

\begin{proof}
  The proof relies on the definition of an order for removing inward
  half-edges.  Start with the half-edges incident to outer vertices
  (that are all oriented inward): this clearly inverses the
  completion step of the closure. Each further removal must correspond
  to a local closure, that is, the removed half-edge must have the
  outer face on its right. 
 
  Let $M_k$ be the submap of the dissection induced by remaining
  half-edges after $k$ removals. Then $M_k$ covers the $n$ inner
  vertices, and, as long as some inward half-edge remains, it has at
  least $n$ entire edges (see
  Lemma~\ref{proposition:triOrientationDissection}). Hence, there is
  at least one cycle, and a simple one $\cC$ can be extracted from the
  boundary of the outer face of $M_k$. Since there is no clockwise
  circuit, at least one edge of $\cC$ is simply oriented with the
  interior of $\cC$ on its left; the corresponding inward half-edge
  can be selected for the next removal.  \hfill $\qed$\end{proof}

Assuming Theorem~\ref{theorem:uniquenessexistence}, the bijective
result follows from Lemmas~\ref{lemma:closopen}
and~\ref{lemma:openclos}:

\begin{theorem}
  \label{theorem:bijection}
  For each $n\geq 1$, the closure mapping is a bijection between the
  set $\nBn$ of binary trees with $n$ nodes and the set $\nDn$ of
  irreducible dissections with $n$ inner vertices.

  For each integer pair $(i,j)$ with $i+j\geq 1$, the closure mapping
  is a bijection between the set $\mathcal{B}_{ij}$ of bicolored
  binary trees with $i$ black nodes and $j$ white nodes, and the set
  $\mathcal{D}_{ij}$ of bicolored irreducible dissections with $i$
  black inner vertices and $j$ white inner vertices.

  The inverse mapping of the closure is the opening.
\end{theorem}

We can state three analogous versions of Theorem~\ref{theorem:bijection}
for rooted objects:

\begin{theorem}
  \label{theorem:bijectionRooted}
  The closure mapping induces the following correspondences
  between sets of rooted objects:
  \begin{eqnarray*}
    \Bn\times\{1,\ldots,6\} & \equiv &
    \Dn\times\{1,\ldots ,n+2\},\\
    \Bij\times\{1,2,3\}  &\equiv &
    \Dij\times\{1,\ldots ,i+j+2\},\\
    \cBijb\times\{1,2,3\}  &\equiv &
    \Dij\times\{1,\ldots ,2i-j+1\}.
  \end{eqnarray*}
\end{theorem}

\begin{proof}
  We define a \emph{bi-rooted irreducible dissection} as a rooted
  irreducible dissection endowed with its tri-orientation without
  clockwise circuit and where a simply oriented edge is marked. We
  write $\mathcal{D}_n''$ for the set of bi-rooted irreducible
  dissections with $n$ inner vertices.  Opening and rerooting on the
  stem corresponding to the marked edge defines a surjection from
  $\mathcal{D}_n''$ onto \Bn, for which each element of \Bn has
  clearly six preimages, since the dissection could have been rooted at
  any edge of the hexagon.  Moreover, erasing the mark clearly defines
  a surjection from $\mathcal{D}_n''$ to \Dn, for which each element
  of \Dn has $n+2$ preimages according to
  Lemma~\ref{proposition:triOrientationDissection}.  Hence, the
  closure defines a $(n+2)$-to-6 mapping between \Bn and \Dn. The
  proof of the $(i+j+2)$-to-3 correspondence between $\Bij$ and $\Dij$
  is the same.

  The $(2i-j+1)$-to-3 correspondence between \cBijb
  and \Dij induced by the closure can be proved similarly, with the 
  difference that the marked simply oriented edge has
  to have a black vertex as origin.  Then the result follows from the fact
  that an object of \Dij endowed with its tri-orientation without
  clockwise circuit has $(2i-j+1)$ simply oriented edges whose origin
  is a black vertex.
\hfill $\qed$\end{proof}
Let us mention that the $(i+j+2)$-to-3 correspondence between \Bij and 
$\Dij$ is a key ingredient to the planar graph generators
presented in~\cite{Fusy}.

The coefficient $|\Bn|$ is well-known to be the $n$-th Catalan number
$\frac1{n+1}{2n\choose n}$, and refinements of the standard proofs
yield $|\cBijb|=\frac{1}{2j+1}{2j+1\choose i}{2i\choose j}$, as detailed
below in Section~\ref{sec:countbic}. 
Theorem~\ref{theorem:bijectionRooted} thus implies the following
enumerative results:

\begin{corollary}
  \label{corollary:counting}
The coefficients counting rooted irreducible dissections have the
following expressions,
  \begin{equation}
  |\Dn|
  ~=~\frac{6}{n+2}|\Bn|
  ~=~ \frac6{(n+2)(n+1)}\binom{2n}{n},\end{equation}
  \begin{equation}
   \label{eq:coefDij}
  |\Dij|
  ~=~\frac{3}{2i-j+1}|\cBijb|
  ~=~\frac{3}{(2i+1)(2j+1)}\binom{2j+1}{i}\binom{2i+1}{j}.\end{equation}
\end{corollary}

These enumerative results have already been obtained by~\citeN{Mu}
using algebraic methods. Our method provides a direct bijective proof.

Notice that the cardinality of \Dn is $\frac12S(n,2)$ where
$S(n,m)=\frac{(2n)!(2m)!}{n!m!(n+m)!}$ is the $n$-{th} super-Catalan
number of order $m$. (These numbers are discussed by~\citeN{Ge}.)
Our bijection gives an interpretation of these numbers for $m=2$.

\subsection{Specialization to triangulations}

A nice feature of the closure mapping is that it specializes to a
bijection between plane triangulations and a simple subfamily of
binary trees. In this way, we get the first bijective proof for the
formula giving the number of unrooted plane triangulations with $n$
vertices, found by~\citeN{Br}, and recover the counting formula
for rooted triangulations, already obtained by~\citeN{T62} and
by~\citeN{PS03b} using a different bijection.

\begin{theorem}
The closure mapping is a bijection between the set $\cT_n$ of
(unrooted) plane triangulations with $n$ inner vertices and the set
$\cS_n$ of bicolored binary trees with $n$ black nodes and no stem
(i.e., leaf) incident to a black node.

The closure mapping induces the following correspondence between the
set $\cT_n'$ of rooted triangulations with $n$ inner vertices and the
set $\cS_n'$ of trees in $\cS_n$ rooted at a stem:
\[
\cS_n'\times\{1,2,3\}\equiv \cT_n'\times \{1,\ldots,3n+3\}.
\]
\end{theorem}

\begin{proof}
  Plane triangulations are exactly 3-connected planar maps where all
  faces have degree~3.  Hence, the angular mapping with border
  (Theorem~\ref{thm:bijOuter}) induces a bijection between $\cT_n$ and
  the set of complete bicolored irreducible dissections with $n$ inner
  black vertices and all inner white vertices of degree~3.  In a
  tri-orientation, the indegree of each inner white vertex $v$ is
  $\mathrm{deg}(v)-3$ and the indegree of each outer white vertex $v$
  is $\mathrm{deg}(v)-2$, hence the dissections considered here have
  no ingoing half-edge incident to a white vertex.  Hence the opening
  of the dissection (by removing ingoing half-edges) is a binary tree
  with no stem incident to a black node.  Conversely, starting from
  such a binary tree, the half-edges created during the closure
  mapping are opposite to a stem.  As all stems are incident to white
  vertices, the half-edges created are incident to black vertices.
  Hence the degree of each white vertex does not increase during the
  closure mapping, i.e., remains equal to~3 for inner white vertices
  and equal to~2 for outer white vertices. This concludes the proof of
  the bijection $\cS_n\equiv\cT_n$.

  The bijection $\cS_n'\times\{1,2,3\}\equiv \cT_n'\times
  \{1,\ldots,3n+3\}$ follows easily (see the proof of
  Theorem~\ref{theorem:bijectionRooted}), using the fact that a tree of
  $\cS_n$ has $3n+3$ leaves.  
  \hfill $\qed$
\end{proof}

\begin{figure}
  \def\etalon{.38\linewidth}
  \centering
  \subfigure[]{\includegraphics[width=\etalon]{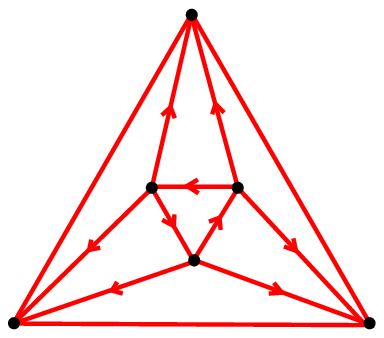}\par\mbox{}\par}\qquad\qquad
  \subfigure[]{\includegraphics[width=\etalon]{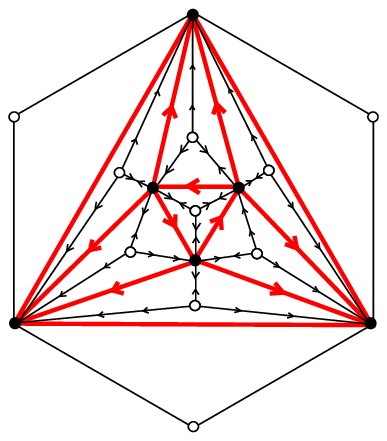}}

  \subfigure[]{\includegraphics[width=\etalon]{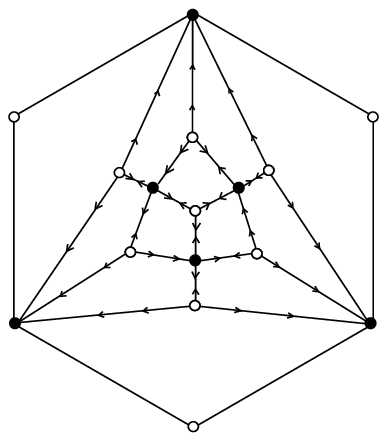}}\qquad\qquad
  \subfigure[]{\includegraphics[width=\etalon]{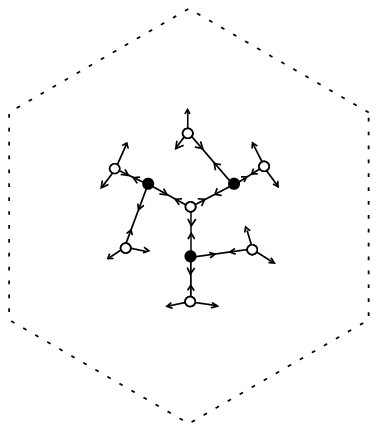}}
  \caption{The bijection between triangulations and bicolored binary
    trees with no leaf incident to a black node.}
  \label{fig:bij_triang}
\end{figure}

This bijection, illustrated in Figure~\ref{fig:bij_triang}, makes it
possible to count plane unrooted and rooted triangulations, as the
subfamily of binary trees involved is easily enumerated.

\begin{corollary}
For $n\geq 0$, the number of rooted triangulations with $n$ inner
vertices is
\[
|\cT_n'|=2\frac{(4n+1)!}{(n+1)!(3n+2)!}.
\]
The number of unrooted plane triangulations with $n$ inner vertices is

\begin{eqnarray*}
|\cT_n|&=&\frac{2}{3}\frac{(4n+1)!}{(n+1)!(3n+2)!} \mathrm{\ \ \ \ \ \ \ \ \ \ \ \ \ \ \ \ \ \ \ \ \ \ \ \ \ \ \! if}\ n\equiv 2\!\!\mod 3,\\[0.2cm]
|\cT_n|&=&\frac{2}{3}\frac{(4n+1)!}{(n+1)!(3n+2)!}+\frac{4}{3}\frac{(4k+1)!}{k!(3k+2)!} \mathrm{\ \ \ \ \  \! \ \ if}\ n\equiv 1\!\!\mod 3\ \ [n=3k+1],\\[0.2cm]
|\cT_n|&=&\frac{2}{3}\frac{(4n+1)!}{(n+1)!(3n+2)!}+\frac{2}{3}\frac{(4k)!}{k!(3k+1)!} \mathrm{\ \ \ \ \ \ \ \! if}\ n\equiv 0\!\!\mod 3\ \ [n=3k].
\end{eqnarray*}
\end{corollary}

\begin{proof}
  Let $\cS'=\cup_n\cS_n'$ be the class of rooted binary trees with no
  leaf incident to a black node and let $\cR'=\cup_n\cR_n'$ be the
  class of rooted binary trees where the root leaf is incident to a
  black node and all other leaves are incident to white nodes.  Let
  $S(x)$ and $R(x)$ be the generating functions of $\cS'$ and $\cR'$
  with respect to the number of black nodes.  Clearly the two subtrees
  pending from the (white) root node of a tree of $\cS'$ are either
  empty or in $\cR'$.  Hence $S(x)=(1+R(x))^2$.  Similarly, a tree in
  $\cR'$ decomposes at the root node into two trees in $\cS'$, so that
  $R(x)=xS(x)^2$. Hence, $R(x)=x(1+R(x))^4$ is equal to the generating
  function of quaternary trees, and $S(x)=(1+R(x))^2$ is equal to the
  generating function of pairs of quaternary trees (the empty tree
  being allowed). Using a Lukaciewicz encoding and the cyclic lemma,
  the number of pairs of quaternary trees with a total of $n$ nodes is
  easily shown to be $\frac{2}{4n+2}\frac{(4n+2)!}{n!(3n+2)!}$.  This
  expression of $|\cS_n'|$ and the $(3n+3)$-to-3 correspondence
  between $\cS_n'$ and $\cT_n'$ yield the expression of $|\cT_n'|$.

  Let us now prove the formula for $|\cT_n|=|\cS_n|$. Clearly, the
  only possible symmetry for a bicolored binary tree is a rotation of
  order~3. Let $\cSsy_n$ be the set of trees of $\cS_n$ with a
  rotation symmetry and let $\cSas_n$ be the set of trees of $\cS_n$
  with no symmetry.  Let $\cSasy_n$ and $\cSsym_n$ be the sets of
  trees of $\cSas_n$ and $\cSsy_n$ that are rooted at a leaf. It is
  easily shown that a tree of $\cS_n$ has $3n+3$ leaves. Clearly the
  tree gives rise to $3n+3$ rooted trees if it is asymmetric and gives
  rise to $n+1$ rooted trees if it is symmetric. Hence
  $|\cSas_n|=|\cSasy_n|/(3n+3)$ and
  $|\cSsy_n|=|\cSsym_n|/(n+1)$. Using $|\cS_n|=|\cSas_n|+|\cSsy_n|$
  and $|\cS_n'|=|\cSasy_n|+|\cSsym_n|$, we obtain
  \[
  |\cS_n|=\frac{1}{3n+3}|\cS_n'|+\frac{2}{3}|\cSsy_n|.
  \]
  The centre of rotation of a tree in $\cSsy_n$ is either a black
  node, in which case $n=3k+1$ for some integer $k\geq 0$, or is a
  white node, in which case $n=3k$ for some integer $k\geq 0$.  In the
  first case, a tree $\tau\in\cSsy_n$ is obtained by attaching to a
  black node 3 copies of a tree in $\cS_k'$.  Hence
  $|\cSsy_{3k+1}|=|\cS_k'|=2\frac{(4k+1)!}{k!(3k+2)!}$.  In the second
  case, a tree $\tau\in\cSsy_n$ is obtained by attaching to a white
  node 3 copies of a tree in $\cR_k'$.  Hence
  $|\cSsy_{3k}|=|\cR_k'|=\frac{(4k)!}{k!(3k+1)!}$.  The result
  follows.
\hfill $\qed$
\end{proof}

\subsection{Counting, coding and sampling rooted bicolored binary trees}
\label{sec:countbic}
\begin{figure}
  \centering
  \subfigure[$\mathcal{A}_{\bullet\circ}$,\label{figure:alph_noir_blanc}]{\input{Figures/Alphabet_noir_blanc.pstex_t}}\qquad\qquad
  \subfigure[$\mathcal{A}_{\bullet}$,\label{figure:alph_noir}]{\qquad\includegraphics{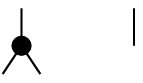}\qquad}\qquad
  \subfigure[$\mathcal{A}_{\circ}$.\label{figure:alph_blanc}]{\qquad\includegraphics{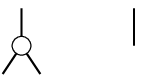}\qquad}
  \caption{The three alphabets for words associated to bicolored
    binary trees.}
  \label{figure:Alphabets}
\end{figure}

\subsubsection{From a bicolored tree to a pair of words}\label{section:BicoloredTreesCode}
There exist general methods to encode a family of trees specified by
several parameters. This section makes such methods explicit for the
family of bicolored binary trees.  Let $T$ be a black-rooted bicolored
binary tree with $i$ black nodes and $j$ white nodes.  Doing a
depth-first traversal of $T$ from left to right, we obtain a word
$w_{\bullet\circ}$ of length $(2j+1)$ on the alphabet
$\mathcal{A}_{\bullet\circ}$ represented in
Figure~\ref{figure:alph_noir_blanc}, see
Figure~\ref{figure:BicoloredTreeToWord} for an example, the mapping
being denoted by $\Psi$. Classically, the sum of the weights of the
letters of any strict prefix of $w_{\bullet\circ}$ is nonnegative and
the sum of the weights of the letters of $w_{\bullet\circ}$ is equal
to~-1. In addition, $w_{\bullet\circ}$ is the unique word in its
cyclic equivalence-class that has these two properties.

The second step is to map $w_{\bullet\circ}$ to a pair
$(w_{\bullet},w_{\circ}):=\Phi(w_{\bullet\circ})$ of words such that:
\begin{longitem}
\item $w_{\bullet}$ is a word of length $(2j+1)$ on the alphabet
  $\mathcal{A}_{\bullet}$ shown in
  Figure~\ref{figure:alph_noir} with $i$ black-node-letters.
\item $w_{\circ}$ is a word of length $2i$ on the alphabet
  $\mathcal{A}_{\circ}$ shown in
  Figure~\ref{figure:alph_blanc} with $j$ white-node-letters.
\end{longitem}
Figure~\ref{figure:WordToWordPair} illustrates the
mapping $\Phi$ on an example.

\begin{figure}
  \centering
  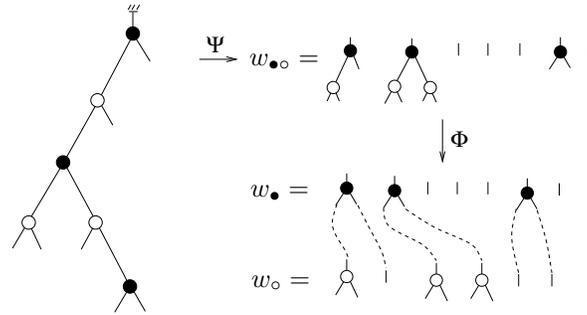
  \caption{A bicolored rooted binary tree, and the corresponding
    words $w_{\bullet\circ}$, $w_{\bullet}$, and $w_{\circ}$.}
  \label{figure:BicoloredTreeToWord}\label{figure:WordToWordPair}
\end{figure}

\subsubsection{Inverse mapping: from a pair of words to a tree}
\label{section:BicoloredTreesSample}

Conversely, let $(w_{\bullet},w_{\circ})$ be a pair of words such that 
$w_{\bullet}$ is of length $(2j+1)$ on
$\mathcal{A}_{\bullet}$ and has $i$ black-node-letters, and $w_{\circ}$ is
of length $2i$ on $\mathcal{A}_{\circ}$ and has $j$ white-node-letters.
First, to the pair $(w_{\bullet},w_{\circ})$ we associate a word 
$\widetilde{w}_{\bullet\circ}$ of length $(2j+1)$ on
$\mathcal{A}_{\bullet\circ}$ by doing the inverse  of the
mapping $\Phi$ shown in the right part of 
Figure~\ref{figure:WordToWordPair}. The
word $\widetilde{w}_{\bullet\circ}$ 
has the property that the sum of the weights of its letters is
equal to -1.  There is a unique word $w_{\bullet\circ}$ in the
cyclic equivalence-class of $\widetilde{w}_{\bullet\circ}$ such that
 the sum of the weights of the letters of
 any strict prefix is nonnegative. We associate to
$w_{\bullet\circ}$ the binary tree of $\cBijb$ 
obtained by doing the
inverse of the mapping $\Psi$ shown 
in Figure~\ref{figure:BicoloredTreeToWord}.

This method allows us to sample uniformly objects of $\cBijb$ in linear time and
ensures that
\begin{equation}
|\cBijb|=\frac{1}{2j+1}\binom{2j+1}{i}\binom{2i}{j}.
\end{equation}

\section{Application: counting rooted 3-connected maps}
\label{section:counting}

\subsection{Generating functions of rooted dissections}

Even if the counting formulas obtained in
Corollary~\ref{corollary:counting} are simple, it proves useful
to have an expression of the corresponding generating functions. 
Indeed, the decomposition-method
we develop is suitably handled by generating functions. 
   
Let 
$r_1(x_{\bullet},x_{\circ}):=\sum |\cBijb|x_{\bullet}^ix_{\circ}^j$
 and $r_2(x_{\bullet},x_{\circ}):=\sum |\cBijw|x_{\bullet}^ix_{\circ}^j$
be the series of black-rooted and white-rooted bicolored binary trees.
By decomposition at the root, $r_1(\xb,\xw)$ and $r_2(\xb,\xw)$ 
 are the solutions of the system:
\begin{equation}
\label{eq:sysr}
\left\{
\begin{array}{ccc}
r_1(x_{\bullet},x_{\circ})&=&x_{\bullet}\left(
1+r_2(x_{\bullet},x_{\circ})\right) ^2,  \\
r_2(x_{\bullet},x_{\circ})&=&x_{\circ}\left(
1+r_1(x_{\bullet},x_{\circ})\right) ^2. 
\end{array}
\right.
\end{equation}

Define an \emph{edge-marked bicolored binary tree} as a bicolored binary tree
with a marked inner edge. Let $\mBij$ be the set of 
edge-marked bicolored binary trees with $i$ black
nodes and $j$ white nodes. Cutting the marked edge of such a tree
yields a pair made of a black-rooted and a white-rooted binary
tree. As a consequence, the generating function counting edge-marked
bicolored binary trees is $r_1\cdot r_2$, i.e., 
$r_1\cdot r_2=\sum_{ij}|\mBij|x_{\bullet}^ix_{\circ}^j$.

Let us consider bi-rooted objects as in the proof of
Theorem~\ref{theorem:bijectionRooted}; since any
object of \nBij has $(2i-j+1)$ white leaves (connected to a black
node) 
and $(2j-i+1)$ black leaves (connected to a white node),
\[
|\cBijw|=\frac{2j-i+1}{2i-j+1}|\cBijb|.
\]
Similarly, counting in two ways  the objects of \cBijb having a marked edge yields
\[
|\mBij|=\frac{i+j-1}{2i-j+1}|\cBijb|.
\]
Thus, we have
$|\cBijb|+|\cBijw|-|\mBij|=\frac{3}{2i-j+1}|\cBijb|=|\mathcal{D}_{ij}'|$
(using~(\ref{eq:coefDij})), so that
\begin{equation}
  \label{eq:arbres}
  \sum_{i,j}|\Dij|x_{\bullet}^ix_{\circ}^j ~=~
  r_1(x_{\bullet},x_{\circ}) + r_2(x_{\bullet},x_{\circ}) -
  r_1(x_{\bullet},x_{\circ}) r_2(x_{\bullet},x_{\circ}).
\end{equation}
Substituting $x_{\bullet}$ and $x_{\circ}$ by $x$, we
obtain:
\begin{equation}
  \label{eq:arbressub}
  \sum_{n}|\Dn|x^n=2r(x)-r(x)^2,
\end{equation}
where $r(x)=x\left( 1+r(x)\right) ^2$ is the generating function of
binary trees according to the number of inner nodes.

\subsection{Generating function of rooted 3-connected maps}

\paragraph{Injection from \Q to \D}

Let us consider the mapping $\iota$ defined on rooted
quadrangulations by the removal of the root-edge and rerooting on
the next edge in counterclockwise order around the
root-vertex; $\iota$ is clearly injective, and for any quadrangulation
$Q$, $\iota(Q)$ has only quadrangular faces but the outer one, which
is hexagonal. In addition, $\iota(Q)$ 
can not have more separating 4-cycles than $Q$.
Hence the restriction of $\iota$ to \Q is an injection from \Q
to~\D, more precisely from \Qn to $\D_{n-4}$ and from \Qij to
$\D_{i-3,j-3}$. 

It is however not a bijection, since the inverse edge-adding
operation $\pi$, performed on an irreducible dissection, can create a
separating 4-cycle on the obtained quadrangulation. Precisely, given
 $D$ a rooted irreducible dissection ---with $s$ the root-vertex and
 $t$ the vertex of the hexagon opposite to $s$--- a path of length~3
between $s$ and $t$ is called a \emph{decomposition path}.  The
two paths of edges of the hexagon connecting $s$ to $t$
 are called \emph{outer decomposition paths}, and the other
ones, if any, are called \emph{inner decomposition paths} of~$D$.

Observe that inner decomposition paths of~$D$ are in one-to-one
correspondence with separating 4-cycles of the quadrangulation
$\pi(D)$ (i.e., the quadrangulation 
obtained from $D$ by adding a root-edge between $s$
and $t$ oriented out of $s$).

A rooted irreducible dissection without inner decomposition path is
said to be \emph{undecomposable}. The corresponding class is denoted
by $\mathcal{U}'$. 
The discussion on decomposition paths yields the following result.

\begin{lemma}
  \label{lemma:images}
  Denote by $\U_{n}$ the set of rooted undecomposable dissections with
  $n$ inner vertices and by $\U_{ij}$ the set of rooted undecomposable
  dissections with $i$ inner black vertices and $j$ inner white
  vertices.  Then $\U_{n-4}$ is in bijection with $\Pn$ and
  $\U_{i-3,j-3}$ is in bijection with $\Pij$.
\end{lemma}

\begin{proof}
A rooted irreducible quadrangulation is mapped
 by $\iota$ to a rooted dissection
 such that the inverse edge-adding operation $\pi$
does not create a separating 4-cycle, i.e., an undecomposable dissection.
Moreover, Euler's relation ensures that the image 
of a quadrangulation with $n$ faces has $n-4$ inner vertices. By injectivity,
$\iota$ is bijective to its image, i.e., $\iota$ is a bijection
between $\Qn$ and $\U_{n-4}$; 
and a bijection between $\Qij$ and $\U_{i-3,j-3}$.
The result follows, as $\Qn$ and $\Qij$ are 
respectively in bijection with $\Pn$ and $\Pij$ 
via the angular mapping (Theorem~\ref{theorem:tutte}).
\phantom{11111111} \hfill $\qed$  
\end{proof}

Thanks to Lemma~\ref{lemma:images}, enumerating rooted 3-connected maps
reduces to enumerating rooted undecomposable dissections.

\paragraph{Decomposition of rooted irreducible dissections}

Since irreducible dissections do not have multiple edges nor cycles of
odd length, decomposition paths satisfy the following properties:

\begin{lemma}
  \label{lemma:intersectpaths}
  Let $D$ be a rooted irreducible dissection, and let $\mathcal{P}_1$ and
  $\mathcal{P}_2$ be two different decomposition paths of~$D$. Then:
  \begin{longitem}
  \item either $\mathcal{P}_1 \cap \mathcal{P}_2 =\{ s,t\}$, in which
    case $\mathcal{P}_1$ and $\mathcal{P}_2$ are said to be \emph{internally
    disjoint};
  \item or there exists one inner vertex $v$ such that $\mathcal{P}_1
    \cap \mathcal{P}_2 =\{ s\} \cup \{ t\} \cup \{ v\}$, in which case
    $\mathcal{P}_1$ and $\mathcal{P}_2$ are said to be \emph{upper} or
    \emph{lower joint} whether $v$ is adjacent to $s$
    or~$t$. 
  \end{longitem}
\end{lemma}


Lemma~\ref{lemma:intersectpaths} implies in particular that two
decomposition paths can not cross each other. Hence the decomposition
paths of an irreducible dissection $D$ follow a left-to-right order, from the
outer decomposition path containing the root ---called left outer path---
to the other
outer decomposition path ---called right outer path.

\begin{lemma}
  \label{lemma:unique}
  Let $D$ be a rooted irreducible dissection, and let $\mathcal{P}_1$
  and $\mathcal{P}_2$ be two upper joint (resp. lower joint)
  decomposition paths of $D$. Then the interior of the area delimited
  by $\mathcal{P}_1$ and $\mathcal{P}_2$ consists of a unique face
  incident to $t$ (resp. to~$s$).
\end{lemma}

\begin{proof}
  Follows from the fact that the interior of each 4-cycle of $D$ is a
  face. \hfill $\qed$
\end{proof}

\paragraph{Decomposition word of an irreducible dissection}
Let $D\in \D$ and let $\{\mathcal{P}_0, \ldots,\mathcal{P}_\ell\}$ be
the sequence of decomposition paths of $D$ ordered from left to right.
Let us consider the alphabet
$\mathcal{A}=\{s\} \cup \{t\} \cup \U$; the \emph{decomposition word}
of $D$ is the word $w= w_1 \dots w_\ell$ of length $\ell$ on
$\mathcal{A}$ such that, for any $1\leq i\leq \ell$:
if $\mathcal{P}_{i-1}$ and $\mathcal{P}_i$ are
upper joint, then $w_i=s$; if $\mathcal{P}_{i-1}$ and $\mathcal{P}_i$ are
lower joint, then $w_i=t$; if $\mathcal{P}_{i-1}$ and
$\mathcal{P}_i$ are internally disjoint, then $w_i=U$, where $U$ is
the undecomposable dissection delimited by
$\mathcal{P}_{i-1}$ and $\mathcal{P}_i$, rooted at the first edge of
$\mathcal{P}_{i-1}$ and with $s$ as root-vertex, 
see Figure~\ref{figure:decomp}.
This encoding is injective, an easy consequence of
Lemma~\ref{lemma:unique}.

\begin{figure}
  \centering\psset{unit=1.1em}
  \begin{pspicture}(-5,-4)(24,4)
    \rput(0,-3.5){$t$}\rput(0,3.5){$s$}
    \rput(0,0){\includegraphics[width=12\psunit]{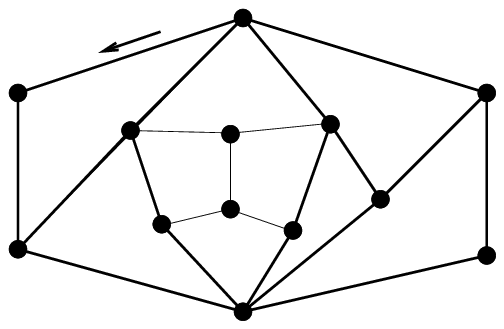}}
    \rput[l](6,0){\normalsize$\implies \; w = tsUsts,$ where $U = $} 
    \rput(21,0){\includegraphics[height=4\psunit]{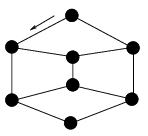}}
  \end{pspicture}
  \caption{Example of decomposition of a rooted irreducible
    dissection and of its associated decomposition word.}
  \label{figure:decomp}
\end{figure}

\paragraph{Characterization of decomposition words of elements of \D}

The fact that $D$ has no separating 4-cycle easily implies that its
decomposition word has no factor $ss$ nor $tt$, and these are the only
forbidden factors. Moreover, as a dissection has at least one inner
vertex, a decomposition word can neither be the
empty word, nor the one-letter words $s$ and~$t$, nor the two-letter words $st$
and $ts$.  It is easily seen
that all other words encode irreducible dissections of the hexagon.

This leads to the following equation linking the
generating functions $D(x)$ and $U(x)$ counting \D and~\U according to
the number of inner vertices,

\begin{equation}
  \label{eq:Dn}
  x^2D(x)+2x^2+2x+1 ~=~ \left( 1 + \frac{2x}{1-x} \right)  \cdot
  \frac{1}{1-x^2U(x)\left(1+ \frac{2x}{1-x}\right)}.
\end{equation}


Similarly, let $D(x_{\bullet},x_{\circ}):=\sum
|\Dij|x_{\bullet}^{i}x_{\circ}^{j}$ and
$U(x_{\bullet},x_{\circ}):=\sum
|\Uij|x_{\bullet}^{i}x_{\circ}^{j}$.
Then the characterization of the coding words gives
\begin{multline}\label{eq:Dij}
  x_\bullet x_\circ D(x_{\bullet},x_{\circ}) + 2x_\bullet x_\circ +x_\bullet + x_\circ + 1 \\
  =  (1+x_{\bullet}) \cdot \frac{1}{1-x_{\circ}x_{\bullet}} \cdot (1+x_{\circ}) 
  ~\cdot~
  \frac{1}{1-x_\bullet x_\circ U(x_{\bullet},x_{\circ})(1+x_{\bullet})
  \frac{1}{1-x_{\circ}x_{\bullet}}(1+x_{\circ})}. 
\end{multline}

\begin{theorem}
  Let $\P_{n}$ be the number of rooted 3-connected maps with $n$
edges and $\P_{ij}$ the number of rooted 3-connected maps with $i$
vertices and $j$ faces. Then
  \[
  \sum_n |\P_{n+2}|x^n ~=~ \frac{1-x}{1+x} ~-~ \frac{1}{1+2x+2x^2+x^2(2r(x)-r(x)^2)},
  \]
  where $r(x)=x\left( 1+r(x)\right) ^2$, and
  \begin{multline}
    \sum_{i,j} |\P_{i+2,j+2}| x_{\bullet}^i x_{\circ}^j \\ ~=~
    \frac{1-x_{\bullet}x_{\circ}}{(1+x_{\bullet})(1+x_{\circ})}
    ~-~\frac{1}{1+x_{\bullet}+x_{\circ}+2x_{\bullet}x_{\circ}+x_\bullet
      x_\circ (r_1+r_2-r_1 r_2)} ,
  \end{multline}
  where $\left\{
  \begin{array}{ccc}
    r_1(x_{\bullet},x_{\circ})&=&x_{\bullet}\left(
    1+r_2(x_{\bullet},x_{\circ})\right) ^2  \\
    r_2(x_{\bullet},x_{\circ})&=&x_{\circ}\left(
    1+r_1(x_{\bullet},x_{\circ})\right) ^2 
  \end{array}
  \right.$.
\end{theorem}

\begin{proof}
  Lemma~\ref{lemma:images} ensures that
  $\sum_n|\mathcal{P}_{n+2}'|x^n=x^2U(x)$ and, more precisely,
  $\sum_{i,j}|\mathcal{P}_{i+2,j+2}'|x_{\bullet}^ix_{\circ}^j=x_{\bullet}x_{\circ}U(x_{\bullet},x_{\circ})$.
  Moreover, Equations~(\ref{eq:Dn}) and Equation~(\ref{eq:Dij}) yield
  expressions of $x^2U(x)$ and
  $x_{\bullet}x_{\circ}U(x_{\bullet},x_{\circ})$ respectively in terms
  of $D(x)$ and $D(x_{\bullet},x_{\circ})$. In these expressions,
  replace $D(x)$ and $D(x_{\bullet},x_{\circ})$ by their respective
  expression in terms of $r$ and of $r_1$ and $r_2$, as given by
  Equations~(\ref{eq:arbres}) and~(\ref{eq:arbressub}).  \hfill $\qed$
\end{proof}

\section{Application: sampling rooted 3-connected maps}
\label{section:sampling}

\subsection{Sampling rooted 3-connected maps with $n$ edges}
\label{section:sampleaccordingtoedges}

Theorem~\ref{theorem:bijectionRooted} (first identity)
ensures that the following
algorithm samples rooted 3-connected maps with $n$
edges uniformly at random:
\begin{enumerate}
\item Sample an object $T\in\B_{n-4}$ uniformly (e.g. using parenthesis words).
\item Perform the closure of $T$ to obtain an irreducible dissection
  $D$  with $(n-4)$ vertices. Choose randomly one of the six edges of the
  hexagon of $D$ to carry the root. If $D$ is not undecomposable, then
reject and restart.
\item 
  Connect by a new edge $e$ the root-vertex of $D$ to the opposite
outer vertex. Take
  $e$ as root edge, with the same root-vertex as in $D$. This
  gives a rooted irreducible quadrangulation $Q$ with $n$ faces.
\item 
Return the rooted 3-connected map in
  $\Pn$ associated to $Q$ by the angular mapping.
\end{enumerate}

\begin{proposition}
\label{pro:randomn}
The success probability of the sampler at each trial is equal to
$|\Pn|/|\D_{n-4}|$, which satisfies
\[
\frac{|\Pn|}{|\D_{n-4}|}~\mathop{\to}_{n\to\infty}~ \frac{2^8}{3^6}.
\]
Hence, the number of rejections follows a geometric law
whose mean is asymptotically $c=3^6/2^8$.  As the
closure mapping has linear-time complexity, the sampling algorithm
has expected linear-time complexity.
\end{proposition}

\begin{proof}
According to Section~\ref{section:statementtheo},
$|\Dn|=\frac{6}{n+2}|\Bn|=\frac{6(2n)!}{(n+2)!n!}$.  Stirling
formula yields $ |\D_{n-4}|\sim \frac{3}{128 \sqrt{\pi}}
\frac{4^n}{n^{5/2}}$. Moreover, according to~\cite{Tu63},
$|\Pn|\sim \frac{2}{3^5\sqrt{\pi}}\frac{4^n}{n^{5/2}}$. This
yields the limit of $|\Pn|/|\D_{n-4}|$.
\hfill $\qed$\end{proof}

\subsection{Sampling rooted 3-connected maps with $i$ vertices and $j$ 
faces}

Similarly, Theorem~\ref{theorem:bijectionRooted} (third identity),
ensures that the following algorithm samples rooted 3-connected
maps with $i$ vertices and $j$ faces uniformly at random:
\begin{enumerate}
\item Sample an object $T\in\mathcal{B}_{i-3,j-3}^{\bullet}$ uniformly
at random. A simple method
 is described in Section~\ref{section:BicoloredTreesSample}.
\item Perform the closure of $T$ to obtain an irreducible dissection
  $D$ with $(i-3)$ inner black vertices and $(j-3)$ inner white
  vertices.  Choose randomly the root-vertex among the three black
  vertices of the hexagon. If the dissection is not undecomposable,
  then reject and restart.
\item Connect by a new edge $e$ the root-vertex of $D$ to the opposite
  outer vertex. Take $e$ as root edge, with the same root-vertex as in
  $D$. This gives a rooted irreducible quadrangulation $Q$ with $i$
  black vertices and $j$ white vertices.
\item Return the rooted 3-connected map in $\mathcal{P}_{ij}'$
  associated to $Q$ by the angular mapping.
\end{enumerate}

\begin{proposition}\label{pro:random}
  The success probability of the sampler at each trial is equal to
  $|\Pij|/|\D_{i-3,j-3}|$.  Let $\alpha \in \rbrack 1/2,2\lbrack$; if
  $i$ and $j$ are correlated by $\frac{i}{j}\to\alpha$ as
  $i\to\infty$, then
  \[
  \frac{|\Pij|}{|\D_{i-3,j-3}|} ~\sim~
  \frac{2^8}{3^6}\frac{(2-\alpha)^2(2\alpha -1)^2}{\alpha ^2} =:
  \frac{1}{c_\alpha}.  
  \]
  Hence, when $\frac{i}{j}\to \alpha$, the number
  of rejections follows a geometric law whose mean is asymptotically
  $c_{\alpha}$. 
  Under these conditions, the sampling algorithm has
  an expected linear-time complexity, the linearity factor being
  asymptotically proportional to $c_{\alpha}$.

  Moreover, in the worst case of triangulations where $j=2i-4$, the
  mean number of rejections is quadratic, so that the sampling complexity is
  cubic.
\end{proposition}

\begin{proof}
  These asymptotic results are easy consequences of the expression of
  $|\Dij|$ obtained in Corollary~\ref{corollary:counting} and of the
  asymptotic result  ${|\Pij|\sim \frac1{3^5 2^2
  ij}\binom{2i-2}{j+2}\binom{2j-2}{i+2}}$ given in~\cite{B87}.
\hfill $\qed$\end{proof}

\section{Application: coding 3-connected maps}
\label{section:codingAlgo}

This section introduces an algorithm, derived from the inverse of the
closure mapping, to encode a 3-connected map. Precisely, the algorithm
encodes an outer-triangular 3-connected map, but it is then easily
extended to encode any 3-connected map. Indeed, if the outer face of
$G$ is not triangular, fix three consecutive vertices $v$, $v'$ and
$v''$ incident to the outer face of $G$ and link $v$ and $v''$ by an
edge to obtain an outer-triangular 3-connected planar map
$\widetilde{G}$; the coding of $G$ is obtained as the coding of
$\widetilde{G}$ plus one bit indicating if an edge-addition has been
done.

\begin{figure}
  \def\etalon{.4\linewidth}
  \centering
  \subfigure[]{\includegraphics[width=\etalon]{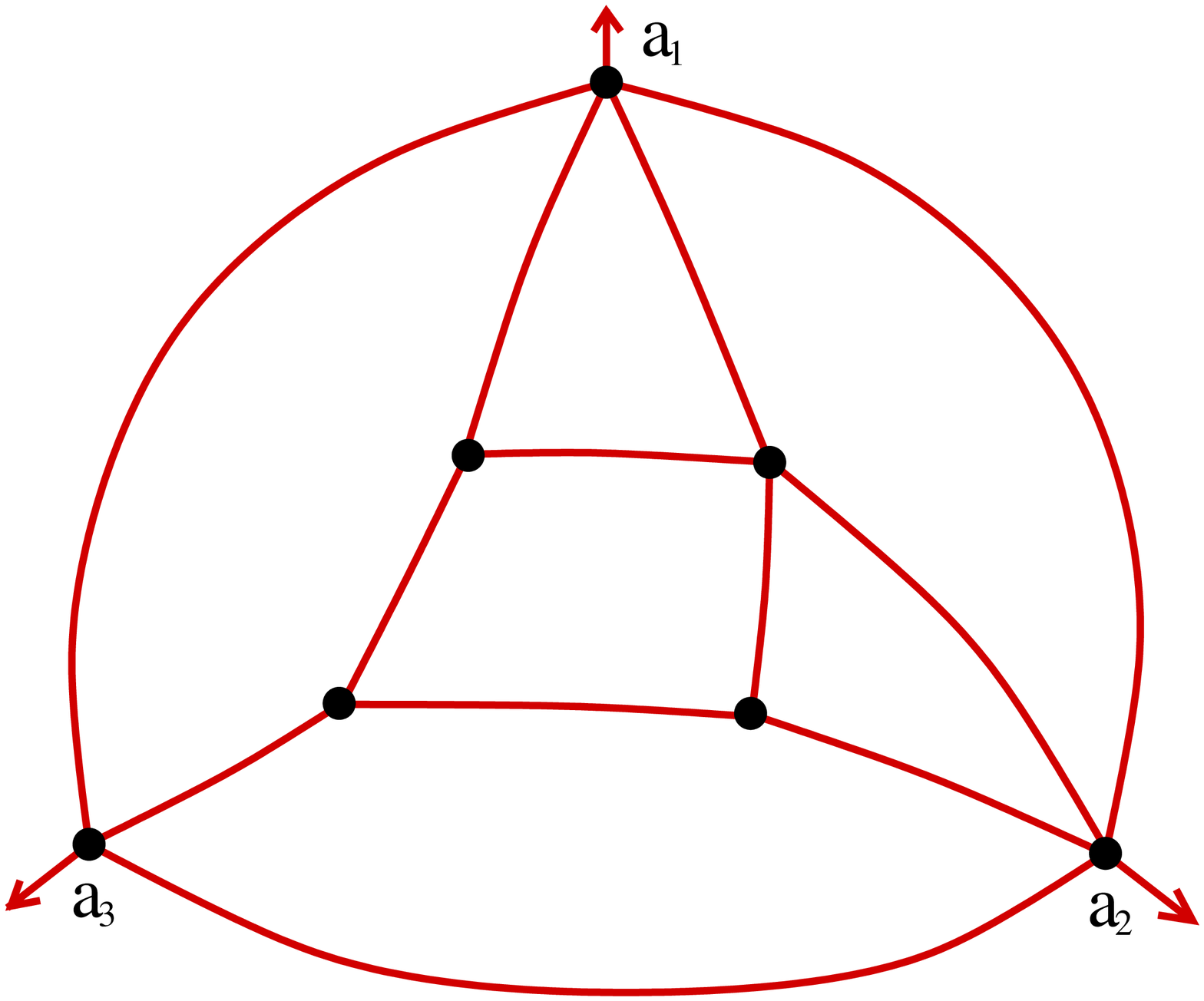}}\qquad\qquad
  \subfigure[]{\includegraphics[width=\etalon]{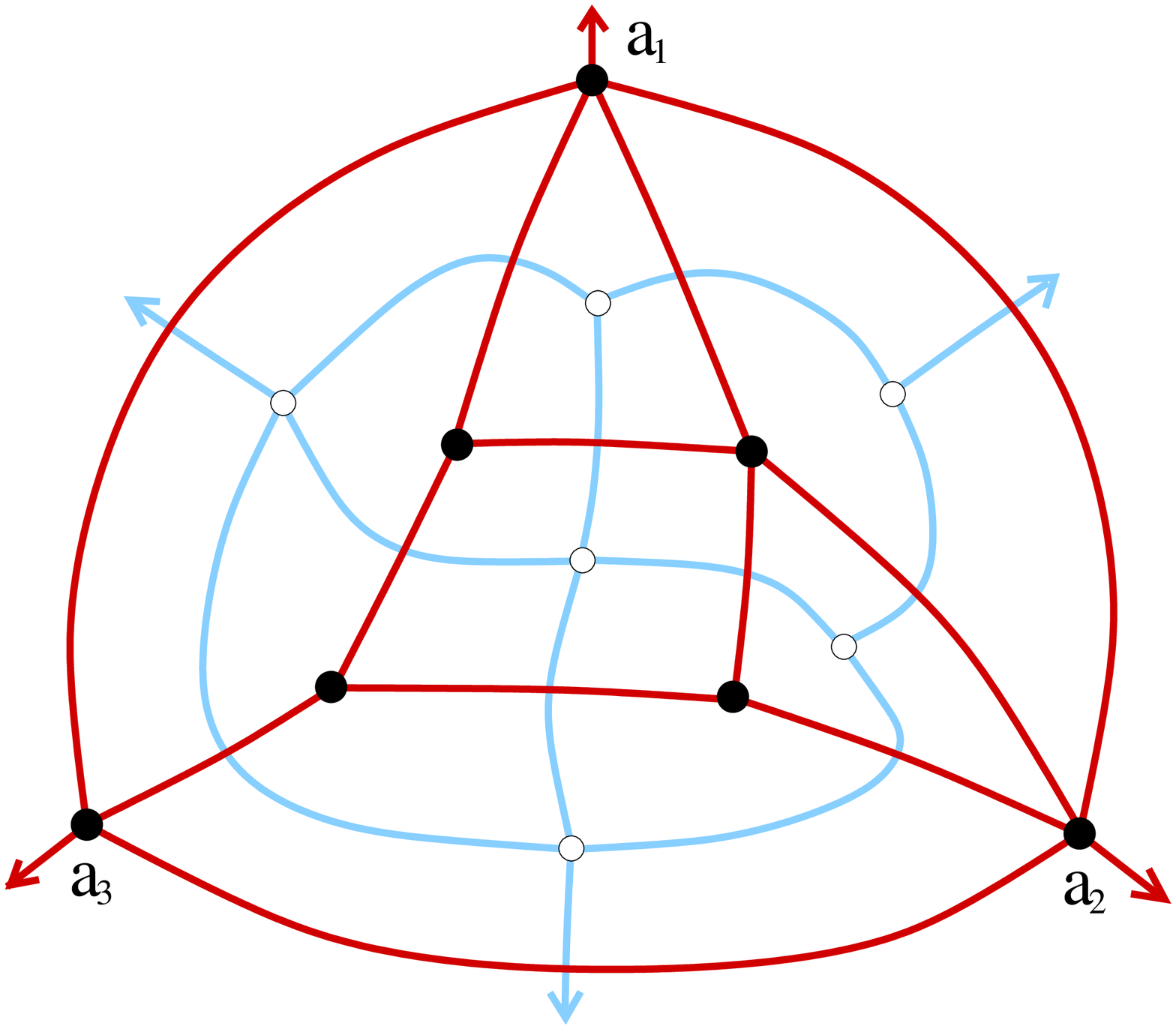}}\qquad\qquad
  \subfigure[]{\includegraphics[width=\etalon]{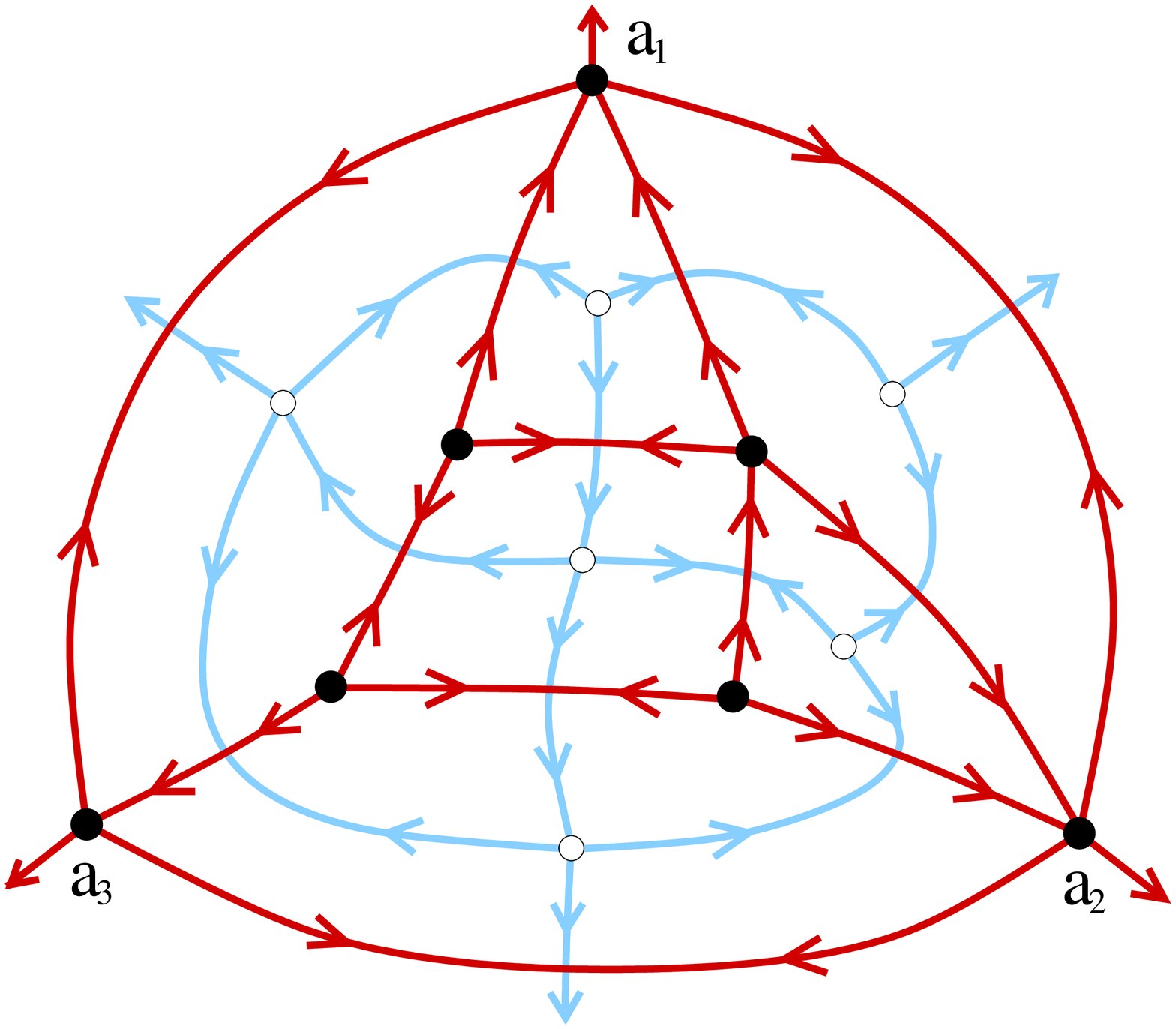}}\qquad\qquad
  \subfigure[]{\includegraphics[width=\etalon]{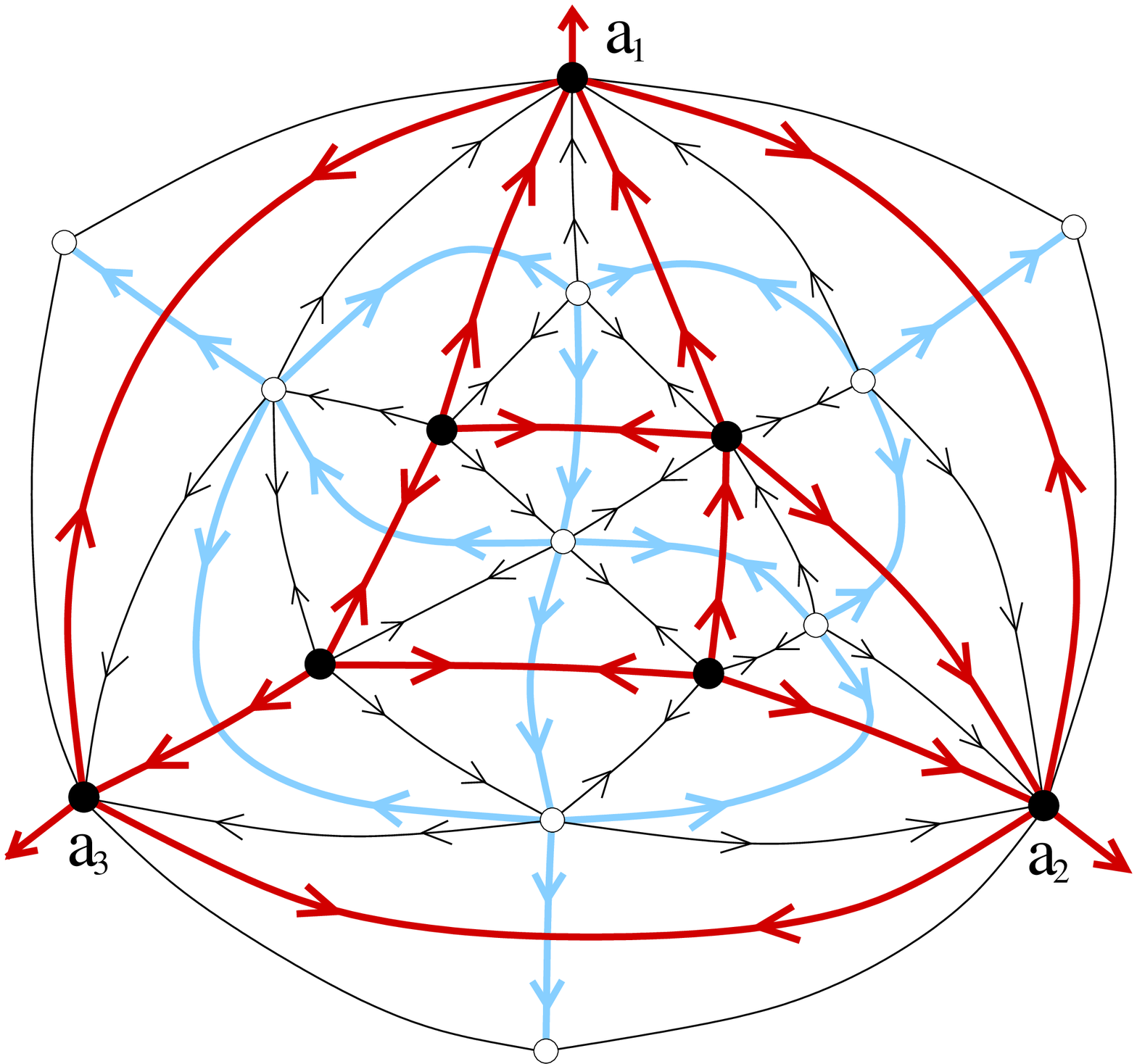}}\qquad\qquad
  \subfigure[]{\includegraphics[width=\etalon]{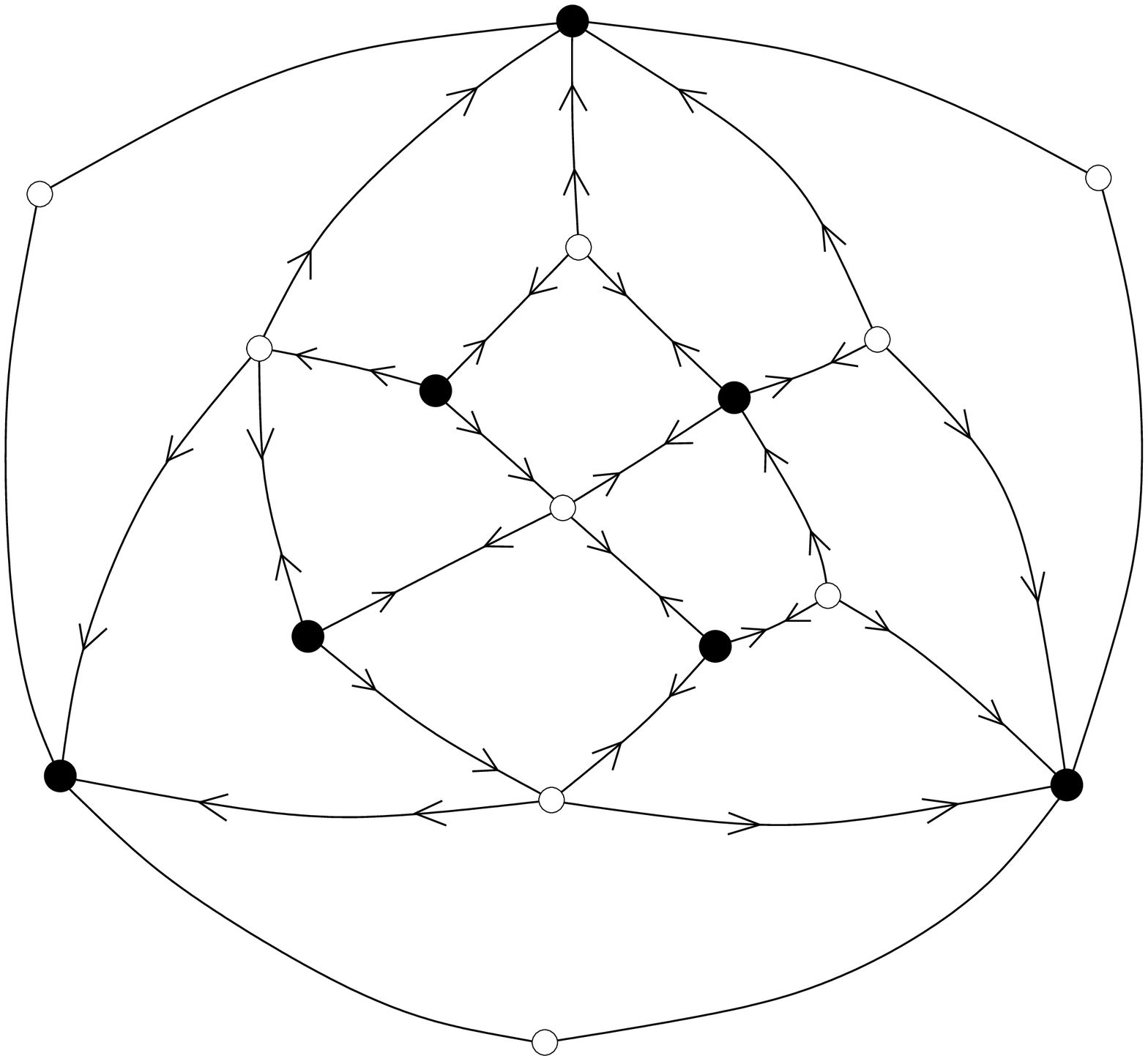}}\qquad\qquad
  \subfigure[]{\includegraphics[width=\etalon]{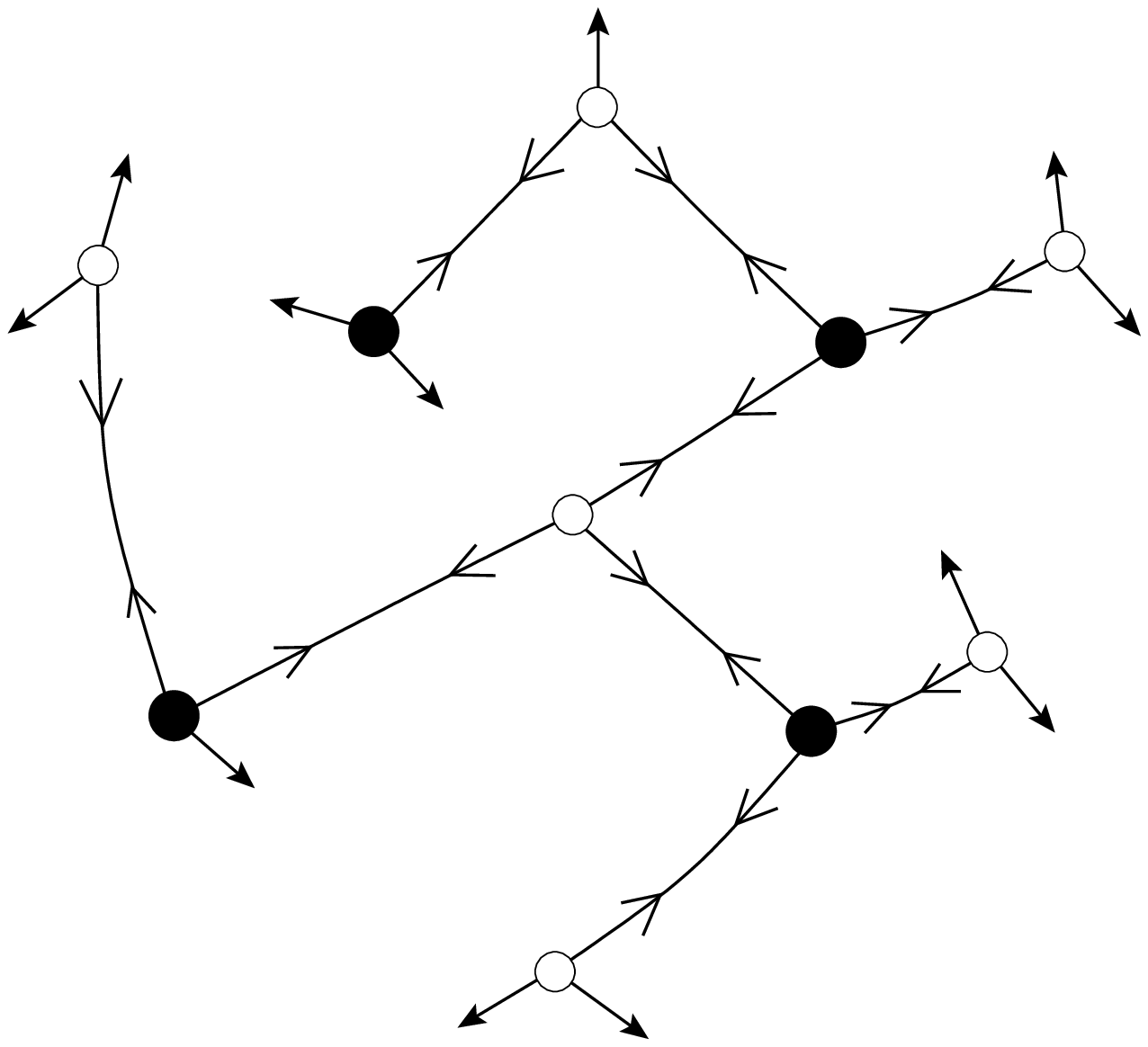}}
  \caption{Execution of the encoding algorithm on an example.}
  \label{fig:Coding}
\end{figure}

\subsection{Description of the coding algorithm}
Let $G$ be an outer-triangular 3-connected map and let $G'$ be its
derived map, as defined in Section~\ref{section:bijectionViaHexagon}. 
The coding algorithm relies on the following steps,
illustrated in Figure~\ref{fig:Coding}. 

\subsubsection{Compute a particular orientation of the derived map $G'$ (Fig.~\ref{fig:Coding}(b)-(c))}
The first step of the algorithm is to compute a specific orientation
$X_0$ of the edges of the derived map $G'$, such that $X_0$ has no
clockwise circuit, each primal or dual vertex has outdegree~3 and each
edge-vertex has outdegree~1. Such an orientation of $G'$ exists and is
unique, as we will see in Theorem~\ref{theorem:felsner}.  A linear
time algorithm to compute $X_0$ is given in
Section~\ref{section:compute}.

\subsubsection{Compute the irreducible dissection $D$ associated to $G$ (Fig.~\ref{fig:Coding}(d))}
Consider the bicolored complete irreducible dissection $D$ associated
to $G$ by the bijection presented in
Section~\ref{section:bijectionViaHexagon} (and reformulated in
Section~\ref{section:derivatedmap}), i.e., the dissection
 having the same derived map as $G$. Notice
 that $D$ has $n$ inner faces if $G$
has $n$ edges. Hence, according to Euler's relation, $D$ has
$n-2$ inner vertices. Similarly, if $G$ has $i$ vertices and $j$ inner faces,
then $D$ has $i$ black vertices and $j+3$ white vertices.

\subsubsection{Compute the tri-orientation of $D$ without clockwise circuit (Fig.~\ref{fig:Coding}(d))}
We orient each half-edge $h$ of $D$ belonging to an inner edge as follows:
$h$ is directed 
inward if its incident vertex belongs to the hexagon;
otherwise, $h$ receives the orientation of the cw-following edge of $G'$. 
As shown in
Section~\ref{section:proof} (more precisely in
Lemma~\ref{lemma:derivatedtodis}, composed with the correspondence of
Figure~\ref{figure:completedOrientations}), this process yields
the unique tri-orientation of $D$ without clockwise circuit.

\subsubsection{Open the dissection $D$ into a binary tree $T$ (Fig.~\ref{fig:Coding}(f))}
Once the tri-orientation without clockwise circuit is computed, $D$ is
opened into a binary tree $T$, by deleting outer vertices, outer edges,
and all ingoing half-edges (see Section~\ref{section:opening}).

\subsubsection{Encode the tree $T$}
First, choose an arbitrary leaf of $T$, root $T$ at this leaf, and
encode the obtained rooted binary tree using a parenthesis word 
(also called \emph{Dyck word}). The opening of a 3-connected
map with $n$ edges is a binary tree with $n-2$ inner
nodes, yielding an encoding Dyck word of length $2(n-2)$.

Similarly, the opening of a 3-connected map with $i$ vertices and $j$
inner faces is a black-rooted bicolored binary tree
with $i-3$ black nodes and $j$ white nodes. A black-rooted 
bicolored binary trees with a given number of black and
white nodes is encoded by a pair of words, as explained in
Section~\ref{section:BicoloredTreesCode}. Then the two words
can be asymptotically optimally encoded in linear time, 
according to~\cite[Lem.7]{BGH03}. 

\begin{theorem}
  The coding algorithm has linear-time complexity and is
  asymptotically optimal: the number of bits per edge of the code of
  a map in \Pn (resp. in \Pij) is asymptotically equal to the binary 
  entropy per edge, defined as  
  $\frac1{n}\log_2 (|\Pn|)$ (resp. $\frac1{i+j-2}\log_2(|\Pij|)$).
\end{theorem}

\begin{proof}
  It is clear that the encoding algorithm has linear-time
  complexity, provided the algorithm computing the constrained
  orientation without clockwise circuit of the derived map has
  linear-time complexity (which will be proved in
  Section~\ref{section:compute} and
  Section~\ref{section:proofCorrect}).
   
  According to Corollary~\ref{corollary:counting},
  Proposition~\ref{pro:randomn} and~\ref{pro:random},
  $|\Bn|/|\Pn|$ and $|\cBijb|/|\Pij|$ are bounded by
  fixed polynomials.  Hence, the entropy per edge of \Bn and \Pn are asymptotically equal,
  and the binary entropy per edge of \cBijb and \Pij are asymptotically equal. As the encoding of
  objects of \Bn (\cBijb) using parenthesis words is
  asymptotically optimal, the encoding of objects of \Pn
  (\Pij, respectively) is also asymptotically optimal.
\hfill $\qed$\end{proof}



\section{Proof of Theorem~\ref{THEOREM:UNIQUENESSEXISTENCE}}
\label{section:proof}
This section is devoted to the proof of 
Theorem~\ref{theorem:uniquenessexistence}, which states that each 
irreducible dissection has a unique tri-orientation without clockwise circuit. 

\subsection{$\alpha$-orientations and outline of the proof}
\label{section:alphaorientation}


\paragraph{Definition}
Let $G=(V,E)$ be a planar map. Consider a function
 $\alpha:V\rightarrow \mathbb{N}$. 
An $\alpha$-orientation of $G$ is an orientation of the edges 
of $G$ such that the outdegree of each vertex $v$ of $G$ is $\alpha(v)$. 
If an $\alpha$-orientation exists, then the function $\alpha$ is said to be 
\emph{feasible} for $G$.

\paragraph{Existence and uniqueness of $\alpha$-orientations}
The following results are proved in \cite{Fe03} (the first point had
already been proved in~\cite{Oss}):

\begin{theorem}[(\cite{Fe03})] \label{theorem:felsner}
Given a planar map $G$ and a feasible function $\alpha$, there exists
a unique $\alpha$-orientation of $G$ without clockwise circuit.  This
$\alpha$-orientation is called the minimal~\footnote{The term minimal refers
to the fact that the set of all $\alpha$-orientations
of $G$ forms a distributive lattice, the ``flip'' operation being a circuit
reversion.} $\alpha$-orientation of $G$.

Given the derived map of an outer-triangular  3-connected planar map, 
the function $\alpha_0$ such that $\alpha_0(v)=3$ for all primal and dual 
vertices and $\alpha_0(v)=1$ for all edge-vertices is a feasible function. 
\end{theorem}

Theorem~\ref{theorem:felsner} ensures uniqueness of
the orientation without clockwise circuit of a graph
with prescribed outdegree for  each vertex. However, this property does
not  directly imply uniqueness in
Theorem~\ref{theorem:uniquenessexistence}, because a
tri-orientation has \emph{bi-oriented} edges.

To use Theorem~\ref{theorem:felsner}, we work with the derived map
$G'$ of an irreducible dissection~$D$, as defined in
Section~\ref{section:derivatedmap}. We have defined derived maps only
for a subset of irreducible dissections, namely for bicolored complete
irreducible dissections (recall that these are bicolored dissections
such that the 3 outer white vertices have degree~2). As a consequence,
a first step toward proving Theorem~\ref{theorem:uniquenessexistence}
is to reduce its proof to the proof of existence and uniqueness of a
so-called complete-tri-orientation (a slight adaptation of the
definition of tri-orientation) without clockwise circuit for any
bicolored complete irreducible dissection.

We prove that a complete-tri-orientation without clockwise circuit of
a bicolored complete irreducible dissection $D$ is transposed
injectively into an $\alpha_0$-orientation without clockwise circuit
of its derived map $G'$. By injectivity and by uniqueness of the
$\alpha_0$-orientation without clockwise circuit of $G'$, this implies
uniqueness of a tri-orientation without clockwise circuit for~$D$.

The final step will be to prove that an $\alpha_0$-orientation without
clockwise circuit of $G'$ is transposed into a
complete-tri-orientation without clockwise circuit of $D$. By
existence of an $\alpha_0$-orientation without clockwise circuit for
$G'$ (Theorem~\ref{theorem:felsner}), this implies the existence of a
complete-tri-orientation without clockwise circuit of~$D$.

\subsection{Reduction to the case of bicolored complete dissections}
\label{section:complete}

\paragraph{Introduction}
The aim of this section is to reduce the proof of
Theorem~\ref{theorem:uniquenessexistence} to the class of complete 
bicolored irreducible dissections. We state the following proposition where 
the term ``complete-tri-orientation'', to be defined later, is a slight 
adaptation of the notion of tri-orientation.

\begin{proposition}\label{proposition:complete}
  The existence and uniqueness of a complete-tri-orientation without
  clockwise circuit for any bicolored complete irreducible dissection
  implies the existence and uniqueness of a tri-orientation without
  clockwise circuit for any irreducible dissection, i.e., implies
  Theorem~\ref{theorem:uniquenessexistence}.
\end{proposition}

The rest of this subsection is devoted 
to the proof of Proposition~\ref{proposition:complete}. The proof
is done in two steps. First, reduce the proof of 
Theorem~\ref{theorem:uniquenessexistence} to the existence
and uniqueness of a \emph{tri-orientation} without clockwise circuit 
for any bicolored complete irreducible dissection. Then, prove that
this reduces to the existence and uniqueness of a 
\emph{complete-tri-orientation}
without clockwise circuit for any bicolored complete irreducible dissection.

\paragraph{Completion of a bicolored irreducible dissection}
For any bicolored irreducible dissection~$D$, we define its \emph{completed
dissection} $D^c$ as follows .  For each white vertex $v$ of
the hexagon, we denote by $\eleft(v)$ ($\eright(v)$) 
the outer edge starting from $v$ with the interior of the hexagon on the left
(right, respectively) and denote by $l(v)$ and $r(v)$ the neighbours of $v$
 incident to $\eleft(v)$ and to $\eright(v)$. We perform the following
operation: if $v$ has degree at least 3, a new white vertex
$v'$ is created outside of the hexagon and
 is linked to $l(v)$ and to $r(v)$ by
two new edges $\eleft(v')$ and $\eright(v')$, see Figure~\ref{figure:TriOri}.
The vertex $v'$ is said to \emph{cover} the vertex $v$.

The dissection obtained  is a bicolored dissection of the
hexagon such that the three white vertices of the hexagon have two incident
edges, see the transition between
Figure~\ref{figure:completedOrientations}(a) and
Figure~\ref{figure:completedOrientations}(b)  (ignore here the
orientation of edges).

\begin{lemma}
  \label{lemma:completedComplete}
  The completion $D^c$ of a bicolored irreducible dissection $D$ is a bicolored 
  complete irreducible dissection.
\end{lemma}

\begin{proof}
  The outer white vertices of $D^c$ have degree 2 by construction.
  Hence, we just have to prove that $D^c$ is irreducible. As $D$ is
  irreducible, if a separating 4-cycle $\cC$ appears in $D^c$ when the
  completion is performed, then it must contain a white vertex $v'$ of
  the hexagon of $D^c$ added during the completion, so as to cover an
  outer white vertex $v$ of degree greater than~2. Two edges of $\cC$
  are the edges $\eleft(v')$ and $\eright(v')$ incident to $v'$ in
  $D^c$. The two other edges $\epsilon_1$ and $\epsilon_2$ of $\cC$
  form a path of length 2 connecting the vertices $l(v)$ and $r(v)$
  and passing by the interior of $D$ (otherwise, $\cC$ would enclose a
  face). As $D$ is irreducible, the 4-cycle $\mathcal{C}'$ of $D$
  consisting of the edges $\eleft(v)$, $\eright(v)$, $\epsilon_1$ and
  $\epsilon_2$ delimits a face.  Hence $\eleft(v)$ and $\eright(v)$
  are incident to the same inner face of~$D$, which implies that $v$
  has degree~2, a contradiction.  
\hfill $\qed$\end{proof}

\paragraph{Tri-orientations}
Let $D$ be a bicolored irreducible dissection and let $D^c$ be its
completed bicolored dissection. We define a mapping  $\Phi$
from the tri-orientations of $D^c$ to the tri-orientations of
$D$. Given a tri-orientation $Y$ of $D^c$, we remove the edges that
have been added to obtain $D^c$ from~$D$, erase the orientation of
the edges of the hexagon of~$D$, and orient inward all inner half-edges
incident to an outer vertex of~$D$. We obtain
thus a tri-orientation $\Phi(Y)$ of~$D$, see the
transition between Figure~\ref{figure:completedOrientations}(b)
 and Figure~\ref{figure:completedOrientations}(a).

\begin{lemma}\label{lemma:triCompleted}
  Let $Y$ be a tri-orientation of $D^c$ without clockwise
  circuit. Then the tri-orientation $\Phi(Y)$ of $D$ has no
  clockwise circuit.

  For each tri-orientation $X$ of $D$ without clockwise circuit, there
  exists a tri-orientation $Y$ of $D^c$ without clockwise
  circuit such that $\Phi(Y)=X$.
\end{lemma}

\begin{proof}
  The first point is trivial, as the tri-orientation $\Phi(Y)$ is just
  obtained by removing some edges and some orientations of half-edges.

  For the second point, the preimage $Y$ is constructed as
  follows. Consider each white vertex $v$ of the hexagon of $D$ which
  has degree at least 3. Let $(h_1,\ldots, h_m)$ ($m\geq 3)$ be the
  series of half-edges incident to $v$ in $D$ in counter-clockwise
  order around $v$, with $h_1$ and $h_2$ belonging respectively to the
  edges $\eright(v)$ and $\eleft(v)$.  As $m\geq 3$, the vertex $v$
  gives rise to a covering vertex $v'$ with two incident edges
  $\eleft(v')$ and $\eright(v')$ such that the edges $\eleft(v)$,
  $\eright(v)$, $\eleft(v')$ and $\eright(v')$ form a new face
  $f$. The edges $\eleft(v)$ and $\eright(v)$ become inner edges of
  $D^c$ when $v'$ is added, and have thus to be directed.

  We orient the two half-edges of $\eleft(v)$ and $\eright(v)$
  respectively toward $l(v)$ and toward $r(v)$, see
  Figure~\ref{figure:TriOri}. The vertex $v$ receives thus two
  outgoing half-edges, and we have to give to $v$ a third outgoing
  half-edge. The suitable choice to avoid the appearance of a
  clockwise circuit is to orient $h_3$ outward, see
  Figure~\ref{figure:TriOri}. Indeed, assume a contrario that a simple
  clockwise circuit $\mathcal{C}$ is created. Then the circuit must
  pass by $v$. It goes into $v$ using one of the half-edges $h_{i}$
  directed toward $v$, i.e., $i\geq 4$. Moreover, it must go out of
  $v$ using the half-edge $h_3$ (indeed, if the circuit uses $h_1$ or
  $h_2$ to go out of $v$, then it reaches an outer vertex,
  which has outdegree 0).  Hence, the interior of the
  clockwise circuit $\mathcal{C}$ must contain all faces incident to
  $v$ that are on the right of $v$ when we traverse $v$ from $h_{i}$
  and go out using $h_3$. Hence, the interior of $\mathcal{C}$ must
  contain the new face $f$ of $D^c$, see
  Figure~\ref{figure:TriOri}. But $f$ is incident to outer edges
of $D^c$, hence the clockwise 
circuit $\mathcal{C}$ must pass
  by outer edges of $D^c$, which are not oriented, 
a contradiction. 
Thus, we have constructed
  a tri-orientation $Y$ of $D^c$ without clockwise circuit and such
  that $\Phi(Y)=X$. An example of this construction can be seen as the
  transition between Figure~\ref{figure:completedOrientations}(a) and
  Figure~\ref{figure:completedOrientations}(b).  
\hfill$\qed$
\end{proof}

\begin{figure}
  \centering
  \includegraphics[width=.8\linewidth]{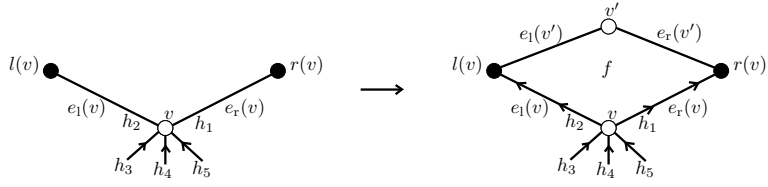}
  \caption{From a tri-orientation $X$ of $D$ without clockwise circuit,
    construction of a tri-orientation $Y$ of $D^c$
    without clockwise circuit such that $\Phi(Y)=X$.}
  \label{figure:TriOri}
\end{figure}

\begin{figure}
  \def\etalon{.3\linewidth}
  \centering
  \subfigure[]{\includegraphics[width=\etalon]{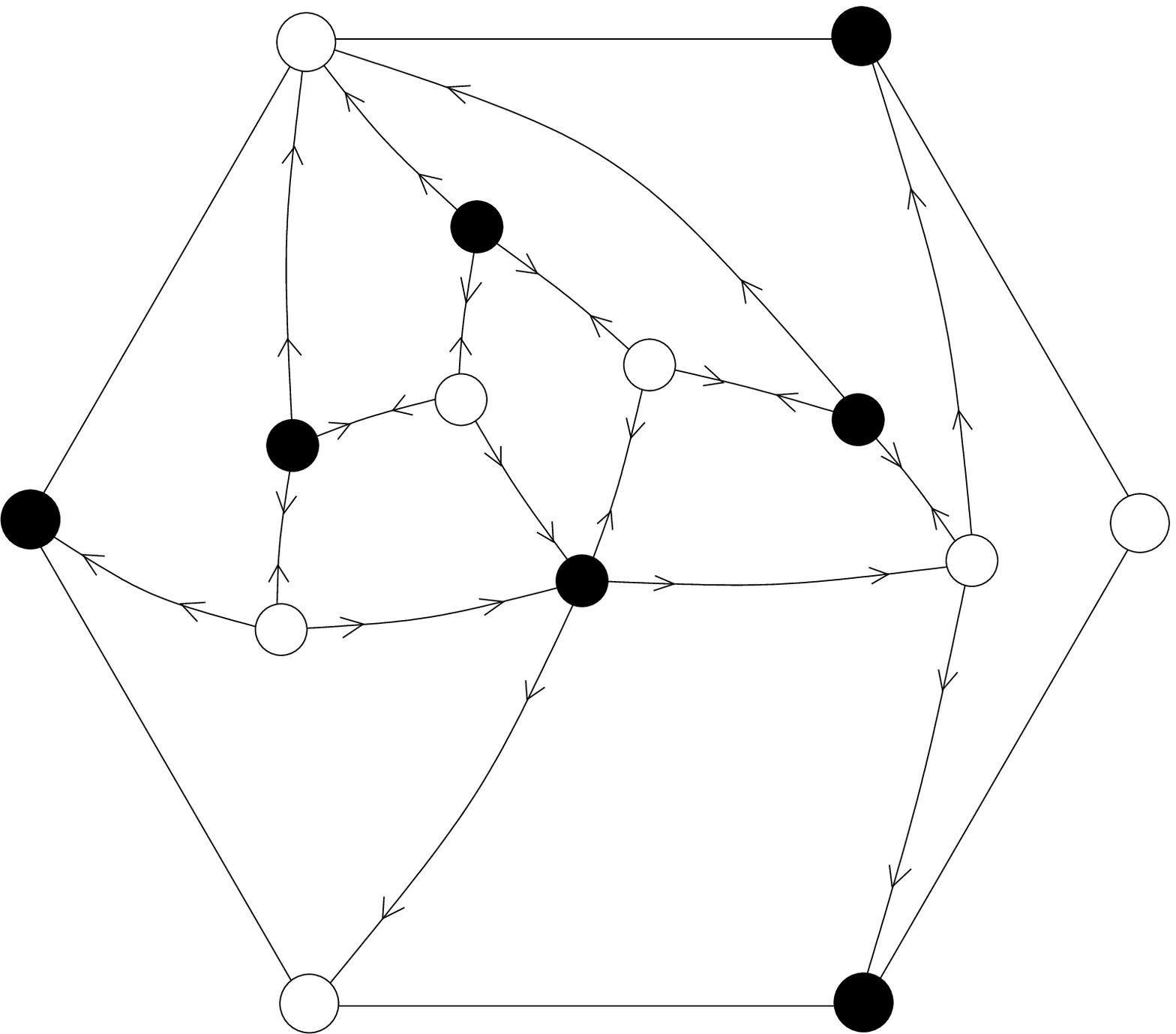}}\qquad
  \subfigure[]{\includegraphics[width=\etalon]{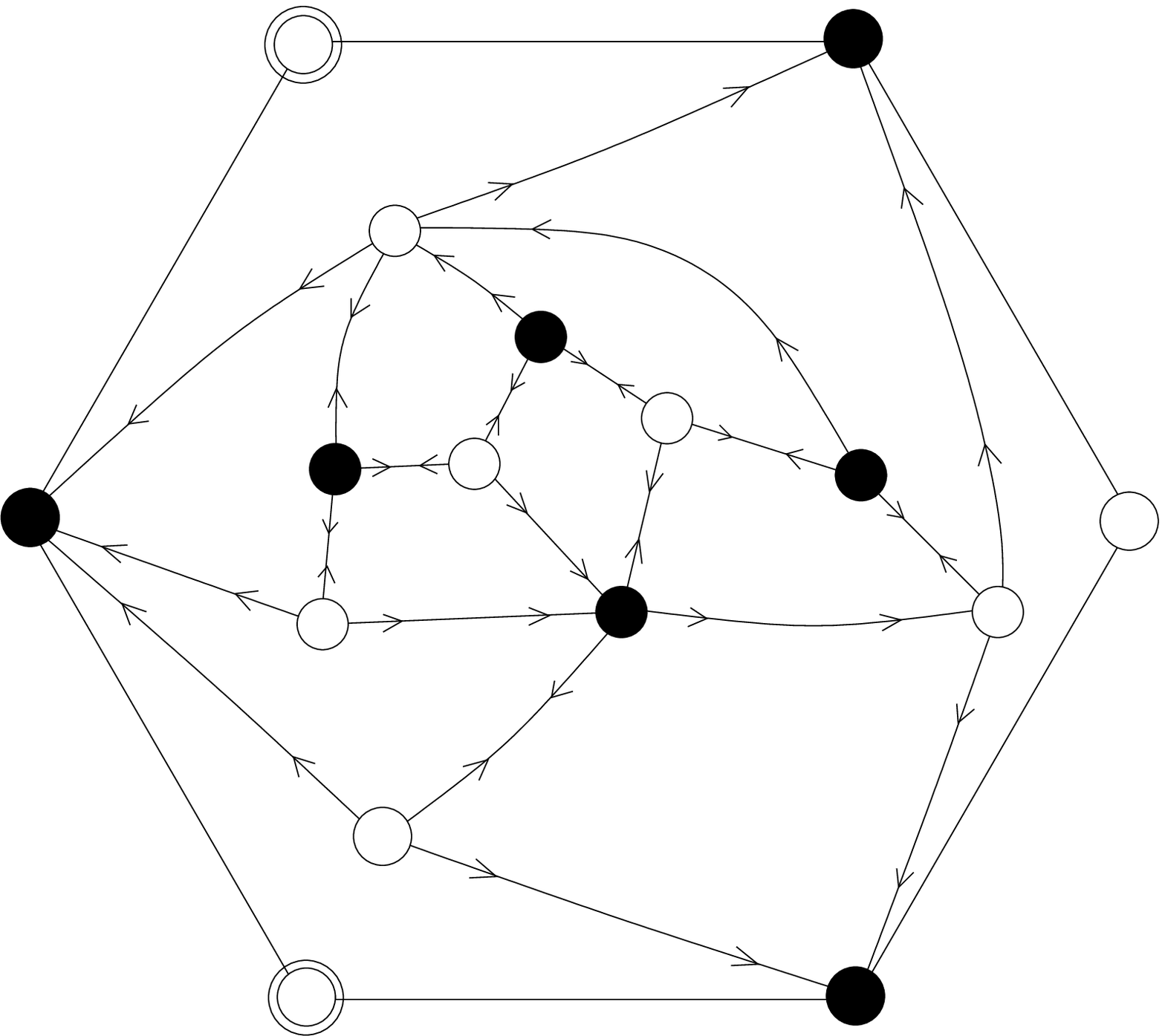}}\qquad
  \subfigure[]{\includegraphics[width=\etalon]{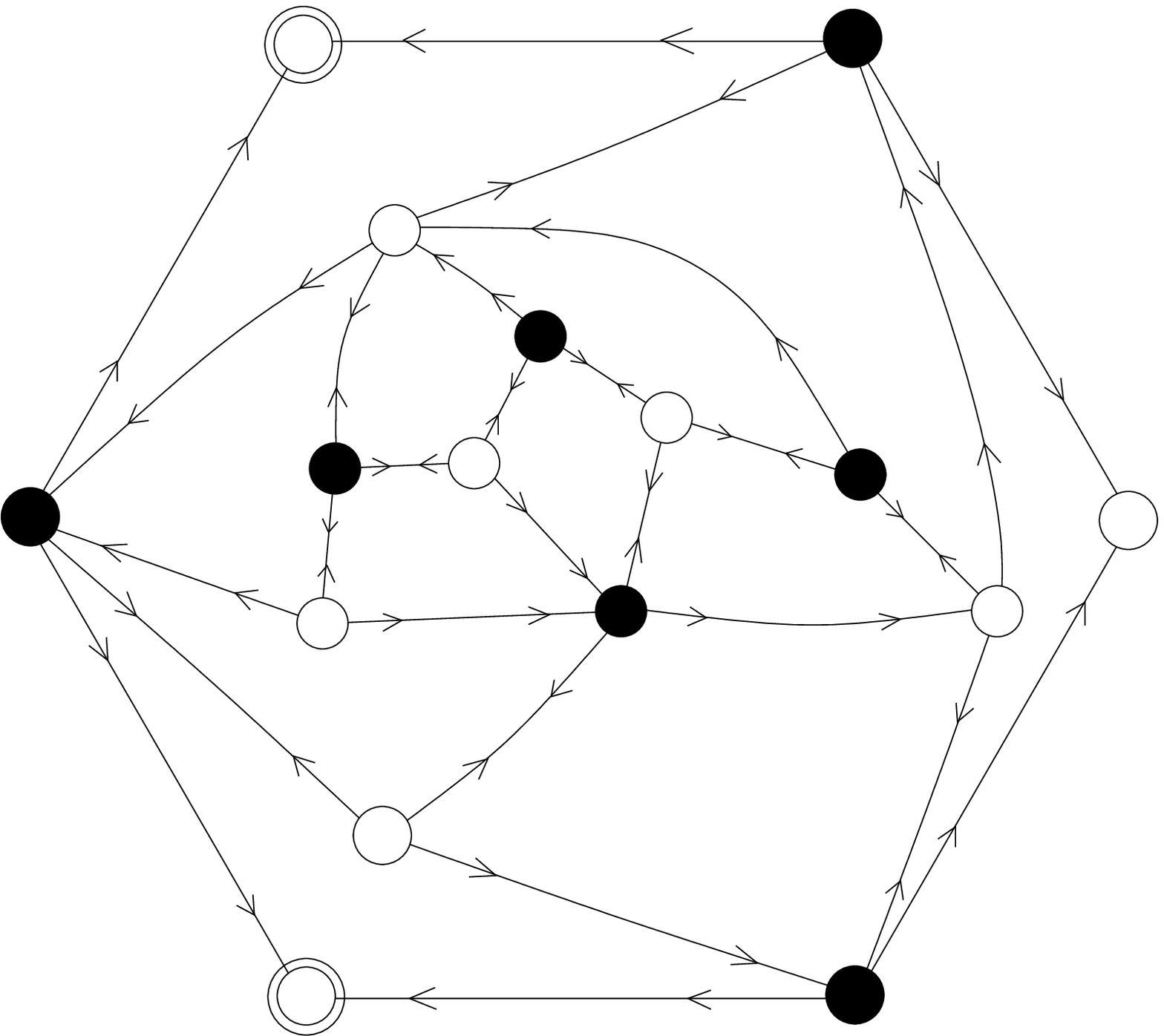}}
  \caption{A bicolored irreducible dissection $D$ endowed with a 
    tri-orientation $X$
      without clockwise circuit (Figure a). The
      associated completed dissection $D^c$ (the two added white vertices
      are surrounded) endowed with the tri-orientation $Y$ such that $\Phi(Y)=X$
      (Figure b). The dissection $D^c$ endowed with the 
      complete-tri-orientation $Z$ such that $\Psi(Z)=Y$ (Figure c).}
  \label{figure:completedOrientations}
\end{figure}

\begin{lemma} \label{proposition:existunique}
  The existence and uniqueness of a tri-orientation without clockwise
  circuit for any bicolored complete irreducible dissection implies
  the existence and uniqueness of a tri-orientation without clockwise
  circuit for any irreducible dissection, i.e., implies
  Theorem~\ref{theorem:uniquenessexistence}.
\end{lemma}

\begin{proof}
  This is a clear consequence of
  Lemma~\ref{lemma:completedComplete} and
  Lemma~\ref{lemma:triCompleted}.
  \hfill $\qed$
\end{proof}

\paragraph{Complete-tri-orientations}
A \emph{complete-tri-orientation} of a bicolored complete irreducible
dissection $D$ is an orientation of the half-edges of $D$ that
satisfies the following conditions (very similar to the conditions of
a tri-orientation): all black vertices and all inner white vertices of
$D$ have outdegree 3, the three white vertices of the hexagon have
outdegree 0, and the two half-edges of an edge of $D$ can not both be
oriented inward. The difference with the definition
of tri-orientation is that the half-edges of the hexagon are oriented,
with prescribed outdegree for the outer vertices. Similarly 
as in a tri-orientation,
edges of $D$ are distinguished into simply-oriented edges and
bi-oriented edges. 

\begin{lemma}
  \label{lemma:completeTriGraph}
  Let $D\in \mathcal{D}_n$ be a bicolored complete irreducible
  dissection endowed with a complete-tri-orientation without clockwise
  circuit. Then the subgraph $T$ of $D$ consisting of the bi-oriented
  edges of $D$ is a tree incident to all vertices of $D$ except the
  three outer white vertices.
\end{lemma}

\begin{proof}
  We reason similarly as in
  Lemma~\ref{proposition:triOrientationDissection}. Let $r$ and $s$ be
  the numbers of bi-oriented and simply oriented edges of~$D$. From
  Euler's relation (using the degrees of the faces of~$D$), $D$ has
  $2n+7$ edges, i.e., $r+s=2n+7$. In addition, the $n$ inner vertices
  and the three black (resp. white) vertices of the hexagon of $D$
  have outdegree 3 (resp. 0). Hence, $2r+s=3(n+3)$. Thus, $r=n+2$ and
  $s=n+5$. Hence, the subgraph $T$ has $n+2$ edges, has no cycle
  (otherwise, a clockwise circuit of $D$ would exist), and is incident
  to at most $(n+3)$ vertices, which are the inner vertices and the
  three outer black vertices of~$D$. A classical result of graph
  theory ensures that $T$ is a tree spanning these $(n+3)$ vertices.
  \hfill $\qed$
\end{proof}

\begin{lemma}
  \label{lemma:thirdbioriented}
  Let $D\in \mathcal{D}_n$ be a bicolored complete irreducible
  dissection endowed with a complete-tri-orientation $Z$ without
  clockwise circuit.  Then, for each outer black vertex $v$ of~$D$,
  the unique outgoing inner half-edge incident to $v$ belongs to a bi-oriented
  edge.
\end{lemma}

\begin{proof}
  The subgraph $T$ consisting of the bi-oriented edges of $D$ is a
  tree spanning all vertices of $D$ except the three outer white
  vertices. Hence, there is a bi-oriented edge $e$ incident to each
  black vertex $v$ of the hexagon and this edge consitutes the third
  outgoing edge of $v$.  
  \hfill $\qed$
\end{proof}

Let $D$ be a bicolored complete irreducible dissection and $Z$ be a
complete-tri-orientation of $D$ without clockwise circuit.  We
associate to $Z$ a tri-orientation $\Psi(Z)$ as follows: erase the
orientation of the edges of the hexagon of~$D$; for each black vertex
$v$ of the hexagon, change the orientation of the unique outgoing
inner half-edge $h$ of $v$.  According to
Lemma~\ref{lemma:thirdbioriented}, $h$ belongs to a bi-oriented edge
$e$, so that the change of orientation of $h$ turns $e$ into an edge
simply oriented toward $v$. Thus, the obtained orientation $\Psi(Z)$
is a tri-orientation.

\begin{lemma}
  \label{lemma:tricomplete}
  Let $D$ be a bicolored complete irreducible dissection.
  Let $Z$ be a complete-tri-orientation of $D$ without clockwise
  circuit. Then the tri-orientation $\Psi(Z)$ of $D$ has no
  clockwise circuit.

  For each tri-orientation $Y$ of $D$ without clockwise circuit, there
  exists a complete-tri-orientation $Z$ of $D$ without clockwise
  circuit such that $\Psi(Z)=Y$.
\end{lemma}

\begin{proof}
  The first point is trivial. For the second point, we reason
  similarly as in Lemma~\ref{lemma:triCompleted}. For each black
  vertex $v$ of the hexagon of~$D$, let $(h_1,\ldots, h_m)$ ($m\geq
  3)$ be the sequence of half-edges of $D$ incident to $v$ in
  counter-clockwise order around $v$, with $h_1$ and $h_2$ belonging
  to the two outer edges $\eright(v)$ and $\eleft(v)$ of $D$
  that are incident to $v$. To construct the preimage $Z$ of $Y$, we
  make the edges $\eleft(v)$ and $\eright(v)$ simply oriented toward
  their incident white vertex. The third outgoing half-edge is chosen
  to be $h_3$, which is the ``leftmost'' inner half-edge of $v$. An
  argument similar as in the proof of the second point of
  Lemma~\ref{lemma:triCompleted} ensures that this choice is judicious
  to avoid the creation of a clockwise circuit. An example of this
  construction is shown in
  Figure~\ref{figure:completedOrientations}(b)-(c).  
\hfill $\qed$
\end{proof}

Finally, Proposition~\ref{proposition:complete} follows directly from
Lemma~\ref{proposition:existunique} and Lemma~\ref{lemma:tricomplete}.

Proposition~\ref{proposition:existunique} reduces the proof of
Theorem~\ref{theorem:uniquenessexistence} to proving the existence and
uniqueness of a complete-tri-orientation without cw circuit for any
bicolored complete irreducible dissection.  From now on, we will work
with these dissections.

\subsection{Transposition rules for orientations}
\label{section:derivatedMapOrientation}
Let $D$ be a bicolored complete irreducible dissection and let $G'$ be
the derived map of $D$.  We associate to a complete-tri-orientation of
$D$ an orientation of the edges of $G'$ of $D$ as follows, see
Figure~\ref{figure:localRules}: each edge $e=(v,v')$ ---with $v$ the
primal/dual vertex and $v'$ the edge-vertex--- receives the direction
of the half-edge of $D$ following $e$ in ccw order around $v$.

\begin{figure}
  \def\etalon{.3\linewidth}
  \centering
  \subfigure[]{\includegraphics[height=\etalon]{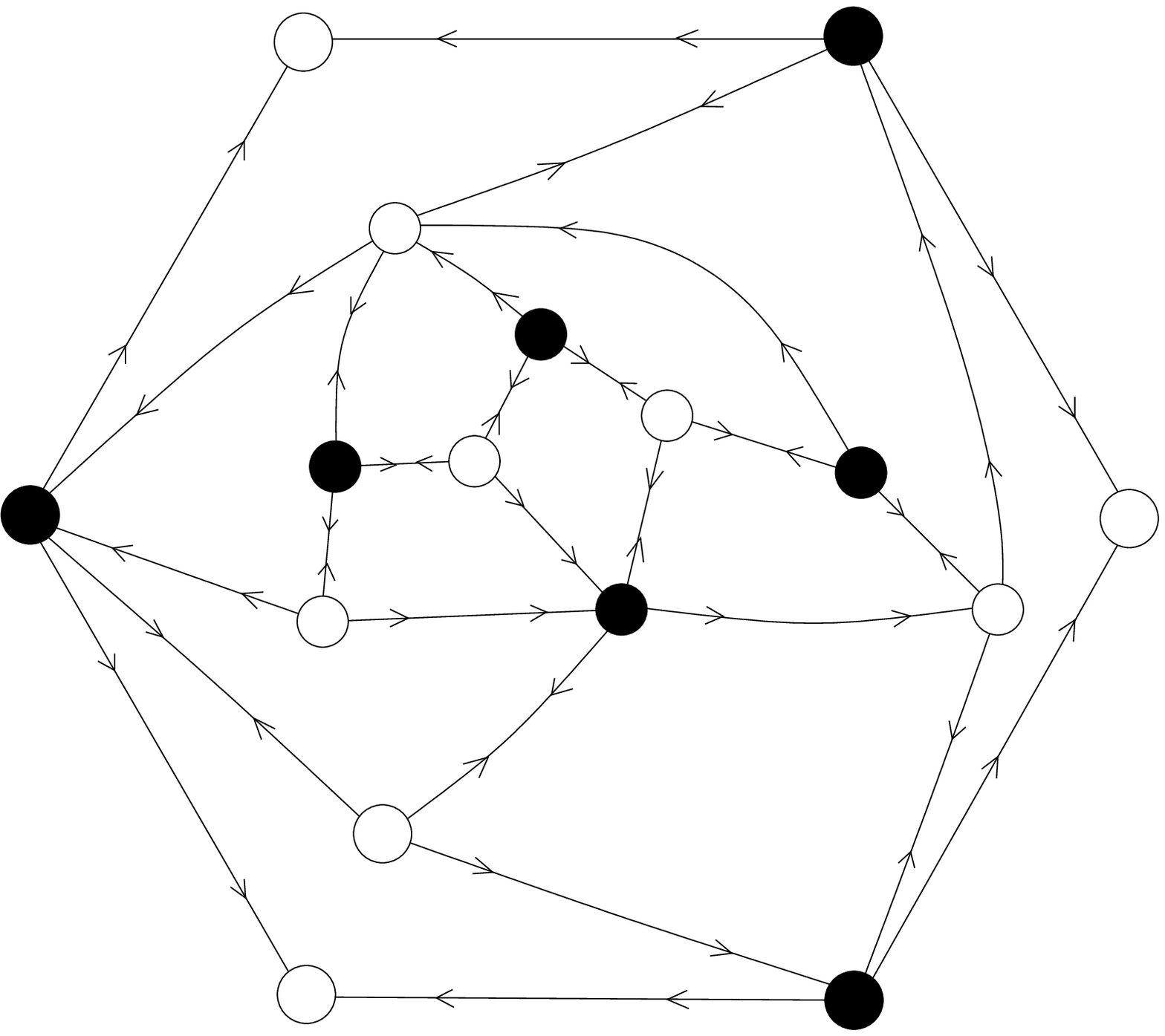}}\qquad
  \subfigure[]{\includegraphics[height=\etalon]{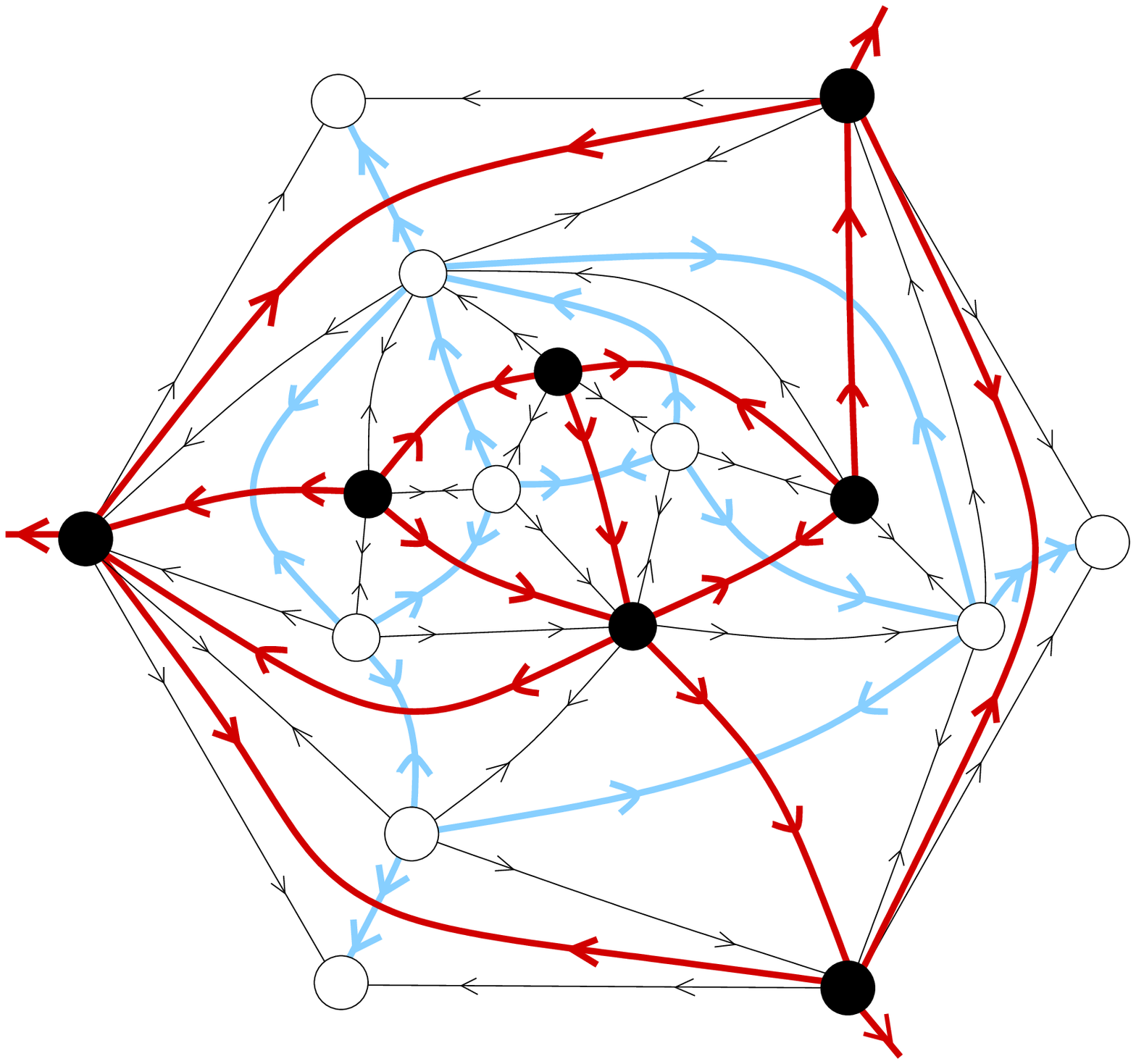}}\qquad
  \subfigure[]{\includegraphics[height=\etalon]{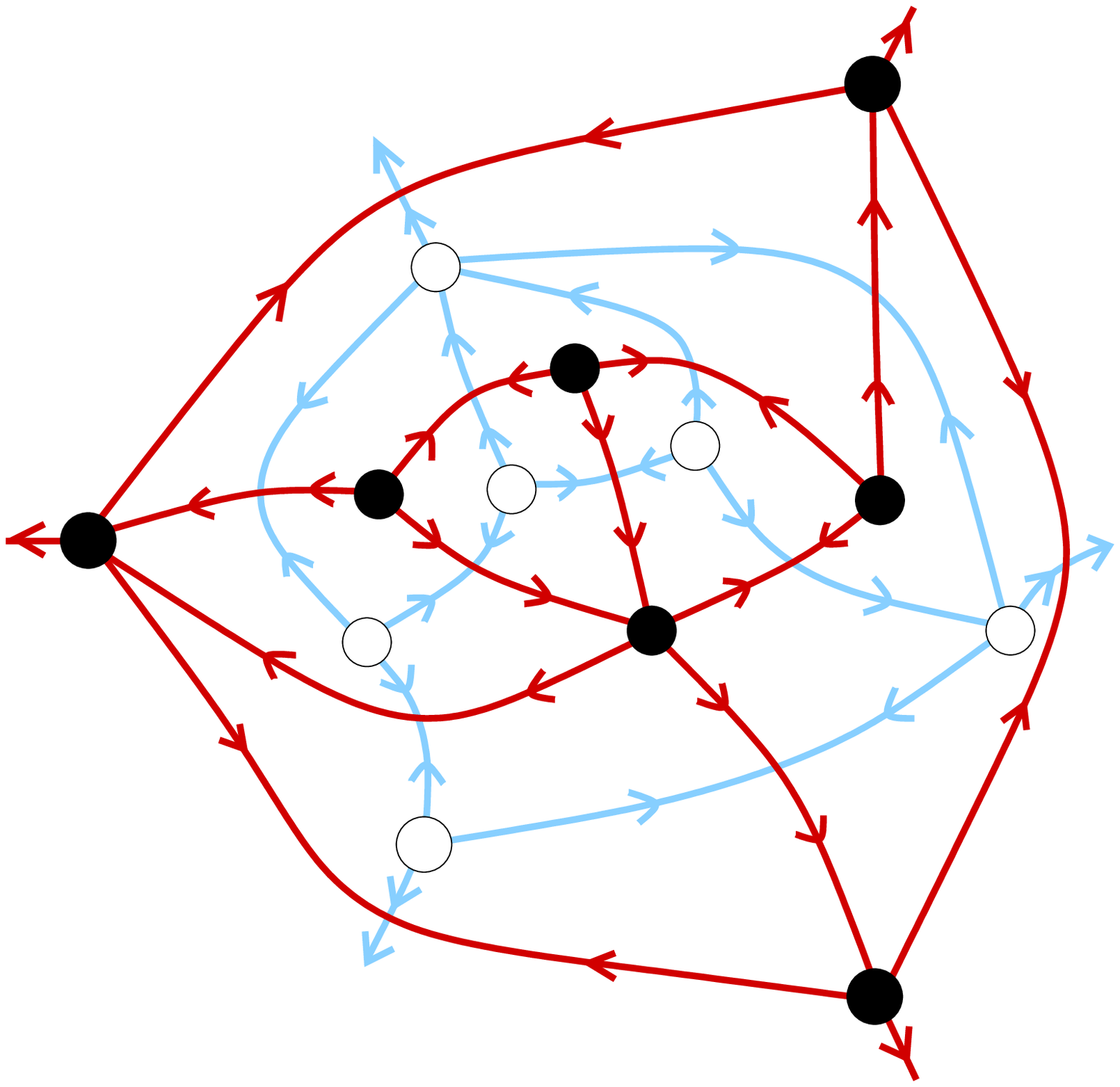}}
  \caption{The construction of the derived map of a bicolored complete
    irreducible dissection. The dissection is endowed with a
    complete-tri-orientation without clockwise circuit, and the
    derived map is endowed with the orientation obtained using the
    transposition rules for orientations.}
  \label{figure:localRules}
\end{figure}

\begin{lemma}
  \label{lemma:orientationDerivated}
  Let $D$ be a bicolored complete irreducible dissection endowed with
  a complete-tri-orientation without clockwise circuit. Then the
  orientation of the derived map $G'$ of $D$ obtained using the
  transposition rules has the following properties:
  \begin{itemize}
  \item
    each primal or dual vertex of $G'$ has outdegree 3.  
  \item
    each edge-vertex of $G'$ has outdegree 1.
  \end{itemize}
  In other words, the orientation of $G'$ obtained by applying the
  transposition rules is an $\alpha_0$-orientation.
\end{lemma}

\begin{proof}
  The first point is trivial.  For the second point, let $f$ be an
  inner face of $D$ and $v_f$ the associated edge-vertex of $G'$ (we
  recall that $v_f$ is the intersection of the two diagonals of
  $f$). The transposition rules for orientation ensures that the
  outdegree of $v_f$ in $G'$ is the number $n_f$ of inward half-edges
  of $D$ incident to $f$.  Hence, to prove that each edge-vertex of
  $G'$ has outdegree~1, we have to prove that $n_f=1$ for each inner
  face $f$ of~$D$. Observe that $n_f$ is a positive number, otherwise
  the contour of $f$ would be a clockwise circuit. Let $n$ be the
  number of inner vertices of~$D$.  Euler's relation implies that $D$
  has $(n+2)$ inner faces and $(4n+14)$ half-edges.  By definition of
  a complete-tri-orientation, $3(n+3)$ half-edges are outgoing. Hence,
  $(n+5)$ half-edges are ingoing. Among these $(n+5)$ ingoing
  half-edges, exactly three are incident to the outer face (see
  Figure~\ref{figure:completedOrientations}(c)). Hence, $D$ has
  $(n+2)$ half-edges incident to an inner face, so that
  $\sum_fn_f=n+2$. As $\sum_fn_f$ is a sum of $(n+2)$ positive numbers
  adding to $(n+2)$, the pigeonhole's principle ensures that $n_f=1$
  for each inner face $f$ of~$D$.  \hfill $\qed$
\end{proof}

\subsection{Uniqueness of a tri-orientation without clockwise circuit}

The following lemma is the companion of Lemma~\ref{lemma:orientationDerivated}
and is crucial to establish the uniqueness of a tri-orientation
without clockwise circuit for any irreducible dissection.

\begin{lemma}
  \label{lemma:orientationDerivatedComplete}
  Let $D$ be a bicolored complete irreducible dissection endowed with
  a complete-tri-orientation $Z$ without clockwise circuit. Let $G'$
  be the derived map of~$D$. Then the $\alpha_0$-orientation $X$ of
  $G'$ obtained from $Z$ by the transposition rules has no clockwise
  circuit.
\end{lemma}

\begin{proof}
  Assume that $X$ has a clockwise circuit $\cC$. Each edge of $G'$
  connects an edge-vertex and a vertex of the original
  dissection~$D$. Hence, the circuit $\cC$ consists of a sequence of
  pairs $(\underline{e},\overline{e})$ of consecutive edges of $G'$
  such that $\underline{e}$ goes from a vertex $\underline{v}$ of the
  dissection toward an edge-vertex $v'$ of $G'$ and $\overline{e}$
  goes from $v'$ toward a vertex $\overline{v}$ of the dissection. Let
  $(e_1',\ldots,e_m')$ be the sequence of edges of $G'$ between
  $\underline{e}$ and $\overline{e}$ in clockwise order around $v'$,
  so that $e_1'=\underline{e}$; and $e_m'=\overline{e}$ and let
  $(v_1,\ldots,v_m)$ be their respective extremities, so that
  $v_1=\underline{v}$ and $v_m=\overline{v}$. Notice that $2\leq m\leq
  4$.

\begin{figure}
  \centering
  \scalebox{.7}{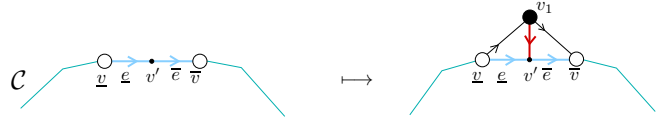}
  \caption{An oriented path of edges of the dissection can be associated
      to each pair $(\underline{e},\overline{e})$ of consecutive edges of
      $\cC$ sharing an edge-vertex.}
  \label{figure:coupeEdges}
\end{figure}

  As each edge-vertex has outdegree 1 in $X$ and as $e_m'$ is going
  out of $v'$, the edges $e_1',\ldots,e_{m-1}'$ are directed toward
  $v'$. Hence, the transposition rules for orientations ensure that the
  edges $(v_i,v_{i+1})$, for $1\leq i\leq m-1$, are all bi-oriented or
  oriented from $v_i$ to $v_{i+1}$ in the complete-tri-orientation $Z$
  of~$D$. Hence, we can go from $\underline{v}$ to $\overline{v}$
  passing by the exterior of $\cC$ and using only edges of~$D$, see
  Figure~\ref{figure:coupeEdges} for an example, where $m=3$.

  Concatenating the paths of edges of $D$ associated to each pair
  $(\underline{e},\overline{e})$ of $\cC$, we obtain a closed oriented path
  of edges of $D$ enclosing the interior of $\cC$ on its right. Clearly,
  a simple clockwise circuit can be extracted from this closed path,
  see Figure~\ref{figure:extractSimpleCycle}. As the
  complete-tri-orientation $Z$ has no clockwise circuit, this yields a
  contradiction.  
\phantom{1}\hfill $\qed$
\end{proof}

\begin{figure}
  \centering
  \scalebox{.4}{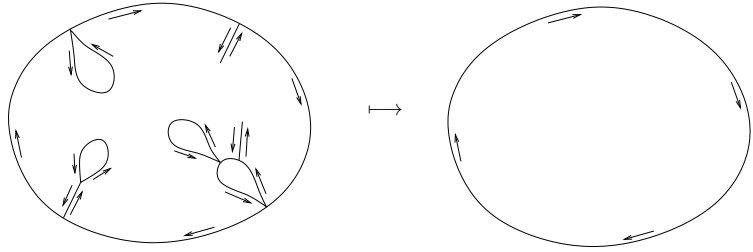}
  \caption{A simple clockwise circuit can be extracted from an oriented
    path enclosing a bounded simply connected region on its right.}
\label{figure:extractSimpleCycle}
\end{figure}

\begin{proposition}\label{proposition:uniqueness}
  Each irreducible dissection has at most one tri-orientation without
  clockwise circuit.
\end{proposition}

\begin{proof}
  Let $D$ be a bicolored complete irreducible dissection and $G'$ its
  derived map. A first important remark is that the transposition
  rules for orientations clearly define an injective mapping. In
  addition, Lemma~\ref{lemma:orientationDerivatedComplete} ensures
  that the image of a complete-tri-orientation of $D$ without
  clockwise circuit is an $\alpha_0$-orientation of $G'$ without
  clockwise circuit. Hence, injectivity of the mapping and uniqueness
  of an $\alpha_0$-orientation without clockwise circuit of $G'$
  (Theorem~\ref{theorem:felsner}) ensure that $D$ has at most one
  complete-tri-orientation without clockwise circuit. Hence,
  Proposition~\ref{proposition:complete} implies that each irreducible
  dissection has at most one tri-orientation without clockwise
  circuit.  
\hfill $\qed$
\end{proof}

\subsection{Existence of a tri-orientation without clockwise circuit}
\label{section:existence}

\paragraph{Inverse of the transposition rules}
Let $D$ be a bicolored complete irreducible dissection and $G'$ its
derived map. Given an $\alpha_0$-orientation of $G'$, we
associate to this orientation an orientation of the half-edges of $D$
by performing the inverse of the transposition rules: each
half-edge $h$ of $D$ receives the orientation of 
the edge of $G'$ that follows $h$ in clockwise order around
its incident vertex, see Figure~\ref{figure:localRules}(b).

\begin{lemma}
  \label{lemma:invRulesTrans}
  Let $D$ be an irreducible dissection and $G'$ the derived map
  of~$D$, endowed with its minimal $\alpha_0$-orientation.  Then the
  inverse of the transposition rules for orientations yields a
  complete-tri-orientation of~$D$.
\end{lemma}
 
\begin{proof}
  The inverse of the transposition rules is clearly such that a vertex
  has the same outdegree in the orientation of $D$ as in the
  $\alpha_0$-orientation of $G'$. Hence, each vertex of $D$ has
  outdegree 3 except the 3 outer white vertices that have outdegree~0,
  see Figure~\ref{figure:localRules}(b).

  To prove that the orientation of $D$ is a complete-tri-orientation,
  it remains to show that the two half-edges of an edge $e$ of $D$ can
  not both be oriented inward. Assume a contrario that there exists
  such an edge $e$.  The transposition rules for orientation and the
  fact that each edge-vertex of $G'$ has outdegree 1 imply that the
  boundary of the face $f_e$ of $G'$ associated to $e$ is a clockwise
  circuit, see Figure~\ref{figure:inverseTri}. This yields a
  contradiction with the minimality of the $\alpha_0$-orientation.
  \hfill $\qed$
\end{proof}

\begin{figure}
  \centering
  \scalebox{.4}{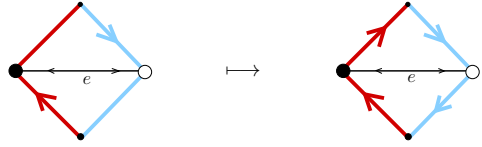}
  \caption{The case where the two half-edges of $e$ are oriented
    inward implies that the boundary of the associated face of $G'$
    is a clockwise circuit.}
    \label{figure:inverseTri}
\end{figure}


\begin{lemma}\label{lemma:derivatedtodis}
  Let $D$ be a bicolored complete irreducible dissection and let $G'$
  be its derived map.  Then the complete-tri-orientation of $D$
  associated with the minimal $\alpha_0$-orientation of $G'$ has no cw
  circuit.
\end{lemma}

\begin{proof}
  Let $X$ be the minimal $\alpha_0$-orientation of $G'$ and let $Z$ be
  the associated complete-tri-orientation of~$D$.  Assume that $Z$ has
  a clockwise circuit $\mathcal{C}$. For each vertex $v$ on
  $\mathcal{C}$, we denote by $h_v$ the half-edge of $\mathcal{C}$
  starting from $v$ with the interior of $\mathcal{C}$ on its
  right, and we denote by $e_v$ the edge of $G'$ that follows $h_v$ in
  clockwise order around $v$.  As $\mathcal{C}$ is a clockwise circuit
  for $Z$, $h_v$ is going out of $v$.  Hence, by definition of the
  transposition rules, $e_v$ is going out of $v$. Observe that, in the
  interior of $\mathcal{C}$, $e_v$ is the most counter-clockwise edge
  of $G'$ incident to $v$.

  We use this observation to build iteratively a clockwise circuit of
  $X$, yielding a contradiction.  First we state the following result
  proved in~\cite{Fe03}: ``for each vertex $v\in G'$ there exists a
  simple oriented path $\cP_v$ in $G'$, called the \emph{straight
    path} of $v$, which starts at $v$ and ends at a vertex incident to
  the outer face of $G'$".  Let $v_0$ be a vertex on $\mathcal{C}$,
  and $\mathcal{P}_{v_0}$ be the straight path starting at $e_{v_0}$
  for the orientation $X$.  Then $\mathcal{P}_{v_0}$ has to reach
  $\mathcal{C}$ at a vertex $v_1$ different from $v_0$.  Denote by
  $P_1$ the part of $\mathcal{P}_{v_0}$ between $v_0$ and $v_1$, by
  $\Lambda_1$ the part of the clockwise circuit $\mathcal{C}$ between
  $v_1$ and $v_0$, and by $\mathcal{C}_1$ the cycle enclosed by the
  concatenation of $P_1$ and $\Lambda_1$.  Let $\mathcal{P}_{v_1}$ be
  the straight path starting at $e_{v_1}$. The fact that $e_{v_1}$ is
  the most counterclockwise incident edge of $v_1$ in the interior of
  $\mathcal{C}$ ensures that $\mathcal{P}_{v_1}$ starts in the
  interior of $\mathcal{C}_1$.  Then, the path $\mathcal{P}_{v_1}$ has
  to reach $\mathcal{C}_1$ at a vertex $v_2\neq v_1$.  We denote by
  $P_2$ the part of the path $\mathcal{P}_{v_1}$ between $v_1$ and
  $v_2$.  If $v_2$ belongs to $P_1$, then the concatenation of the
  part of $P_1$ between $v_2$ and $v_1$ and of the part of $P_2$
  between $v_1$ and $v_2$ is a clockwise circuit, a contradiction.
  Hence, $v_2$ is on $\Lambda_1$ strictly between $v_1$ and $v_0$.  We
  denote by $\overline{P}_2$ the concatenation of $P_1$ and $P_2$, and
  by $\Lambda_2$ the part of $\mathcal{C}$ going from $v_2$ to
  $v_0$. As $v_2$ is strictly between $v_1$ and $v_0$, $\Lambda_2$ is
  strictly included in $\Lambda_1$.  Finally, we denote by
  $\mathcal{C}_2$ the cycle made of the concatenation of
  $\overline{P}_2$ and $\Lambda_2$. Hence, similarly as for the path
  $\mathcal{P}_{v_1}$, the straight path $\mathcal{P}_{v_2}$ starting
  at $e_{v_2}$ must start in the interior of $\mathcal{C}_2$.

  Then we continue iteratively, see Figure~\ref{figure:cyclepaths}.
  At each step $k$, we consider the straight path
  $\mathcal{P}_{v_{k}}$ starting at $e_{v_{k}}$. This path starts in
  the interior of the cycle $\mathcal{C}_{k}$, and reaches
  $\mathcal{C}_{k}$ at another vertex $v_{k+1}$. This vertex $v_{k+1}$
  can not belong to $\overline{P}_{k}:=P_1\cup\ldots\cup P_{k}$,
  otherwise a clockwise circuit of $X$ would be created. Hence,
  $v_{k+1}$ is on $\mathcal{C}$ strictly between $v_{k}$ and
  $v_{0}$. In particular the path $\Lambda_{k+1}$ going from $v_{k+1}$
  to $v_{0}$ on $\mathcal{C}$, is strictly included in the path
  $\Lambda_{k}$ going from $v_{k}$ to $v_{0}$ on $\mathcal{C}$, i.e.,
  $\Lambda_k$ shrinks strictly at each step. Thus, there must be a
  step $k_0$ when $\mathcal{P}_{v_{k_0}}$ reaches $\mathcal{C}_{k_0}$
  at a vertex on $\overline{P}_{k_0}$, creating a clockwise circuit of
  $X$, a contradiction.  
\hfill $\qed$
\end{proof}

\begin{figure}
  \centering
  \scalebox{.8}{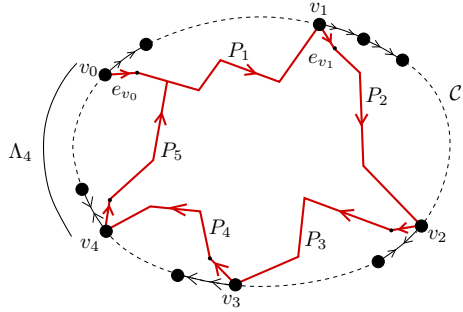}
  \caption{The presence of a clockwise circuit in $Z$ implies the
    presence of a clockwise circuit in $X$.}
  \label{figure:cyclepaths}
\end{figure} 

\begin{proposition}
  \label{propo:existsTri}
  For each irreducible dissection, there exists a tri-orientation without 
  clockwise circuit.
\end{proposition}

\begin{proof}
  Lemma~\ref{lemma:derivatedtodis} ensures that each bicolored
  complete irreducible dissection $D$ has a complete-tri-orientation
  $Z$ without clockwise circuit; and
  Proposition~\ref{proposition:complete} ensures that the existence of
  a complete-tri-orientation without clockwise circuit for any
  bicolored complete irreducible dissection implies the existence of a
  tri-orientation without clockwise circuit for any irreducible
  dissection.  
\hfill $\qed$
\end{proof} 

Finally, Theorem~\ref{theorem:uniquenessexistence} follows from
Proposition~\ref{proposition:uniqueness} and
Proposition~\ref{propo:existsTri}.


\newcommand{\gs}[1]{\marginpar{\begin{raggedleft}#1\end{raggedleft}}}

\section{Computing the minimal
  $\alpha_0$-orientation of a derived map} 
\label{section:compute}
We describe in this section a linear-time algorithm to compute the
minimal $\alpha_0$-orientation of the derived map of an
outer-triangular 3-connected plane graph. This result is crucial for
the encoding algorithm of Section~\ref{section:codingAlgo} to have
linear time complexity (see the transition between
Figure~\ref{fig:Coding}(b) and Figure~\ref{fig:Coding}(c)).

 As discussed in \cite{Fe03}, given a 3-connected map $G$ and its
 derived map $G'$, an $\alpha_0$-orientations of $G'$ corresponds to a
 so-called \emph{Schnyder wood} of $G$. These Schnyder woods of
 3-connected maps are the right generalisations of Schnyder woods of
 triangulations \cite{S90}.  Quite naturally, our algorithm is a
 generalization of the algorithm to compute the minimal Schnyder wood
 of a triangulation \cite{Bre02}.  The ideas for the extension to
 3-connected maps have already been introduced by \cite{Ka96}
 and~\cite{DiTa}. The algorithm of~\cite{DiTa} outputs a Schnyder
 wood of a 3-connected map; which can be subsequently made minimal by
 iterated circuit reversions with a linear overall complexity, as
 easily follows from ideas presented in~\cite{Khu93}.  Our algorithm
 relies on similar principles, suitably modified so as to ouput
 \emph{directly} the minimal Schnyder wood (i.e., the Schnyder wood
 associated with the minimal $\alpha_0$-orientation), also
in linear time.  In itself our
 algorithm for 3-connected maps is only slightly more involved than
 the algorithm for triangulations, as opposed to the correctness
 proof, which is much harder (see the discussion at the beginning of
 Section~\ref{section:proofCorrect}). Because of this we give a rather
 proof-oriented description of the algorithm.

Our algorithm is also of independent interest in connection with
Schnyder woods, and it has applications in the context of graph
drawing. Indeed, the minimal Schnyder wood orientation is also a key
ingredient for the straight-line drawing algorithm presented in
\cite{BFM04}. This algorithm relies on operations of edge-deletion,
embedding of the obtained graph, and then embedding of the deleted
edges. The grid size is guaranteed to be bounded by $(n-2)\times(n-2)$
---equalling at least Schnyder's algorithm~\cite{S90}--- provided the
Schnyder wood used is the one associated to the minimal
$\alpha_0$-orientation. An implementation of this drawing algorithm
including our orientation algorithm has been made available by
Bonichon in \cite{taxiplan}.

\subsection{Principle of the algorithm}

Let $G$ be an outer-triangular 3-connected planar graph and let $G'$
be its derived map and $G^*$ its dual map.  We denote by $a_1$, $a_2$
and $a_3$ the outer vertices of $G$ in clockwise order. We describe
here a linear-time iterative algorithm to compute the minimal
$\alpha_0$-orientation of $G'$. The idea is to maintain a simple cycle
of edges of $G$; at each step $k$, the cycle, denoted by
$\mathcal{C}_k$, is shrinked by choosing a so-called \emph{eligible}
vertex $v$ on $\mathcal{C}_k$, and by removing from the interior of
$\mathcal{C}_k$ all faces incident to $v$. The eligible vertex is
always different from $a_2$ and $a_3$, so that the edge $(a_2,a_3)$,
called \emph{base-edge}, is always on $\mathcal{C}_k$.  The edges of
$G'$ ceasing to be on $\mathcal{C}_k$ or in the interior of
$\mathcal{C}_k$ are oriented so that the following invariants remain
satisfied.

\smallskip

\noindent\emph{Orientation invariants:}
\begin{longitem}
\item
  For each edge $e$ of $G$ outside $\mathcal{C}_k$, the 4 edges of
  $G'$ incident to the edge-vertex $v_e$ associated to $e$ have been
  oriented at a step $j< k$ and $v_e$ has outdegree~1.
\item
  All other edges of $G'$ are not yet oriented.
\end{longitem}
Moreover, the edges that correspond to half-edges of $G$ also
receive a label in $\{1,2,3\}$, so that the following invariants
for labels remain satisfied: 

\smallskip

\noindent\emph{Labelling invariants:}
\begin{longitem}
\item At each step $k$, every vertex $v$ of $G$ outside of $\mathcal{C}_k$
  has one outgoing half-edge for each label 1, 2 and 3 and these outgoing
  edges appear in clockwise order around $v$. In addition, all edges
  between the outgoing edges with labels $i$ and $i+1$ are incoming with
  label $i-1$, see Figure~\ref{figure:invariantslabelling}(a).
\item Let $v$ be a vertex of $G$ on $\mathcal{C}_k$ having at least
  one incident edge of $G'$ outside of $G_k$. Then exactly one of
  these edges, denoted by $e_1'$, is going out of $v$. In addition it
  has label~1. The edges of $G'$ incident to $v$ and between $e_1'$
  and its left neighbour on $\mathcal{C}_k$ are incoming with label~2;
  and the edges incident to $v$ in $G'$ between $e_1'$ and its right
  neighbour on $\mathcal{C}_k$ are incoming with label~3, see
  Figure~\ref{figure:invariantslabelling}(b).
\item For each edge $e$ of $G$ outside of $G_k$, let $e'$ be the unique
  outgoing edge of its associated edge-vertex $v_e$. Two cases can occur:
  \begin{longitem}
  \item If $e'$ is an half-edge of $G$ then the two edges of $G'$
    incident to $v_e$ and forming the edge $e$ are identically
    labelled. This corresponds to the case where $e$ is ``simply
oriented''.
  \item If $e'$ is an half-edge of $G^*$, we denote by $1\leq i\leq 3$ the
    label of the edge of $G'$ following $e'$ in clockwise order
    around $v_e$. Then the edge of $G'$ following $e'$ in
    counter-clockwise order around $v_e$ is labelled $i+1$, see
    Figure~\ref{figure:invariantslabelling}(c). This corresponds to the case
where $e$ is ``bi-oriented''.
  \end{longitem}
\end{longitem}

Actually, the labels are not needed to compute the orientation, 
but they will be very useful to prove
that the algorithm outputs the minimal $\alpha_0$-orientation. These
labels are in fact the ones of the Schnyder woods of $G$, as discussed
in~\cite{Fe03}.

In the following, we write $G_k$ for the submap of $G$ obtained by
removing all vertices and edges outside of $\mathcal{C}_k$ (at step $k$).
In addition, we order the vertices of $\mathcal{C}_k$ from left to 
right according to the order induced by the path 
$\mathcal{C}_k\backslash\{a_2,a_3\}$, 
with $a_3$ as left extremity and $a_2$ as right extremity.
In other words, a vertex $v\in\mathcal{C}_k$ is on the left of
a vertex $v'\in\mathcal{C}_k$ if the path of $\mathcal{C}_k$
going from $v$ to $v'$ without passing by the edge $(a_2,a_3)$ 
has the interior of $\mathcal{C}_k$ on its right. 

\begin{figure}
  \def\etalon{.18\linewidth}
  \centering
  \subfigure[]{\includegraphics[height=\etalon]{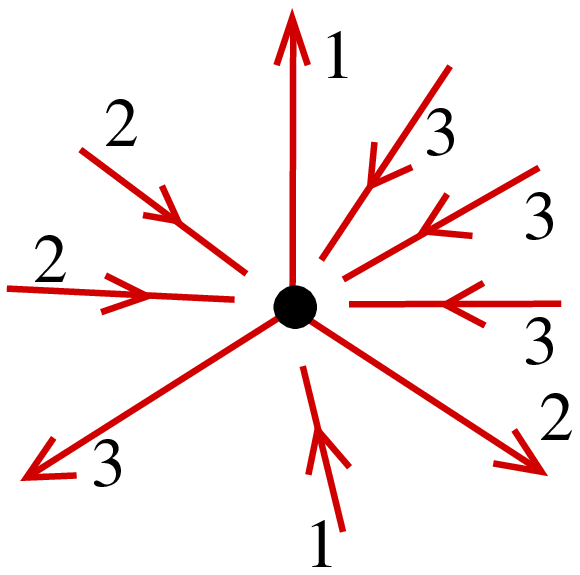}}\qquad\qquad
  \subfigure[]{\includegraphics[height=\etalon]{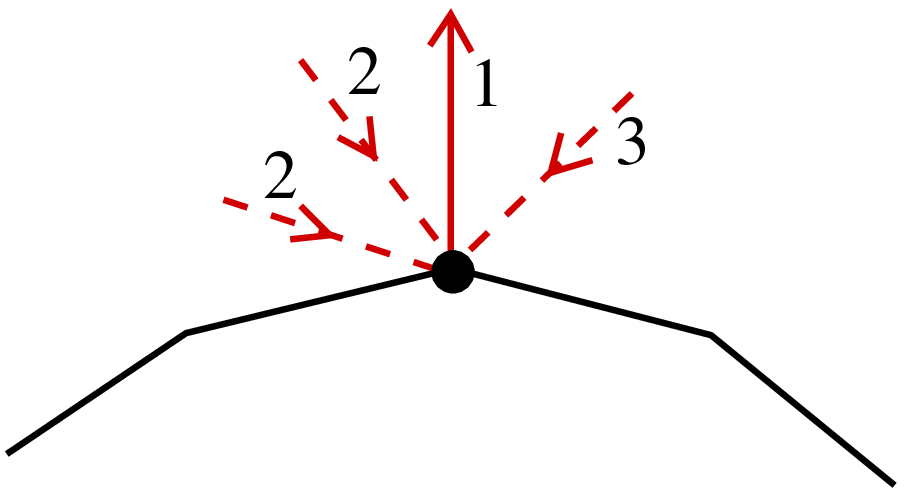}}\qquad\qquad
  \subfigure[]{\includegraphics[height=\etalon]{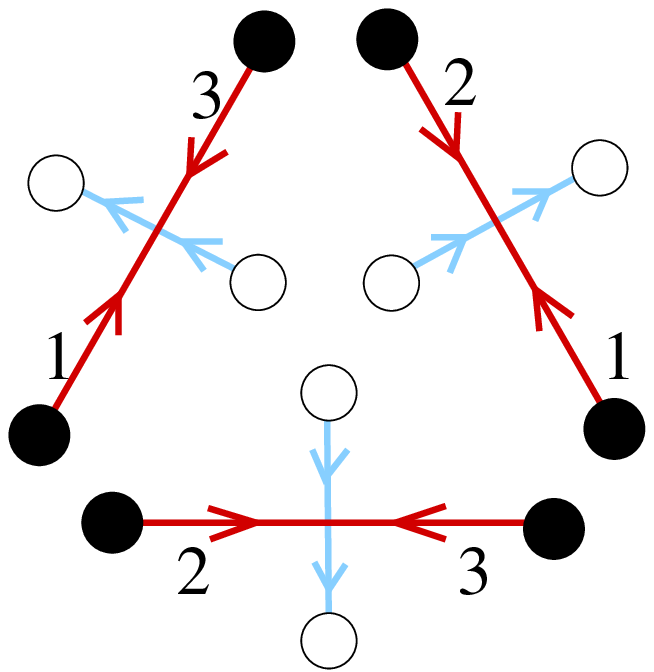}}
    \caption{The invariants for the labels of the half-edges of
      $G$ maintained during the algorithm.}
    \label{figure:invariantslabelling}
\end{figure}

\subsection{Description of the main iteration}

Let us now describe the $k$-{th} step of the algorithm, during which
the cycle $\mathcal{C}_k$ is shrinked so that the invariants for
orientation and labelling remain satisfied. The description
requires some definitions.

\paragraph{Definitions}
A vertex of $\mathcal{C}_k$ is said to be \emph{active} if it is incident to
at least one edge of $G\backslash G_k$. Otherwise, the vertex is
\emph{passive}. By convention, before the first step of the algorithm, the
vertex $a_1$ is considered as active and its incident half-edge
directed toward the outer face is labelled~1.

For each pair of vertices $(v_1,v_2)$ of $\mathcal{C}_k$ ---with
$v_1$ is on the left of $v_2$---, the path on 
$\mathcal{C}_k$ going from $v_1$
to $v_2$ without passing by the edge $(a_2,a_3)$ is denoted by
$\lbrack v_1,v_2 \rbrack $. We also write $\rbrack v_1, v_2
\lbrack$ for $\lbrack v_1,v_2 \rbrack$ deprived from the endvertices
$v_1$ and $v_2$. 

A pair $( v_1,v_2)$ of vertices of $\mathcal{C}_k$ is
\emph{separating} if there exists an inner face $f$ of $G_k$ such that
$v_1$ and $v_2$ are incident to $f$ but the edges of $\lbrack v_1,v_2
\rbrack $ are not all incident to $f$. Such a face is called a
\emph{separating face} and the triple $(v_1,v_2,f)$ is called a
\emph{separator}. The (closed) area delimited by the path $\lbrack
v_1,v_2 \rbrack $ and by the path of edges of $f$ going from $v_1$ to
$v_2$ with the interior of $f$ on its right is called the
\emph{separated area} of $(v_1,v_2,f)$ and is 
denoted by $\mathrm{Sep}(v_1,v_2,f)$.

A vertex $v$ on $\mathcal{C}_k$ is said to be \emph{blocked} if it 
belongs to a separating pair. It is easily checked that a vertex is blocked
iff it is
incident to a separating face of $G_k$. In particular, a non blocked
vertex does not belong to any separating pair of vertices. By
convention, the vertices $a_2$ and $a_3$ are always considered as
blocked.
A vertex $v$ on $\mathcal{C}_k$ is \emph{eligible} if it is
active and not blocked.

Finally, for each vertex $v$ of $\mathcal{C}_k$, we define its
\emph{left-connection vertex} $\mathrm{left}(v)$ as the leftmost vertex on
$\mathcal{C}_k$ such that the vertices of $\rbrack \mathrm{left}(v),v \lbrack
$ all have degree 2 in $G_k$. The path $\lbrack
\mathrm{left}(v),v \rbrack $ is called the \emph{left-chain} of $v$ and the
first edge of $\lbrack \mathrm{left}(v),v \rbrack $ is called the
\emph{left-connection edge} of $v$. Similarly, we define the
\emph{right-connection vertex}, the \emph{right-chain}, and the
\emph{right-connection edge} of $v$.
Notice that all vertices of $\rbrack\mathrm{left}(v),v
\lbrack $ and of $\rbrack v,\mathrm{right}(v) \lbrack$ are active,
as each vertex of a 3-connected graph has degree at least 3.

\paragraph{Operations at step $k$}

First, we choose the rightmost eligible vertex of $\mathcal{C}_k$ and
we call $v^{(k)}$ this vertex. (We will prove in
Lemma~\ref{lemma:terminates} that there always exists an eligible
vertex on $\mathcal{C}_k$ as long as $G_k$ is not reduced to the edge
$(a_2,a_3)$.) Notice that this eligible vertex can not be $a_2$ nor
$a_3$ because $a_2$ and $a_3$ are blocked.

We denote by $f_1,\ldots ,f_m$ the bounded 
faces of $G_k$ incident to $v^{(k)}$ from right to left, and by
$e_1,\ldots ,e_{m+1}$ the edges of $G_k$ incident to $v^{(k)}$ from right to 
left. Hence, for each $1\leq i\leq m$, $f_i$ corresponds to the sector between 
$e_i$ and $e_{i+1}$. 

An important remark is that the right-chain of
$v^{(k)}$ is reduced to one edge. Indeed, if there exists a vertex $v$ in
$\rbrack v^{(k)},\mathrm{right}(v^{(k)})\lbrack$, then $v$ is active, as
discussed above. In addition, $v$ is incident  to only one inner face of
$G_k$, namely $f_1$. As $f_1$ is incident to $v^{(k)}$ and as
$v^{(k)}$ is non blocked, $f_1$ is not separating. Hence $v$ is not
blocked. Thus $v$ is eligible and is on the right of $v^{(k)}$, in
contradiction with the fact that $v^{(k)}$ is the rightmost eligible
vertex on $\mathcal{C}_k$.

We label and orient the edges of $G'$ incident to
the edge-vertices on the left-chain of $v^{(k)}$ and
on the edges $e_1,\ldots e_m$, see Figure~\ref{figure:stepk}:

\begin{figure}
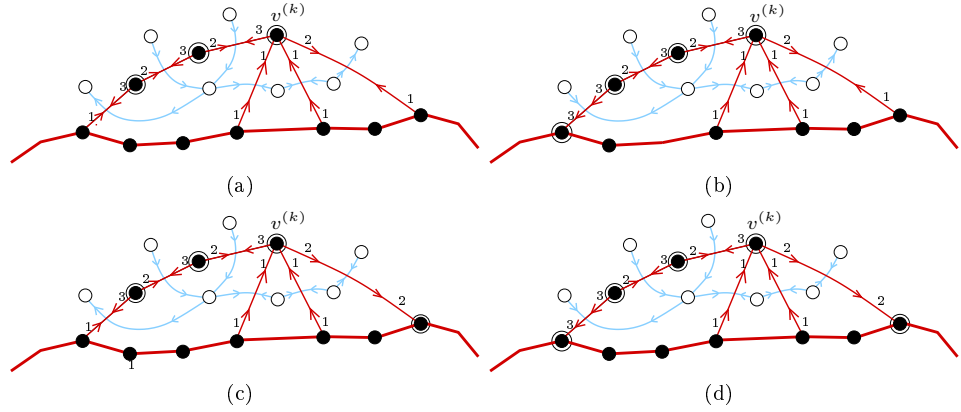

  \centering
  \subfigure[]{{\input{Figures/operationstepk-a.pstex_t}}}\
  \subfigure[]{{\input{Figures/operationstepk-b.pstex_t}}}\
  \subfigure[]{{\input{Figures/operationstepk-c.pstex_t}}}\
  \subfigure[]{{\input{Figures/operationstepk-d.pstex_t}}}
  \caption{The operations performed at step $k$ of the algorithm, whether
    $\mathrm{left}(v^{(k)})$ and $\mathrm{right}(v^{(k)})$ 
    are passive-passive (Fig. a) or
    active-passive (Fig. b) or passive-active (Fig. c) or active-active (Fig. d).
    Active vertices are surrounded.}
  \label{figure:stepk}
\end{figure}

\begin{longitem}
\item \textbf{Inner edges}: For each edge $e_i$ with $2\leq i \leq m$,
  we denote by $v_{e_i}$ the corresponding edge-vertex of $G'$. Orient
  the two edges of $G'$ forming $e_i$ toward $v^{(k)}$ and give label
  1 to these two edges. Orient the two other incident edges of
  $v_{e_i}$ toward $v_{e_i}$, so that $v_{e_i}$ has outdegree~1.
\item \textbf{Left-chain}: For each edge $e$ of the left-chain of $v^{(k)}$
---traversed from $v^{(k)}$ to $\mathrm{left}(v^{(k)})$---
different from the left-connection edge, bi-orient $e$ and give label 3 (resp.
label 2) to the first (resp. second) traversed half-edge. 
 Choose the unique outgoing
  edge of the edge-vertex $v_e$ associated to $e$ to be the
  edge going out of $e$ toward the
  interior of $\mathcal{C}_k$
\item \textbf{Left-connection edge}: If $\mathrm{left}(v^{(k)})$ is passive,
  bi-orient the left-connection edge $e$ of $v^{(k)}$, give label
1 to the half-edge incident to $\mathrm{left}(v^{(k)})$ and
label 3 to the other half-edge, and 
choose the unique outgoing edge of the edge-vertex $v_e$ to be the
  edge going out of $v_e$ toward the exterior of $\mathcal{C}_k$.  
If $\mathrm{left}(v^{(k)})$ is
  active, label 3 and orient toward $\mathrm{left}(v^{(k)})$ the two edges
  of $G'$ forming $e$, and orient the two dual edges incident to
  $v_e$ toward~$v_e$.
\item \textbf{Right-connection edge}: The edge $e_1$,
  which is the right-connection edge of $v^{(k)}$, is treated symmetrically
as the left-connection edge. If $\mathrm{right}(v^{(k)})$ is passive,
  bi-orient $e_1$, give label
1 to the half-edge incident to $\mathrm{right}(v^{(k)})$ and
label 2 to the other half-edge, and 
choose the unique outgoing edge of the edge-vertex $v_{e_1}$ to be the
  edge going out of $v_{e_1}$ toward the exterior of $\mathcal{C}_k$.  
If $\mathrm{right}(v^{(k)})$ is
  active, label 2 and orient toward $\mathrm{right}(v^{(k)})$ the two edges
  of $G'$ forming $e_1$, and orient the two dual edges incident to
  $v_{e_1}$ toward~$v_{e_1}$.
\end{longitem}

After these operations, all faces incident to $v^{(k)}$ are removed from the
interior of $\mathcal{C}_k$, producing a (shrinked) cycle $\mathcal{C}_{k+1}$.
As $a_2$ and $a_3$ are blocked on $\mathcal{C}_k$, 
$\mathcal{C}_{k+1}$ still contains the edge $(a_2,a_3)$.
In addition, if $\mathcal{C}_{k+1}$ is not reduced to $(a_2,a_3)$,
the property of 3-connectivity of $G$ and the fact that the chosen 
vertex $v^{(k)}$ is
not incident to any separating face easily ensure that
$\mathcal{C}_{k+1}$ is a simple cycle,
i.e., it does not contain any separating vertex.

It is also easy to get convinced from
Figure~\ref{figure:invariantslabelling} and Figure~\ref{figure:stepk}
that the operations performed at step $k$ maintain the 
invariants of orientation and labelling.

The purpose of the next two lemmas is to prove that the algorithm 
terminates.

\begin{lemma}\label{lemma:existsvertex}
  Let $(v_1,v_2,f)$ be a separator on $\mathcal{C}_k$. Then there exists an
  eligible vertex in $\rbrack v_1,v_2\lbrack$.
\end{lemma}

\begin{proof}
Consider the (non empty) set of separators whose separated area is
included or equal to the separated area of $(v_1,v_2,f)$, and let
$(v_1',v_2',f')$ be such a separator minimal w.r.t. the inclusion of
the separated areas.  Observe that $v_1'$ and $v_2'$ are in $\lbrack
v_1,v_2\rbrack$.

Assume that no vertex of $\rbrack v_1',v_2' \lbrack$ is active. Then
the removal of $v_1'$ and $v_2'$ disconnects
$\mathrm{Sep}(v_1',v_2',f)$ from $G\backslash
\mathrm{Sep}(v_1',v_2',f)$. This is in contradiction with
3-connectivity of $G$, because these two sets are easily proved to
contain at least one vertex different from $v_1'$ and $v_2'$.

Hence, there exists an active vertex $v$ in $\rbrack
v_1',v_2'\lbrack$, also in $\rbrack v_1,v_2\lbrack$. If $v$ was
incident to a separating face, this face would be included in the
separated area of $(v_1',v_2',f')$, which is impossible by minimality
of $(v_1',v_2',f')$. Hence, the active vertex $v$ is not blocked,
i.e., is eligible.  
\hfill $\qed$
\end{proof}

\begin{lemma} \label{lemma:terminates}
  As long as $\mathcal{C}_k$ is not reduced to $(a_2,a_3)$, there
  exists an eligible vertex on $\mathcal{C}_k$.
\end{lemma}

\begin{proof}
  Assume that there exists no separating pair of vertices on
  $\mathcal{C}_k$. In this case, an active vertex on $\mathcal{C}_k$
  different from $a_2$ and $a_3$ is eligible. Hence we just have to
  prove the existence of such a vertex.  At the first step of the
  algorithm, there exists an active vertex on $\mathcal{C}_1\backslash
  \{ a_2,a_3\}$ because $a_1$ is active by convention. At any other
  step, there exists an active vertex on $\mathcal{C}_k\backslash \{
  a_2,a_3\}$, otherwise the removal of $a_2$ and $a_3$ would
  disconnect $G_k \backslash \{a_2,a_3\}$ from $G\backslash G_k$, in
  contradiction with the 3-connectivity of $G$.

  If there exists at least one separator $(v_1,v_2,f)$,
  Lemma~\ref{lemma:existsvertex} ensures that there exists an eligible
  vertex $v$ in $\rbrack v_1, v_2\lbrack$.
\hfill $\qed$
\end{proof}

\paragraph{Last step of the algorithm}
Lemma~\ref{lemma:terminates} implies that, at the end of the iterations,
only the edge $e=(a_2,a_3)$ remains. To complete the orientation,
bi-orient $e$ and label 3 (resp. label 2) the half-edge of $e$
whose origin is $a_2$ (resp. $a_3$); the outgoing edge of the edge-vertex
$v_e$ (associated to $e$) is chosen to be the edge going out of $v_e$
toward the outer face. We also label
respectively 2 and 3 the half-edges incident to $a_2$ and $a_3$ and
directed toward the outer face.

Figure~\ref{figure:exampleAlgoWoods} illustrates the execution of the
algorithm on an example, where the edges of $\mathcal{C}_k$
are black and bolder. In addition, the active vertices are surrounded
and the rightmost eligible vertex $v^{(k)}$ is doubly surrounded.

\begin{figure}
  \centering 
  \includegraphics[width=.9\linewidth]{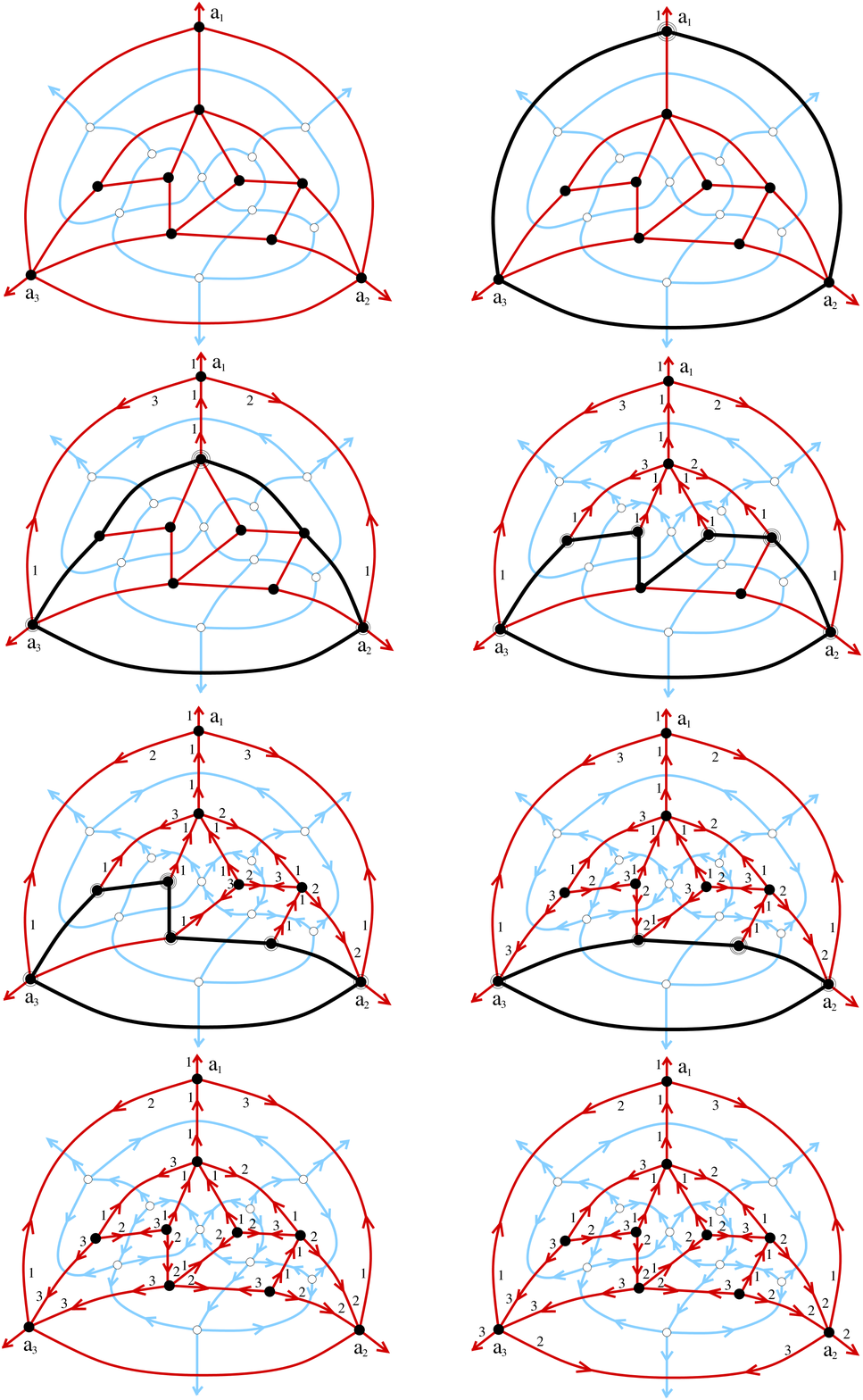}
  \caption{The execution of the algorithm of orientation on an example.}
  \label{figure:exampleAlgoWoods}
\end{figure}

\begin{theorem} \label{theorem:computesOrientation}
\label{THEOREM:COMPUTESORIENTATION}
  The algorithm outputs the minimal $\alpha_0$-orientation of the
  derived map.
\end{theorem}

Section~\ref{section:proofCorrect} is dedicated to the proof
of this theorem.

\paragraph{Remark}
As stated in Theorem~\ref{theorem:computesOrientation}, 
our orientation algorithm outputs a
particular $\alpha_0$-orientation, namely the minimal one. 
The absence of clockwise circuit is due to the fact that among all
eligible vertices, the rightmost one is chosen at each step. The algorithm
is easily adapted to other choices of eligible vertices: the only
difference is that the right-connection chain of the chosen eligible vertex
might not be reduced to an edge, in which case it must be dealt with
in a symmetric way as the left-connection chain (that is, 2 becomes 3
and left becomes right in the description of edge labelling and
orientation). This yields a  ``generic'' algorithm that can produce any
$\alpha_0$-orientations of $G'$. Indeed, given a particular 
$\alpha_0$-orientation $X$ of $G'$, it is easy to compute
 a scenario (i.e., a suitable choice of the
eligible vertex at each step) that outputs $X$. Such a scenario corresponds
to a so-called canonical ordering for treating the vertices, 
see~\cite{Ka96}.


\paragraph{Implementation}
Following~\cite{Ka96} (see also~\cite{Bre02} for the case of
triangulations), an efficient implementation is obtained by
maintaining, for each vertex $v\in\cC_k$, the number $s(v)$ of
separating faces incident to $v$. Thus, a vertex is blocked iff
$s(v)>0$. Notice that a face $f$ is separating iff the numbers $v(f)$
and $e(f)$ of vertices and edges (except $(a_2,a_3)$) of $f$ belonging
to $\cC_k$ satisfy $v(f)>e(f)+1$. Thus, it is easy to test if a face
is separating, so that the parameters $s(f)$ are also easily
maintained. The data structure we use is the half-edge structure,
which allows us to navigate efficiently on the graph.  The
\emph{pointer} is initially on $a_1$, which is the rightmost eligible
vertex at the first step. During the execution, once the vertex
$v^{(k)}$ is treated, the pointer is moved to $v$ the right neighbour
of $v^{(k)}$ on $\cC_k$. The crucial point is that, if $v$ is blocked,
then no vertex on the right of $v$ can be eligible (because of the
nested structure of separating faces). Thus, in this case, the pointer
is moved to the left until an eligible vertex is encountered. Notice
also that $v$ is active after $v^{(k)}$ is treated. Thus, if $v$ is
not blocked, then $v$ is eligible at step $k+1$. In this case, the
nested structure of separating faces ensures that the rightmost
eligible vertex at step $k+1$, if not $v$, is either the
right-connection vertex $r(v)$ of $v$, or the left neighbour of $r(v)$
on $\cC_{k+1}$ (in the case where $r(v)$ is not eligible).  Notice
that, in the case where $v$ is not blocked, the pointer is moved to
the right but the edges traversed will be immediately treated (i.e.,
removed from $\cC_{k+1}$) at step $k+1$.  This ensures that an edge
can be traversed at most twice by the pointer: once from right to left
and subsequently once from left to right. Thus, the complexity is
linear.

\section{Proof of Theorem~\ref{theorem:computesOrientation}}
\label{section:proofCorrect}

Let $G$ be an outer-triangular 3-connected map, and let $X_0$ be the
orientation of the derived map $G'$ computed by the orientation
algorithm. This section is dedicated to proving that $X_0$ is the
minimal $\alpha_0$-orientation of $G'$.

Our proof is inspired by the proof by~\citeN{Bre02} that ensures
that, for a triangulation, the choice of the rightmost eligible vertex
at each step yields the Schnyder woods without clockwise circuit. The
argument is the following: the presence of a clockwise circuit implies
the presence of an ``inclusion-minimal'' clockwise circuit which is,
in the case of a triangulation, a 3-cycle $(x,y,z)$. Then the
clockwise orientation of $(x,y,z)$ determines unambiguously (up to
rotation) the labels of the 3 edges of $(x,y,z)$. These labels
determine an order of treatment of the 3 vertices $x$, $y$ and $z$
that is not compatible with the fact that the eligible vertex chosen
at each step is the rightmost one.

In the general case of 3-connected maps, which we consider here, the
proof is more involved but follows the same lines. This time there is
a finite set of minimal patterns (for a triangulation this set is
restricted to the triangle), such that a minimal clockwise circuit
$\mathcal{C}$ in the orientation $X_0$ of the derived map $G'$ can
only correspond to one of these patterns (the list is shown in
Figure~\ref{figure:configuration}). A common characteristic is that
the presence of a clockwise circuit $\mathcal{C}$ for each of these
patterns implies the presence of three paths $\mathcal{P}_1$,
$\mathcal{P}_2$, $\mathcal{P}_3$ of edges of $G$ whose concatenation
forms a simple cycle in $G$ (in the case of a triangulation, the three
paths are reduced to one edge). In addition, the fact that
$\mathcal{C}$ is clockwise determines unambiguously the labels and
orientations of the edges of $\mathcal{P}_1$, $\mathcal{P}_2$ and
$\mathcal{P}_3$. Writing $v_1$, $v_2$ and $v_3$ for the respective
origins of these three paths, our proof (as in the case of
triangulations, but with quite an amount of technical details) relies
on the fact that the labels of $\mathcal{P}_1$, $\mathcal{P}_2$,
$\mathcal{P}_3$ imply an order for processing $\{v_1, v_2, v_3\}$ that is
not compatible with the fact that the eligible vertex chosen at each
step is the rightmost one.

\subsection{The algorithm outputs an $\alpha_0$-orientation}

By construction of the orientation, each primal vertex of the derived
map $G'$ has one outgoing edge in each label 1, 2 and 3, hence it has
outdegree~3. By construction also, each edge-vertex of $G'$ has
outdegree~1. Hence, to prove that $X_0$ is an $\alpha_0$-orientation,
it just remains to prove that each dual vertex of $G'$ has outdegree 3
in~$X_0$.

Let $f$ be an inner face of $G$ and $v_f$ the corresponding
dual vertex in $G^*$. Let $k$ be the step during which $f$ is
merged with the outer face of $G$. At this step, a sequence of
consecutive edges of $f$ has been removed. This path of removed consecutive 
edges is called the \emph{upper
path} of $f$. The path of edges of $f$ that are not in the upper path
of $f$ is called the \emph{lower path} of $f$. 
By construction of the orientation (see
Figure~\ref{figure:stepk}), exactly two edges of $G'$ connecting $v_f$
to an edge-vertex of the upper path of $f$ are going out of $v_f$:
these are the edge-vertices corresponding to the 
two extremal edges of the upper path.

Hence it just remains to prove that exactly one edge of $G'$
connecting $v_f$ to an edge-vertex of the lower path of $f$ is going
out of $v_f$. First, observe that the lower path $P$ of $f$ is a non
empty path of edges on $\mathcal{C}_{k+1}$, such that the two
extremities $v_l$ and $v_r$ of the path are active and all vertices of
$\rbrack v_l, v_r \lbrack$ are passive on $\mathcal{C}_{k+1}$, see
Figure~\ref{figure:stepk}. The fact that exactly one edge of $G'$
connecting $v_f$ to an edge-vertex of $P$ is going out of $v_f$ is a
direct consequence of the following lemma, see
Figure~\ref{figure:lowerpath}.

\begin{figure}
  \centering
  \scalebox{.4}{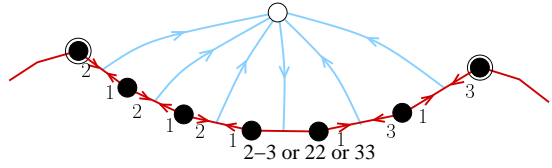}
  \caption{The dual vertex of a face $f$ has one outgoing edge
    connected to the lower path of $f$.}
  \label{figure:lowerpath}
\end{figure}
 
\begin{lemma}\label{lemma:interval}
  At a step $k$ of the algorithm, let $v_1$ and $v_2$ be two active
  vertices on $\mathcal{C}_k$ such that all vertices of $\rbrack v_1,v_2
  \lbrack$ are passive. Then the path $\lbrack v_1,v_2 \rbrack$ on
  $\mathcal{C}_k$ is partitioned into
  \begin{longitem}
  \item a (possibly empty) path $\lbrack v_1,v \rbrack$ whose edges
    are bi-oriented in the finally computed orientation $X_0$, the
    left half-edge having label 2 and the right half-edge label~1,
  \item an edge $e=\lbrack v,v' \rbrack$ either simply oriented with
    label 2 from $v$ to $v'$, or simply oriented with label 3 from
    $v'$ to $v$, or bi-oriented, with label 2 on the half-edge
    incident to $v$ and label~3 on the half-edge incident to~$v'$,
  \item a (possibly empty) path $\lbrack v', v_2 \rbrack$ such that,
    each edge of $\lbrack v', v_2 \rbrack$ is bi-oriented, with label
    1 on the left half-edge and label 3 on the right half-edge.
  \end{longitem}
\end{lemma}

\begin{proof}
  The proof is by induction on the length $L$ of $\lbrack v_1, v_2
  \rbrack$. Assume that $L=1$. Then $\lbrack v_1,v_2\rbrack$ is
  reduced to an edge. If $v_1$ is removed at an earlier step than
  $v_2$, then the edge $(v_1,v_2)$ is simply oriented with label 2
  from $v_1$ to $v_2$. If $v_2$ is removed at an earlier step than
  $v_1$, then the edge $(v_1,v_2)$ is simply oriented with label 3
  from $v_2$ to $v_1$. If $v_1$ and $v_2$ are removed at the same
  step, then $(v_1,v_2)$ is bi-oriented, with label 2 on $v_1$'s side
  and label 3 on $v_2$'s side, see Figure~\ref{figure:stepk}.

  Assume that $L>1$. Observe that the outer path $\lbrack
  v_1,v_2\rbrack$ remains unchanged as long as none of $v_1$ or $v_2$
  is removed. This remark follows from the fact that all vertices of
  $\rbrack v_1,v_2 \lbrack$ are passive, so that no vertex of $\lbrack
  v_1, v_2 \rbrack$ can be treated as long as none of $v_1$ or $v_2$
  is treated.

  Then, two cases can arise: if $v_1$ is removed before $v_2$, the
  right neighbour $v$ of $v_1$ becomes active and the edge $(v_1,v)$
  is bi-oriented, with label 2 on $v_1$'s side and label 1 on $v$'s
  side, see Figure~\ref{figure:stepk}. Similarly if $v_2$ is removed
  before $v_1$, the left neighour $v$ of $v_2$ becomes active and the
  edge $(v,v_2)$ is bi-oriented with label 3 on $v_2$'s side and label
  1 on $v$'s side.

  The result follows by induction on $L$, with a recursive call to the
  path $\lbrack v,v_2\rbrack$ in the first case and to the path
  $\lbrack v_1,v\rbrack$ in the second case.  
\hfill $\qed$
\end{proof}


\subsection{The algorithm outputs the minimal $\alpha_0$-orientation
of the derived map}

\subsubsection{Definitions and preliminary lemmas}
\paragraph{Maximal bilabelled paths}
Let $v$ be a vertex of $G$. For \mbox{$1\leq i\leq 3$}, the $i$-path
of $v$ is the unique path $P_v^{i}=(v_0,\ldots,v_m)$ of edges of $G$
starting at $v$ and such that each edge $(v_p,v_{p+1})$ is the
outgoing edge of $v_p$ with label $i$ (i.e., the edge of $G$
containing the outgoing half-edge of $v_p$ with label $i$).
Acyclicity properties of Schnyder woods ensure that $P_v^i$ ends at
the outer vertex $a_i$, see~\cite{Fe03}.  For $1\leq i\leq 3$ and
$1\leq j\leq 3$ with $i\neq j$, we define the \emph{maximal $i-j$
  path} starting at $v$ as follows. Let $l\leq m$ be the maximal index
such that the subpath $(v_0,\ldots,v_l)$ of $P_v^i$ only consists of
bi-oriented edges with labels $i-j$.  Then the maximal $i-j$ path
starting at $v$ is defined to be the path $(v_0,\ldots,v_l)$ and is
denoted by $P_v^{i-j}$.

At a step $k\geq 2$, let $v^{(k)}$ be the chosen vertex, i.e., the
rightmost eligible vertex on $\mathcal{C}_k$.  First, observe that
there exists an active vertex on the right of $v^{(k)}$. Indeed, the
rightmost vertex $a_2$ is active as soon as $k\geq 2$. In addition
$a_2$ is non eligible on $\mathcal{C}_k$ because it is blocked, so
that $a_2$ is different from $v^{(k)}$. Hence, $a_2$ is an active
vertex on the right of $v^{(k)}$.

We define the \emph{next active vertex on the right} of $v^{(k)}$ as
the unique vertex $v$ on the right of $v^{(k)}$ on $\mathcal{C}_k$
such that all vertices of $\rbrack v^{(k)}, v\lbrack$ are passive.

\begin{figure}
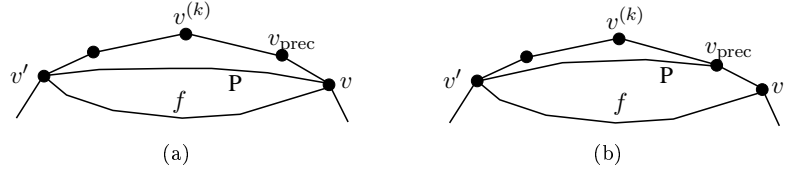

  \centering
  \subfigure[]{\scalebox{.85}{\input{Figures/sep-a.pstex_t}}}\qquad\qquad
  \subfigure[]{\scalebox{.85}{\input{Figures/sep-b.pstex_t}}}
  \caption{The two possible configurations related to the 
    next active vertex on the right of $v^{(k)}$.}
  \label{fig:sep}
\end{figure}

\begin{lemma} \label{lemma:P12}
  At a step $k\geq 2$, let $v^{(k)}$ be the chosen vertex. Let $v$ be
  the next active vertex on the right of $v^{(k)}$. Let
  $v_{\mathrm{prec}}$ be the left neighbour of $v$ on
  $\mathcal{C}_k$. Then, in the orientation $X_0$ finally computed,
  each edge of $\lbrack v^{(k)},v_{\mathrm{prec}} \rbrack$ is
  bi-oriented, with label 2 on its left side and label 1 on its right
  side. The edge $e=(v_{\mathrm{prec}},v)$ is either simply oriented
  with label 2 from $v_{\mathrm{prec}}$ to $v$ or bi-oriented, with
  label 2 on $v_{\mathrm{prec}}$'s side and label 3 on $v$'s side. In
  other words, $P_{v^{(k)}}^{2-1}=\lbrack
  v^{(k)},v_{\mathrm{prec}}\rbrack$ and the outgoing edge of
  $v_{\mathrm{prec}}$ with label 2 is $(v_{\mathrm{prec}},v)$.
\end{lemma}

\begin{proof}
  To prove this lemma, using the result of Lemma~\ref{lemma:interval},
  we just have to prove that $(v_{\mathrm{prec}},v)$ is neither
  bi-oriented with label 1 on $v_{\mathrm{prec}}$'s side and label 3
  on $v$'s side, nor simply oriented with label 3 from $v$ to
  $v_{\mathrm{prec}}$, see Figure~\ref{figure:lowerpath}.

  First, as the active vertex $v$ is on the right of
  $v^{(k)}$, it can not be eligible, so that $v$ is blocked.
As a consequence there exists a vertex $v'$ and a face $f$ such that
  $(v,v',f)$ is a separator. Lemma~\ref{lemma:existsvertex}
  ensures that there exists an eligible vertex in $\rbrack v',v
  \lbrack$. Hence the vertex $v'$ is on the left of $v^{(k)}$ on
  $\mathcal{C}_k$, otherwise $v^{(k)}$ would not be the rightmost
  eligible vertex. Let $P$ be the path on the boundary of $f$ 
going from $v$ to $v'$
  with $f$ on its left. Two cases can arise: 
\begin{longenum}
\item the first edge of $P$ is different from $(v,v_{\mathrm{prec}})$, so
  that $v_{\mathrm{prec}}$ is above $P$, see Figure~\ref{fig:sep}(a).
  Clearly, $v$ remains blocked as long as all vertices above $P$ have
  not been treated. Hence, $v_{\mathrm{prec}}$ will be treated at an
  earlier step that $v$.  As $v$ is active, it implies (see
  Figure~\ref{figure:stepk}) that $(v_{\mathrm{prec}},v)$ is simply
  oriented with label 2 from $v_{\mathrm{prec}}$ to $v$.
\item the first edge of $P$ is $(v,v_{\mathrm{prec}})$, see
  Figure~\ref{fig:sep}(b).  Observe that $v_{\mathrm{prec}}$ can not
  be equal to $v'$.  Indeed $v$ is on the right of $v^{(k)}$, so that
  $v_{\mathrm{prec}}$ is on the right or equal to $v^{(k)}$, whereas
  $v'$ is on the left of $v^{(k)}$. Hence, $P$ has length greater
  than~1. As a consequence, when $f$ will cease to be separating,
  $v_{\mathrm{prec}}$ will only be incident to
  $f$. Figure~\ref{figure:stepk} ensures that, when such a vertex is
  treated, the edge connecting this vertex to its right neighbour is
  always bi-oriented and bi-labelled 2-3, which concludes the proof. 
  \hfill $\qed$
\end{longenum}
\end{proof}

\begin{figure}
  \centering
  \scalebox{.85}{\input{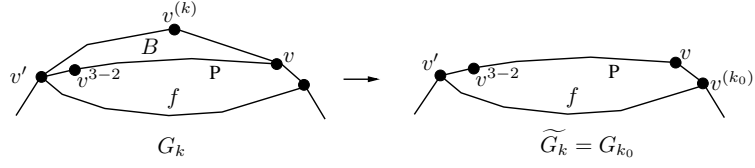}}
  \caption{The path between $v$ and $v^{3-2}$ will consist of
    bi-oriented edges bilabelled 3-2.}
  \label{fig:path}
\end{figure}

\begin{lemma} \label{lemma:P32}
  At a step $k\geq 2$, let $v^{(k)}$ be the rightmost eligible vertex and 
 $v$
  the next active vertex on the right of $v^{(k)}$. Let $v^{3-2}$ be
  the extremity of $P_{v}^{3-2}$ in $X_0$ and $e$ the outgoing
  edge of $v^{3-2}$ with label 3. If $e$ is bi-oriented, it is
bi-labelled 3-1 and we
  define
  $v_1=v^{3-2}$. Otherwise $e$ is simply oriented, we define $v_1$ as
  the extremity of $e$.

  Then $v_1$ belongs to $\mathcal{C}_k$ and is on the left of $v^{(k)}$.
\end{lemma}

\begin{proof}
  First, observe that each vertex $v''$ such that the pair $\{ v'',v\}$
  is separating is on the left of $v^{(k)}$, otherwise,
  Lemma~\ref{lemma:existsvertex} ensures that there exists an eligible
  vertex in $\rbrack v'',v\lbrack$, 
in contradiction with the fact that $v^{(k)}$ is
  the rightmost eligible vertex.

  Observe also that the set $\mathcal{S}$ 
of separators $(v'',v,f)$ involving $v$ and
  endowed with the inclusion-relation for the separated areas 
   is not only a partial order but
  a total order. In particular, for two separators $(v_1'',v,f_1)$ and
  $(v_2'',v,f_2)$, if $v_1''$ is on the left of $v_2''$, then the
  separated area of $(v_2'',v,f_1)$ is strictly included in the
  separated area of $(v_1'',v,f_2)$. In addition, $\mathcal{S}$ is non
  empty because $v$ is the next active vertex on the right of
  $v^{(k)}$, hence $v$ is blocked.

  Let $(v',v,f)$ be the maximal separator for the totally ordered
  set $\mathcal{S}$.  
Then the separated area of $(v',v,f)$ contains all separating
  faces incident to $v$ except $f$. Let $P$ be the path of edges on the 
boundary of
  $f$ going from $v$ to $v'$ with the interior of $f$ on its left, and let
  $B$ be the separated area of $(v',v,f)$. Let
  $\widetilde{G_k}$ be the submap of $G$ obtained by removing $B$ from
  $G_k$, and let $\widetilde{\mathcal{C}_k}$ be the boundary of
  $\widetilde{G_k}$.

  We claim that $f$ is not separating in $\widetilde{G_k}$. Otherwise,
  there would exist a vertex $v_2$ on the right of $v$ such that
  $(v,v_2,f)$ is a separator or there would exist a vertex $v_3$ on
  the left of $v'$ such that $(v_3,v',f)$ is a separator: the first
  case is in contradiction with the fact that all separators $\{
  v,v_2\}$ involving $v$ are such that $v$ is on the right of
  $v_2$. The second case is in contradiction with the fact that
  $(v',v,f)$ is the maximal separator involving $v$.

  We claim that only vertices of $B$ will be removed from step $k$ on,
  until all vertices of $B$ are removed. Indeed, all separating faces
  incident to vertices on the right of $v$ are faces of
  $\widetilde{G_k}$, hence they will remain separating as long as not
  all vertices of $B$ are removed. As all vertices on the right of $v$
  are either blocked or passive, it is easy to see inductively that
  all these vertices will keep the same status until all vertices of
  $B$ are removed.

  Let $k_0$ be the first step where all vertices of $B$ have been
  removed. Then $G_{k_0}=\widetilde{G_k}$. Hence $f$ is not separating
  anymore on $\mathcal{C}_{k_0}$, but all other faces of
  $\widetilde{G_k}$ that are separating at step $k$ are still
  separating at step $k_0$. We have seen that the separating faces incident to
  $v$ at step $k$ are the face $f$ and faces in $B$. In addition, all
  faces of $G_{k_0}$, except $f$, have kept their separating-status
  between step $k$ and step $k_0$. Hence $v$ is eligible on
  $\mathcal{C}_{k_0}$, and the rightmost eligible vertex $v^{(k_0)}$
  at step $k_0$ is a vertex incident to $f$. It is either $v$ or a
  vertex of $f$ on the right of $v$ (on $\mathcal{C}_{k_0}$)
  such that $\lbrack v, v^{(k_0)}\rbrack$ only consists of edges
  incident to $f$ (otherwise $f$ would be separating), 
see Figure~\ref{fig:path},
where $v^{(k_0)}$ is the right neighbour of $v$.

  Moreover, the left-connection vertex of $v^{(k_0)}$ is
  $v'$. Otherwise there would be a vertex of $f$ on
  $\widetilde{\mathcal{C}_k}$ and on the left of $v'$. 
  This vertex would also be on $\mathcal{C}_k$ (because only
  vertices of $B$ are removed to obtain $\widetilde{G_k}$ from
  $G_k$), in contradiction with the fact that $(v',v,f)$ is the
  maximal separator of $\mathcal{C}_k$ involving $v$.

  Then two cases can arise whether $v'$ is passive or active
  on $\mathcal{C}_{k_0}$: 
\begin{longenum}
\item $v'$ is passive on $\mathcal{C}_{k_0}$. Then $v'$ is not
  incident to any edge of $G\backslash G_{k_0}$. In particular $v'$ is
  not incident to any edge of $B\backslash G_{k_0}$.  Hence the right
  neighbour of $v'$ on $\mathcal{C}_{k_0}$ and on $\mathcal{C}_k$ are
  the same vertex, that is, the vertex $v_1$ preceding $v'$ on $P$.
  Observe that $v_1$ is on the left of $v^{(k)}$ on $\mathcal{C}_k$,
  indeed, $v_1$ can not be equal to $v^{(k)}$ at step $k$ because
  $v_1$ is incident to $f$, which is separating at this step. By
  definition of $v_1$ and by construction of the orientation (see
  Figure~\ref{figure:stepk}), $P_{v^{(k_0)}}^{3-2}$ is equal to
  $[v_1,v^{(k_0)}]$ taken from right to left, and $(v_1,v')$ is
  bi-oriented bi-labelled $3-1$ from $v_1$ to $v'$. As $v\in \lbrack
  v_1, v^{(k_0)}\rbrack $ at step $k_0$, $[v,v^{(k_0)}]\subseteq
  [v_1,v^{(k_0)}]$, so that $P_{v}^{3-2}$ is equal to $[v,v^{(k_0)}]$
  taken from right to left. As $(v_1,v')$ is bi-oriented bi-labelled
  $3-1$ from $v_1$ to $v'$, this concludes the proof for the first
  case (i.e., $v_1=v^{3-2}$).
\item  $v'$ is active on $\mathcal{C}_{k_0}$. In this case,
  upon taking $v_1$ to be the vertex $v'$, a similar argument as for
  the previous paragraph applies: indeed $v_1$ is a vertex on
  $\mathcal{C}_k$ on the left of $v^{(k)}$, and $P_{v}^{3-2}$ is the
    path on $\mathcal{C}_{k_0}$ going from $v$ to the right
  neighbour of $v_1$ on $\mathcal{C}_{k_0}$, and the edge connecting
  the right neighbour of $v_1$ to $v_1$ is simply oriented with label
  3 toward $v_1$ (see Figure~\ref{figure:stepk}).
\hfill $\qed$
\end{longenum}
\end{proof}

\begin{lemma}\label{lemma:bordercycle}
  The vertices $a_1$, $a_2$ and $a_3$ can not belong to any clockwise circuit.
\end{lemma}

\begin{proof}
  Let us consider $a_1$ (the cases of $a_2$ and $a_3$ can be dealt
  with identically). The outgoing edge of $a_1$ with label 1 is
  directed toward the outer face. The outgoing edges of $a_1$ with
  labels 2 and 3 connect respectively $a_1$ to two edge-vertices whose
  unique outgoing edge is directed toward the outer face. Hence
  each directed path starting at $a_1$ finishes immediately in the
  outer face.
\hfill $\qed$
\end{proof}

\subsubsection{Possible configurations for a minimal clockwise circuit of $X_0$}

\begin{lemma}
  \label{lemma:contourfaces}
  Let $f$ be an inner face of $G'$. Then the boundary of $f$ is not a
  clockwise circuit in $X_0$.
\end{lemma}

\begin{proof}
  Assume that the contour of $f$ is a clockwise circuit. We recall that
  the contour of $f$ has two edge-vertices, one dual vertex, and one
  primal vertex $v$. Let $i$ be the label of the edge $e'$ of $f$
  going out of $v$. The edge $e'$ is the first half-edge of an edge
  $e$ of $G$. We denote by $v_e$ the edge-vertex of $G'$ associated to
  $e$ and by $v'$ the vertex of $G$ such that $e=(v,v')$. As the
  contour of $f$ is a clockwise circuit, the unique outgoing edge of
  $v_e$ follows the edge $(v_e,v)$ in ccw order
  around $v_e$. Hence, according to
  Figure~\ref{figure:invariantslabelling}(c), the edge $e$ is
  bi-oriented and the second half-edge of $e$ has label $i+1$. We
  denote by $e_{\mathrm{next}}$ the edge of $G$ following $e$ in clockwise
  order around $v$. The edge $e_{\mathrm{next}}'$ of $G'$ following $e'$ in
  clockwise order around $v$ is the edge of $f$ directed toward
  $v$. Hence, the rules of labelling 
(Figure~\ref{figure:invariantslabelling}(a)) ensure that $e_{\mathrm{next}}'$ has label
  $i-1$. As $e_{\mathrm{next}}'$ is the second half-edge of $e_{\mathrm{next}}$, this
  ensures that $e_{\mathrm{next}}$ is simply oriented with label $i-1$ toward~$v$.

\begin{figure}
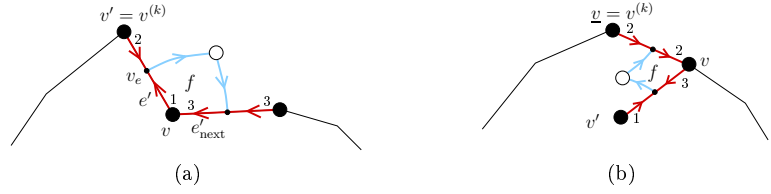

  \centering
  \subfigure[]{\scalebox{.6}{\input{Figures/contourface-a.pstex_t}}}
  \qquad\qquad
  \subfigure[]{\scalebox{.6}{\input{Figures/contourface-b.pstex_t}}}
  \caption{Configuration of a face $f$ of $G'$ whose boundary is a
    clockwise circuit and such that the outgoing edge of the
    unique primal vertex of $f$ has label 1 (Fig. a) and label 3
    (Fig. b).}
  \label{figure:contourface}
\end{figure}

  We now deal separately with the three possible cases $i=1,2,3$:
  
\begin{longitem}
\item Case $i=1$: The edge $e$ is bi-labelled 1-2 from $v$ to $v'$ and
  $e_{\mathrm{next}}$ is simply oriented with label 3 toward $v$, see
  Figure~\ref{figure:contourface}(a). Let $k$ be the step of the
  algorithm during which the vertex $v'$ is treated.
  Figure~\ref{figure:stepk} ensures that, if $v'$ is not equal to the
  rightmost eligible vertex $v^{(k)}$, then the outgoing edge with
  label 2 of $v'$ is bi-oriented with label 3 on the other half-edge,
  which is not the case here. Hence $v'=v^{(k)}$.

  In addition, as $(v',v)$ is bi-labelled 2-1 from $v'$ to $v$, the
  vertex $v$ is passive on $\mathcal{C}_k$. Hence, writing
  $e_{v\rightarrow}$ for the edge of $\mathcal{C}_k$ whose left
  extremity is $v$, there is no edge of $G\backslash G_k$ between $e$
  and $e_{v\rightarrow}$ in clockwise order around $v$, so that
  $e_{v\rightarrow}=e_{\mathrm{next}}$.

  We claim that $k\geq 2$.  Otherwise $v'$ would be equal to $a_1$. As
  $e=(v,v')$ is bi-labelled 1-2 from $v$ to $v'$, $v$ would be equal
  to $a_2$. But according to Lemma~\ref{lemma:bordercycle}, $a_2$ can
  not belong to any clockwise circuit.

  Hence $k\geq 2$ and we can use Lemma~\ref{lemma:P12}. In particular,
  this lemma ensures that $e_{v\rightarrow}$ is the outgoing
  edge of $v$ with label 2. We obtain here a contradiction with the
  fact that $e_{\mathrm{next}}$ is going toward $v$ with label 3 and
  $e_{v\rightarrow}=e_{\mathrm{next}}$.

\item Case $i=2$: The edge $e$ is bi-labelled 2-3 from $v$ to $v'$ and
  $e_{\mathrm{next}}$ is simply oriented with label 1 toward~$v$. Let
  $k$ be the step during which $v$ is treated. By construction of the
  orientation (see Figure~\ref{figure:stepk}), at step $k$ the vertex
  $v$ belongs to $\rbrack \mathrm{left}(v^{(k)}),v^{(k)}\lbrack$ and
  $e_{\mathrm{next}}$ is the outgoing edge of $v$ with label 3. This
  is in contradiction with the fact that $e_{\mathrm{next}}$ is simply
  oriented toward $v$ with label~1.

\item Case $i=3$: The edge $e$ is bi-labelled 3-1 from $v$ to $v'$ and
  $e_{\mathrm{next}}$ is simply oriented with label 2 toward $v$, see
  Figure~\ref{figure:contourface}(b). Let $\underline{v}$ be the
  origin of $e_{\mathrm{next}}$ and let $k$ be the step during which
  $\underline{v}$ is removed from $G_k$. As $e_{\mathrm{next}}$ is
  simply oriented with label 2 from $\underline{v}$ to $v$, we have
  $\underline{v}=v^{(k)}$ and $v=\mathrm{right}(v^{(k)})$.
  Lemma~\ref{lemma:P12} ensures that $v$ is the next active vertex on
  the right of $v^{(k)}$ on $\mathcal{C}_k$. In addition, $k\geq 2$,
  otherwise $v^{(k)}=a_1$, in contradiction with the fact that the
  outgoing edge of $a_1$ with label 2 is bi-oriented. Hence, we can
  use Lemma~\ref{lemma:P32}: here, the next active vertex on the right
  of $v^{(k)}$ is $v$ and the path $P_v^{3-2}$ is empty because the
  outgoing edge with label 3 of $v$ is bi-labelled 3-1. Hence the
  vertex denoted by $v_1$ in the statement of Lemma~\ref{lemma:P32} is
  here $v$.  Lemma~\ref{lemma:P32} ensures that $v$ is a vertex of
  $\mathcal{C}_k$ on the left of $v^{(k)}$, in contradiction with the
  fact that $v$ is the right neighbour of $v^{(k)}$ on
  $\mathcal{C}_k$.  
\hfill $\qed$
\end{longitem}
\end{proof}

\begin{lemma}[\cite{Fe03}]
  The possible configurations of an essential circuit of $X_0$ are
  illustrated in Figure~\ref{figure:configuration}, where $n_e^{(3)}$
  (resp. $n_e^{(4)}$) denotes the numbers of edge-vertices on the
  circuit that have respectively 3 (resp. 4) incident edges on or
  inside the circuit.
\end{lemma}

\begin{proof}
  \citeN[Lem.17]{Fe03} shows that an essential circuit $\cC$ of an
  $\alpha_0$-orientation has no edge in its interior whose origin is
  on $\cC$.  In addition, if $\cC$ is not the boundary of a face, he
  shows that all edge-vertices have either one incident edge or two
  incident edges inside $\cC$, which implies that the length of $\cC$
  is 6, 8, 10, or~12.  The only possible configurations are those
  listed in Figure~\ref{figure:configuration}. As $X_0$ has no
  clockwise circuit of length 4 according to
  Lemma~\ref{lemma:contourfaces}, this concludes the proof.
  \hfill$\qed$
\end{proof}

\begin{figure}
  \def\etalon{.3\linewidth}
  \centering
  \subfigure[${}^{\displaystyle n_e^{(3)}=3}_{\displaystyle n_e^{(4)}=0}$]{%
   \includegraphics[height=\etalon]{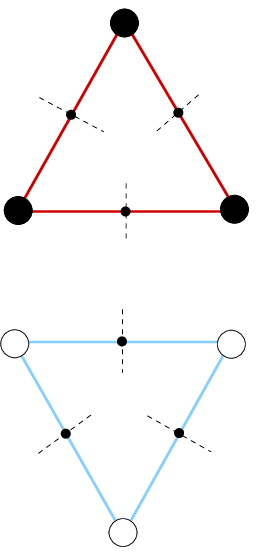}}%
  \qquad\quad
  \subfigure[${}^{\displaystyle n_e^{(3)}=2}_{\displaystyle n_e^{(4)}=2}$]{%
    \includegraphics[height=\etalon]{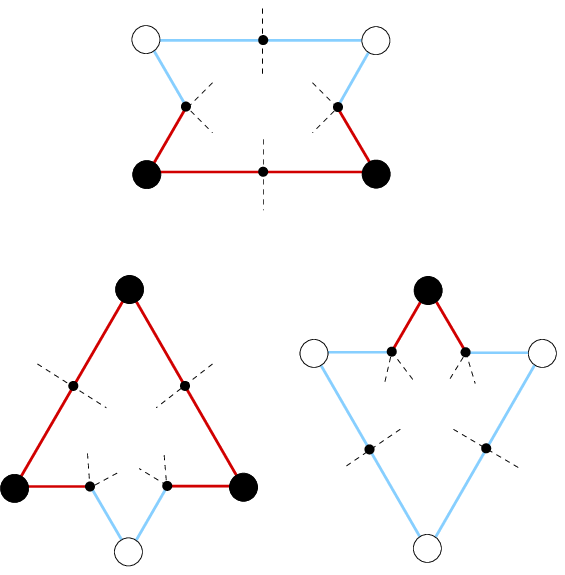}}\qquad\quad
  \subfigure[${}^{\displaystyle n_e^{(3)}=1}_{\displaystyle n_e^{(4)}=4}$]{%
    \includegraphics[height=\etalon]{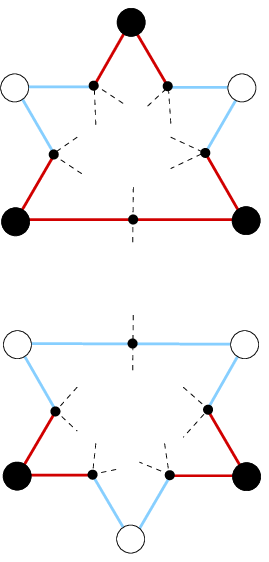}}\qquad\quad
  \subfigure[${}^{\displaystyle n_e^{(3)}=0}_{\displaystyle n_e^{(4)}=6}$]{%
    \includegraphics[height=\etalon]{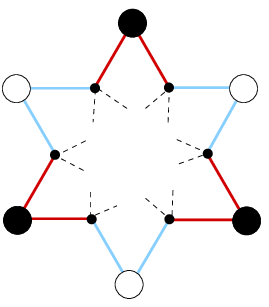}}
  \caption{The possible configurations for a minimal clockwise circuit of $X_0$.}
  \label{figure:configuration}
\end{figure}

\subsubsection{No configuration of Figure~\ref{figure:configuration} can be 
 a clockwise circuit in $X_0$}

We have restricted the number of possible configurations for a
clockwise circuit of $X_0$ to the list represented in
Figure~\ref{figure:configuration}. In this section, we describe a
method ensuring that the presence of a clockwise circuit for each
configuration of Figure~\ref{figure:configuration} yields a
contradiction.
The method relies on Lemma~\ref{lemma:P12}, Lemma~\ref{lemma:P32}, and
on the following lemma:

\begin{lemma}
  \label{lemma:P31} At a step $k$, let $v$ and $v'$ be two vertices on
  $\mathcal{C}_k$ such that $v$ is on the left of $v'$. Assume that
  there exists a path $P=(v_0,\ldots,v_l)$ of edges of $G$ such that
  $v_0=v$, $v_l=v'$, and for each $0\leq i\leq l-1$, the edge
  $(v_i,v_{i+1})$ is the outgoing edge of $v_i$ with label 1 in
  $X_0$. Then $P=\lbrack v,v'\rbrack$ on $\mathcal{C}_k$ and all edges
  of $P$ are bi-oriented bilabelled~1-3.
\end{lemma}

\begin{proof}
  Proving that $P=\lbrack v,v'\rbrack$ comes down to proving that all
  edges of $P$ are on $\mathcal{C}_k$.
  By construction of the orientation (see Figure~\ref{figure:stepk}),
  for each vertex $w$ of $G$, the extremity $w_1\in G$ of the outgoing
  edge of $w$ with label 1 is removed at an earlier step than $w$.
  Moreover, a vertex in $G\backslash G_k$ is removed at a step
  $j<k$. Hence, if $w$ is in $G\backslash G_k$, then $w_1$ is also in
  $G\backslash G_k$.  Hence, if $P$ passes by a vertex outside of
  $G_k$, it can not reach $\mathcal{C}_k$ again. By definition of an
  active vertex of $\mathcal{C}_k$, the extremity of its outgoing edge
  with label 1 is a vertex of $G\backslash G_k$. Hence none of the
  vertices $v_0,\ldots v_{l-1}$ can be active, otherwise $P$ would
  pass by a vertex outside of $G_k$ and could not reach
  $\mathcal{C}_k$ again.

  Hence, all vertices of $\mathcal{C}_k$ encountered by $P$ before
  reaching $v'$ are passive. It just remains to prove that the
  outgoing edge with label 1 of each passive vertex of $\mathcal{C}_k$
  is an edge of $\mathcal{C}_k$ and will be bi-oriented and bilabelled
  1-3 in~$X_0$.

  Let $w$ be a passive vertex of $\mathcal{C}_k$ and let $w_{l}$ and
  $w_{r}$ be respectively the left and the right neighbour of $w$ on
  $\mathcal{C}_k$. We claim that the outgoing edge of $w$ with label 1
  is the edge $(w,w_{l})$ if $w_{l}$ will be removed before $w_{r}$
  and is the edge $(w,w_{r})$ if $w_{r}$ will be removed before
  $w_{l}$. Indeed, as long as none of $w_{l}$ or $w_{r}$ is removed,
  $w$ remains passive and keeps $w_{l}$ and $w_{r}$ as left and right
  neighbour. Let $k_0$ be the first step where $w_l$ or $w_r$ is
  removed. By construction of the orientation, two vertices $v_1$ and
  $v_2$ on the boundary of $\mathcal {C}_{k_0}$ such that $\rbrack
  v_1,v_2\lbrack$ contains a passive vertex can not be removed at the
  same step. Hence, at step $k_0$, either $w_l$ or $w_r$ is
  removed. Assume that the removed vertex at step $k_0$ is
  $w_l$. Then, at step $k_0$, $(w, w_l)$ is given a bi-orientation and
  receives label 1 on $w$'s side and label 2 on $w_l$'s side, see
  Figure~\ref{figure:stepk}. Similarly, if the removed vertex is $w_r$
  then, at step $k_0$, $(w,w_r)$ is bi-orientated and receives label 1
  on $w$'s side and label 3 on $w_r$'s side.

  Finally, it is easy to see that only this second case can happen in
  the path $P$, because the starting vertex of $P$ is on the left of
  the end vertex of $P$ on~$\mathcal{C}_k$.
\hfill $\qed$\end{proof}

\begin{lemma}
  \label{lemma:none}
  None of the configurations of Figure~\ref{figure:configuration} can
  be the boundary of a clockwise circuit in $X_0$.
\end{lemma}

\begin{proof}
  We take here the example of the third configuration of the case $\{
  n_e^{(3)}=2,n_e^{(4)}=2\}$ of Figure~\ref{figure:configuration} and
  show why this configuration can not be a clockwise circuit in
  $X_0$. Let $\cC$ be a clockwise circuit corresponding to such a
  configuration. Then $\cC$ contains two successive dual edges $e_1^*$
  and $e_2^*$ ---in counter-clockwise order around $\cC$--- and a
  unique primal vertex which we denote by~$v_{\cC}$.
  Let $M'$ be the submap of $G'$ obtained by removing all edges and
  vertices outside of $\cC$. Let $M$ be the submap of $G$ obtained by
  keeping only the edges whose associated edge-vertex belongs to $M'$
  and by keeping the vertices incident to these edges.
  As $\cC$ is an essential circuit, no edge inside $\cC$ has its
  origin on $\cC$, see~\cite[Lem.17]{Fe03}.  The rules of labelling
  (see Figure~\ref{figure:invariantslabelling}), the fact that all
  edge-vertices have outdegree~1, and the fact that no edge goes from
  a vertex of $\cC$ toward the interior of $\cC$ determine
  unambiguously the labels and orientations of all the edges on the
  boundary of $M$ in $X_0$, up to the label of the outgoing edge of
  $v_{\cC}$ on $\cC$. Figures~\ref{figure:cases}(a),
  ~\ref{figure:cases}(b) and~\ref{figure:cases}(c) represent the
  respective configurations when the label of the outgoing edge of
  $v_{\cC}$ on $\cC$ is 1, 2 or~3.

  \begin{figure}
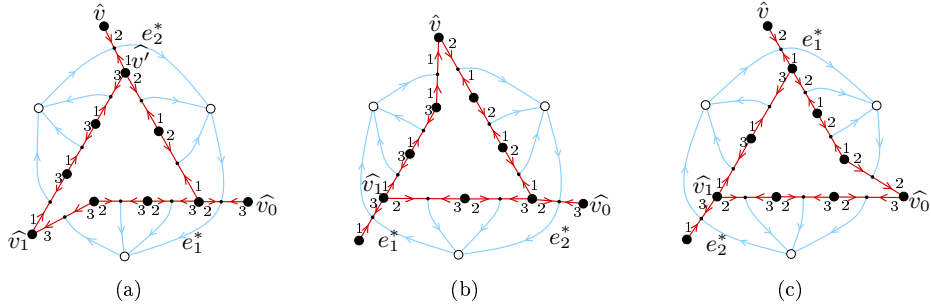

    \centering
    \subfigure[]{\input{Figures/cases2a.pstex_t}}\qquad\qquad
    \subfigure[]{\input{Figures/cases2b.pstex_t}}\qquad\qquad
    \subfigure[]{\input{Figures/cases2c.pstex_t}}
    \caption{The 3 possible cases for the boundary of the map $M$
      associated to the third configuration of the case $\{
        n_e^{(3)}=2,n_e^{(4)}=2\}$ in
        Figure~\ref{figure:configuration}.}
    \label{figure:cases}
  \end{figure} 
  
  First, we deal with the case of Figure~\ref{figure:cases}(a). Let
  $\hat{v}$ (resp. $\widehat{v_0}$) be the primal vertex outside of
  $\cC$ and adjacent to the edge-vertex associated to $e_2^*$
  (resp. $e_1^*$). Let $\widehat{v'}$ be the primal vertex inside of
  $\cC$ and adjacent to the edge-vertex associated to $e_2^*$. Let $k$
  be the step at which $\hat{v}$ is treated.
  As already explained in preceding proofs (for example in the proof
  of Lemma~\ref{lemma:contourfaces}), it is easy to see that $k\geq 2$
  and that $\hat{v}$ is the chosen vertex $v^{(k)}$.  Hence we can use
  Lemma~\ref{lemma:P12} and Lemma~\ref{lemma:P32}.
  Lemma~\ref{lemma:P12} and the configuration of
  Figure~\ref{figure:cases}(a) ensure that $\widehat{v'}$ is the right
  neighbour of $\hat{v}$ on $\mathcal{C}_k$ and that $\widehat{v_0}$
  is the next active vertex on the right of $\hat{v}$ on
  $\mathcal{C}_k$. Moreover, the configuration of
  Figure~\ref{figure:cases}(a) ensures that $\widehat{v_1}$
  corresponds to the vertex $v_1$ in the statement of
  Lemma~\ref{lemma:P32}. Hence Lemma~\ref{lemma:P32} ensures that
  $\widehat{v_1}$ is on $\mathcal{C}_k$ on the left of $\hat{v}$. We
  see on Figure~\ref{figure:cases}(a) that there is an oriented path
  $P$ going from $\widehat{v_1}$ to $\widehat{v}$ such that each edge
  of the path is leaving with label~1. Lemma~\ref{lemma:P31} ensures
  that all edges of $P$ are bilabelled 1-3, in contradiction with the
  fact that $(\widehat{v'},v)$ is bilabelled~1-2.

  We deal with the case of Figure~\ref{figure:cases}(b) similarly. We
  define $\hat{v}:=v_{\cC}$ and denote by $\widehat{v_0}$ the primal
  vertex outside of $\cC$ and adjacent to the edge-vertex associated
  to $e_2^*$. We denote by $\widehat{v_1}$ the primal vertex inside of
  $\cC$ and adjacent to the edge-vertex associated to~$e_1^*$. Let $k$
  be the step where $\hat{v}$ is removed. Then it is easy to see that
  $k\geq 2$ and $\hat{v}=v^{(k)}$. Hence we can use
  Lemma~\ref{lemma:P12} and Lemma~\ref{lemma:P32}.
  Lemma~\ref{lemma:P12} and the configuration of
  Figure~\ref{figure:cases}(b) ensure that $\widehat{v_0}$ is the next
  active vertex on the right of $\hat{v}$ on $\mathcal{C}_k$. We see
  on Figure~\ref{figure:cases}(b) that the vertex $\widehat{v_1}$
  corresponds to the vertex $v_1$ in the statement of
  Lemma~\ref{lemma:P32}. Hence, Lemma~\ref{lemma:P32} ensures that
  $\widehat{v_1}$ is on $\mathcal{C}_k$ on the left of $\hat{v}$. We
  see on Figure~\ref{figure:cases}(b) that there exists an oriented
  path $P$ going from $\widehat{v_1}$ to $\hat{v}$ such that each edge
  of $P$ leaves with label 1; but the last edge of $P$ is simply
  oriented, in contradiction with Lemma~\ref{lemma:P31}.

  The case of Figure~\ref{figure:cases}(c) can be treated
  similarly, as well as all configurations of
  Figure~\ref{figure:configuration}.  \hfill $\qed$
\end{proof}

Finally, Theorem~\ref{theorem:computesOrientation} follows from
Lemma~\ref{lemma:none} and from the fact that all possible
configurations for a clockwise circuit of $X_0$ are listed in
Figure~\ref{figure:configuration}.


\begin{acks}
  N. Bonichon and L. Castelli Aleardi are thanked for fruitful
  discussions.
\end{acks}

\bibliography{soda}

\begin{thebibliography}{}

\bibitem[\protect\citeauthoryear{Alliez and Gotsman}{Alliez and
  Gotsman}{2003}]{AlGo03}
{\sc Alliez, P.} {\sc and} {\sc Gotsman, C.} 2003.
\newblock Recent advances in compression of {3D} meshes.
\newblock In {\em Proc. of the Symp. on Multiresolution in Geometric Modeling}.
  Cambridge.
\newblock Available at \texttt{http://www.inria.fr/rrrt/rr-4966}.

\bibitem[\protect\citeauthoryear{Ambj{\o}rn, Bia{\l}as, Burda, Jurkiewicz, and
  Petersson}{Ambj{\o}rn et al\mbox{.}}{1994}]{ABBJP94}
{\sc Ambj{\o}rn, J.}, {\sc Bia{\l}as, P.}, {\sc Burda, Z.}, {\sc Jurkiewicz,
  J.}, {\sc and} {\sc Petersson, B.} 1994.
\newblock Sampling of random surfaces by baby universe surgery.
\newblock {\em Phys. Lett. B\/}~{\em 325}, 337--346.

\bibitem[\protect\citeauthoryear{di~Battista, Tamassia, and
  Vismara}{\mbox{di~Battista} et al\mbox{.}}{1999}]{DiTa}
{\sc di~Battista, G.}, {\sc Tamassia, R.}, {\sc and} {\sc Vismara, L.} 1999.
\newblock Output-sensitive reporting of disjoint paths.
\newblock {\em Algorithmica\/}~{\em 23,\/}~3, 302--340.

\bibitem[\protect\citeauthoryear{Bender}{Bender}{1987}]{B87}
{\sc Bender, E.~A.} 1987.
\newblock The number of three-dimensional convex polyhedra.
\newblock {\em Amer. Math. Monthly\/}~{\em 94,\/}~1, 7--21.

\bibitem[\protect\citeauthoryear{Bodirsky, Gr\"opl, and Kang}{Bodirsky et
  al\mbox{.}}{2003}]{BGK03}
{\sc Bodirsky, M.}, {\sc Gr\"opl, C.}, {\sc and} {\sc Kang, M.} 2003.
\newblock Generating labeled planar graphs uniformly at random.
\newblock In {\em Proceedings of ICALP'03, Eindhoven, The Netherlands}. LNCS.
  Springer Verlag, 1095--1107.

\bibitem[\protect\citeauthoryear{Bonichon, Felsner, and Mosbah}{Bonichon et
  al\mbox{.}}{2007}]{BFM04}
{\sc Bonichon, N.}, {\sc Felsner, S.}, {\sc and} {\sc Mosbah, M.} 2007.
\newblock Convex drawings of 3-connected planar graphs.
\newblock {\em Algorithmica\/}~{\em 47(4)}, 399--420.

\bibitem[\protect\citeauthoryear{Bonichon, Gavoille, and Hanusse}{Bonichon et
  al\mbox{.}}{2003}]{BGH03}
{\sc Bonichon, N.}, {\sc Gavoille, C.}, {\sc and} {\sc Hanusse, N.} 2003.
\newblock An information-theoretic upper bound of planar graphs using
  triangulations.
\newblock In {\em Proceedings of STACS'03, Berlin}. LNCS. Springer Verlag,
  499--510.

\bibitem[\protect\citeauthoryear{Bouttier, di~Francesco, and Guitter}{Bouttier
  et al\mbox{.}}{2002}]{BdFG02}
{\sc Bouttier, J.}, {\sc di~Francesco, P.}, {\sc and} {\sc Guitter, E.} 2002.
\newblock Census of planar maps: from the one-matrix solution to a
  combinatorial proof.
\newblock {\em Nucl. Phys. B\/}~{\em 645}, 477--499.

\bibitem[\protect\citeauthoryear{Brehm}{Brehm}{2000}]{Bre02}
{\sc Brehm, E.} 2000.
\newblock 3-orientations and {S}chnyder 3-tree-decompositions.
\newblock M.S.\ thesis, Freie Universit\"at Berlin.
\newblock \texttt{http://www.tu-berlin.de/}
  \texttt{\string~felsner/Diplomarbeiten/brehm.ps.gz}.

\bibitem[\protect\citeauthoryear{Brown}{Brown}{1964}]{Br}
{\sc Brown, W.} 1964.
\newblock Enumeration of triangulations of the disk.
\newblock {\em Proceedings of the London Mathematical Society\/}, 746--768.

\bibitem[\protect\citeauthoryear{Castelli-Aleardi and
  Devillers}{Castelli-Aleardi and Devillers}{2004}]{CD04}
{\sc Castelli-Aleardi, L.} {\sc and} {\sc Devillers, O.} 2004.
\newblock Canonical triangulation of a graph, with a coding application.
\newblock 24pp, available at \texttt{http://www.inria.fr/rrrt/rr-5231.html}.

\bibitem[\protect\citeauthoryear{Castelli-Aleardi, Devillers, and
  Schaeffer}{Castelli-Aleardi et al\mbox{.}}{2006}]{Ca}
{\sc Castelli-Aleardi, L.}, {\sc Devillers, O.}, {\sc and} {\sc Schaeffer, G.}
  2006.
\newblock Optimal succinct representations of planar maps.
\newblock In {\em Proceedings of SoCG'06, Sedona (Arizona)}. ACM Press,
  309--318.

\bibitem[\protect\citeauthoryear{Chuang, Garg, He, Kao, and Lu}{Chuang et
  al\mbox{.}}{1998}]{CGHKL98}
{\sc Chuang, R. C.-N.}, {\sc Garg, A.}, {\sc He, X.}, {\sc Kao, M.-Y.}, {\sc
  and} {\sc Lu, H.-I.} 1998.
\newblock Compact encodings of planar graphs via canonical orderings.
\newblock In {\em Proceedings of ICALP'98, Aalborg, Denmark}. LNCS. Springer
  Verlag, 118--129.

\bibitem[\protect\citeauthoryear{Denise, Vasconcellos, and Welsh}{Denise et
  al\mbox{.}}{1996}]{DVW96}
{\sc Denise, A.}, {\sc Vasconcellos, M.}, {\sc and} {\sc Welsh, D. J.~A.} 1996.
\newblock The random planar graph.
\newblock {\em Congr. Numer.\/}~{\em 113}, 61--79.
\newblock Festschrift for C. St.\ J. A. Nash-Williams.

\bibitem[\protect\citeauthoryear{Duchon, Flajolet, Louchard, and
  Schaeffer}{Duchon et al\mbox{.}}{2004}]{DuFlLoSc}
{\sc Duchon, P.}, {\sc Flajolet, P.}, {\sc Louchard, G.}, {\sc and} {\sc
  Schaeffer, G.} 2004.
\newblock Boltzmann samplers for the random generation of combinatorial
  structures.
\newblock {\em Combinatorics, Probability and Computing\/}~{\em 13,\/}~4--5,
  577--625.
\newblock Special issue on Analysis of Algorithms.

\bibitem[\protect\citeauthoryear{Felsner}{Felsner}{2001}]{F01}
{\sc Felsner, S.} 2001.
\newblock Convex drawings of planar graphs and the order dimension of
  3-polytopes.
\newblock {\em Order\/}~{\em 18}, 19--37.

\bibitem[\protect\citeauthoryear{Felsner}{Felsner}{2004}]{Fe03}
{\sc Felsner, S.} 2004.
\newblock Lattice structures for planar graphs.
\newblock {\em Electron. J. Comb.\/}~{\em 11,\/}~1, Research paper R15, 24p.

\bibitem[\protect\citeauthoryear{Flajolet, Zimmermann, and Van~Cutsem}{Flajolet
  et al\mbox{.}}{1994}]{FZVC94}
{\sc Flajolet, P.}, {\sc Zimmermann, P.}, {\sc and} {\sc Van~Cutsem, B.} 1994.
\newblock A calculus for random generation of combinatorial structures.
\newblock {\em Theoret. Comput. Sci.\/}~{\em 132,\/}~2, 1--35.

\bibitem[\protect\citeauthoryear{de~Fraysseix, Ossona~de Mendez, and
  Rosenstiehl}{\mbox{de~Fraysseix} et al\mbox{.}}{}]{taxiplan}
{\sc de~Fraysseix, H.}, {\sc Ossona~de Mendez, P.}, {\sc and} {\sc Rosenstiehl,
  P.}
\newblock {\em Pigale, Automatic Graph Drawing}.
\newblock \textsc{cams}, EHESS, Paris.
\newblock \texttt{http://sourceforge.org/pigale/}.

\bibitem[\protect\citeauthoryear{Fusy}{Fusy}{2005}]{Fusy}
{\sc Fusy, {\'E}.} 2005.
\newblock Quadratic exact size and linear approximate size random generation of
  planar graphs.
\newblock {\em Discrete Mathematics and Theoretical Computer Science\/}~{\em
  AD}, 125--138.

\bibitem[\protect\citeauthoryear{Gessel}{Gessel}{1992}]{Ge}
{\sc Gessel, I.} 1992.
\newblock Super ballot numbers.
\newblock {\em J. Symbolic Comput.\/}~{\em 14,\/}~2/3, 179--194.

\bibitem[\protect\citeauthoryear{Gotsman}{Gotsman}{2003}]{Go03}
{\sc Gotsman, C.} 2003.
\newblock On the optimality of valence-based connectivity coding.
\newblock {\em Computer Graphics Forum\/}, 99--102.

\bibitem[\protect\citeauthoryear{He, Kao, and Lu}{He et
  al\mbox{.}}{1999}]{HKL99a}
{\sc He, X.}, {\sc Kao, M.-Y.}, {\sc and} {\sc Lu, H.-I.} 1999.
\newblock Linear-time succinct encodings of planar graphs via canonical
  orderings.
\newblock {\em SIAM J. on Disc. Math.\/}~{\em 12,\/}~3, 317--325.

\bibitem[\protect\citeauthoryear{He, Kao, and Lu}{He et
  al\mbox{.}}{2000}]{HKL00}
{\sc He, X.}, {\sc Kao, M.-Y.}, {\sc and} {\sc Lu, H.-I.} 2000.
\newblock A fast general methodology for information-theoretically optimal
  encodings of graphs.
\newblock {\em SIAM J. Comput\/}~{\em 30,\/}~3, 838--846.

\bibitem[\protect\citeauthoryear{Kant}{Kant}{1996}]{Ka96}
{\sc Kant, G.} 1996.
\newblock Drawing planar graphs using the canonical ordering.
\newblock {\em Algorithmica\/}~{\em 16}, 4--32.
\newblock (also \emph{FOCS'92}).

\bibitem[\protect\citeauthoryear{Khodakovsky, Alliez, Desbrun, and
  Schr\"oder}{Khodakovsky et al\mbox{.}}{2002}]{KADS02}
{\sc Khodakovsky, A.}, {\sc Alliez, P.}, {\sc Desbrun, M.}, {\sc and} {\sc
  Schr\"oder, P.} 2002.
\newblock Near-optimal connectivity encoding of polygon meshes.
\newblock {\em Graphical Model\/}~{\em 64,\/}~3--4.

\bibitem[\protect\citeauthoryear{Khuller, Naor, and Klein}{Khuller et
  al\mbox{.}}{1993}]{Khu93}
{\sc Khuller, S.}, {\sc Naor, J.}, {\sc and} {\sc Klein, P.~N.} 1993.
\newblock The lattice structure of flow in planar graphs.
\newblock {\em SIAM J. Discrete Math.\/}~{\em 6(3)}, 477--490.

\bibitem[\protect\citeauthoryear{Lu}{Lu}{2002}]{L02}
{\sc Lu, H.-I.} 2002.
\newblock Linear-time compression of bounded-genus graphs into
  information-theoretically optimal number of bits.
\newblock In {\em Proceedings of SODA'02, San Francisco, California}. ACM
  Press, 223--224.

\bibitem[\protect\citeauthoryear{McDiarmid, Steger, and Welsh}{McDiarmid et
  al\mbox{.}}{2005}]{MdStWe03}
{\sc McDiarmid, C.}, {\sc Steger, A.}, {\sc and} {\sc Welsh, D.} 2005.
\newblock Random planar graphs.
\newblock {\em J. Combin. Theory, Series B\/}~{\em 93}, 187--205.

\bibitem[\protect\citeauthoryear{Mullin and Schellenberg}{Mullin and
  Schellenberg}{1968}]{Mu}
{\sc Mullin, R.} {\sc and} {\sc Schellenberg, P.} 1968.
\newblock The enumeration of c-nets via quadrangulations.
\newblock {\em J. Combin. Theory\/}~{\em 4}, 259--276.

\bibitem[\protect\citeauthoryear{Munro and Raman}{Munro and Raman}{1997}]{MR97}
{\sc Munro, J.~I.} {\sc and} {\sc Raman, V.} 1997.
\newblock Succinct representation of balanced parentheses, static trees and
  planar graphs.
\newblock In {\em Proceedings of FOCS'97, Miami, Florida}. ACM Press, 118--126.

\bibitem[\protect\citeauthoryear{Nijenhuis and Wilf}{Nijenhuis and
  Wilf}{1978}]{NiWi78}
{\sc Nijenhuis, A.} {\sc and} {\sc Wilf, H.~S.} 1978.
\newblock {\em Combinatorial Algorithms\/}, Second ed.
\newblock Academic Press.

\bibitem[\protect\citeauthoryear{Ossona~de Mendez}{Ossona~de
  Mendez}{1994}]{Oss}
{\sc Ossona~de Mendez, P.} 1994.
\newblock Orientations bipolaires.
\newblock Ph.D. thesis, Ecole des Hautes Etudes en Sciences Sociales, Paris.

\bibitem[\protect\citeauthoryear{Osthus, Pr{\"o}mel, and Taraz}{Osthus et
  al\mbox{.}}{2003}]{OPT03}
{\sc Osthus, D.}, {\sc Pr{\"o}mel, H.~J.}, {\sc and} {\sc Taraz, A.} 2003.
\newblock On random planar graphs, their number and their triangulations.
\newblock {\em J. Combin. Theory Ser. B\/}~{\em 88,\/}~1, 119--134.

\bibitem[\protect\citeauthoryear{Poulalhon and Schaeffer}{Poulalhon and
  Schaeffer}{2006}]{PS03b}
{\sc Poulalhon, D.} {\sc and} {\sc Schaeffer, G.} 2006.
\newblock Optimal coding and sampling of triangulations.
\newblock {\em Algorithmica\/}~{\em 46(3-4)}, 505--527.

\bibitem[\protect\citeauthoryear{Rossignac}{Rossignac}{1999}]{R99}
{\sc Rossignac, J.} 1999.
\newblock Edgebreaker: Connectivity compression for triangle meshes.
\newblock {\em IEEE Transactions on Visualization and Computer Graphics\/}~{\em
  5,\/}~1, 47--61.

\bibitem[\protect\citeauthoryear{Schaeffer}{Schaeffer}{1997}]{Sc97}
{\sc Schaeffer, G.} 1997.
\newblock Bijective census and random generation of {E}ulerian planar maps with
  prescribed vertex degrees.
\newblock {\em Electron. J. Combin.\/}~{\em 4,\/}~1, \# 20, 14 pp.

\bibitem[\protect\citeauthoryear{Schaeffer}{Schaeffer}{1999}]{Sc99}
{\sc Schaeffer, G.} 1999.
\newblock Random sampling of large planar maps and convex polyhedra.
\newblock In {\em Proceedings of STOC'99, Atlanta}. ACM Press, 760--769.

\bibitem[\protect\citeauthoryear{Schnyder}{Schnyder}{1990}]{S90}
{\sc Schnyder, W.} 1990.
\newblock Embedding planar graphs on the grid.
\newblock In {\em Proceedings of SODA'90, San Francisco, California}. ACM
  Press, 138--148.

\bibitem[\protect\citeauthoryear{Touma and Gotsman}{Touma and
  Gotsman}{1998}]{TG98}
{\sc Touma, C.} {\sc and} {\sc Gotsman, C.} 1998.
\newblock Triangle mesh compression.
\newblock In {\em Graphic Interface Conf.} 26--34.

\bibitem[\protect\citeauthoryear{Tutte}{Tutte}{1962}]{T62}
{\sc Tutte, W.~T.} 1962.
\newblock A census of planar triangulations.
\newblock {\em Canad. J. Math.\/}~{\em 14}, 21--38.

\bibitem[\protect\citeauthoryear{Tutte}{Tutte}{1963}]{Tu63}
{\sc Tutte, W.~T.} 1963.
\newblock A census of planar maps.
\newblock {\em Canad. J. Math.\/}~{\em 15}, 249--271.

\bibitem[\protect\citeauthoryear{Whitney}{Whitney}{1933}]{Whitney33}
{\sc Whitney, H.} 1933.
\newblock 2-isomorphic graphs.
\newblock {\em Amer. J. Math.\/}~{\em 54}, 245--254.

\bibitem[\protect\citeauthoryear{Wilson}{Wilson}{1997}]{Wi97}
{\sc Wilson, D.~B.} 1997.
\newblock Determinant algorithms for random planar structures.
\newblock In {\em Proceedings of SODA'97, New Orleans, Louisiana}. ACM Press,
  258--267.

\bibitem[\protect\citeauthoryear{Wilson}{Wilson}{2004}]{W98}
{\sc Wilson, D.~B.} 2004.
\newblock An annotated bibliography of perfectly random sampling with markov
  chains.
\newblock Maintained on
  \texttt{http://dimacs.rutgers.edu/\string~dbwilson/exact}.

\end{thebibliography}
\end{document}